\documentclass[11pt,a4paper]{article}

\usepackage[left=2.3cm, top=2.3cm,bottom=2.3cm,right=2.3cm]{geometry}

\usepackage{mathtools,amsmath,amssymb,amsthm,mathrsfs,calc,graphicx,stmaryrd,xcolor,dsfont,tikz,pgfplots,bbm,float,array,epsfig,hyperref,color}
\usetikzlibrary{calc}
\usepackage[numbers,square]{natbib}
\usepackage[british]{babel}
\usepackage{amsfonts}              
\usepackage{enumitem}

\usepackage{manfnt} 
\usepackage{tikz} 

\usepackage{tikz}
\usetikzlibrary{arrows.meta}

\numberwithin{equation}{section}
\numberwithin{figure}{section}

\allowdisplaybreaks[4]

\newtheorem {theorem}{Theorem}[section]
\newtheorem {proposition}[theorem]{Proposition}
\newtheorem {lemma}[theorem]{Lemma}
\newtheorem {corollary}[theorem]{Corollary}

{\theoremstyle{definition}
\newtheorem{definition}[theorem]{Definition}

\newtheorem*{convention*}{Convention}

\newtheorem {example}[theorem]{Example}
\newtheorem {remark}[theorem]{Remark}
}

{
\theoremstyle{remark}
}

\newcommand{\Vol}{\operatorname{Vol}}

\renewcommand{\Im}{\operatorname{Im}}  

\newcommand{\ii}{{\rm{i}}}

\def\ba{\begin{array}}
\def\ea{\end{array}}
\def\bea{\begin{eqnarray} \label}
\def\eea{\end{eqnarray}}
\def\be{\begin{equation} \label}
\def\ee{\end{equation}}
\def\bit{\begin{itemize}}
\def\eit{\end{itemize}}
\def\ben{\begin{enumerate}}
\def\een{\end{enumerate}}

\def\lan{\langle}
\def\ran{\rangle}

\def\CC{\mathbb{C}}
\def\E{\mathbb{E}}

\def\N{\mathbb{N}}
\def\P{\mathbb{P}}

\def\R{\mathbb{R}}
\def\RRd1{\mathbb{R}^{d+1}}

\def\dint{\textup{d}}

\newcommand{\eee}{{\rm e}}

\newcommand{\ind}{\mathbbm{1}}
\newcommand{\eps}{\varepsilon}
\newcommand{\pos}{\mathop{\mathrm{pos}}\nolimits}

\newcommand{\lin}{\mathop{\mathrm{lin}}\nolimits}
\newcommand{\conv}{\mathop{\mathrm{conv}}\nolimits}

\newcommand{\dd}{{\rm d}}

\DeclareMathOperator{\relint}{relint}

\def\sgn{\mathrm{sgn}}

\def\factor{2/3*0.1}
\newcommand{\octahedron}{
	\begin{tikzpicture}[thick,scale=5]
	\coordinate (A1) at (0,0);
	\coordinate (A2) at (0.6*\factor,0.2*\factor);
	\coordinate (A3) at (1*\factor,0);
	\coordinate (A4) at (0.4*\factor,-0.2*\factor);
	\coordinate (B1) at (0.5*\factor,0.5*\factor);
	\coordinate (B2) at (0.5*\factor,-0.5*\factor);

	\draw[solid][line width=0.35pt] (A1) -- (A4) -- (B1);
	\draw[solid][line width=0.35pt] (A1) -- (A4) -- (B2);
	\draw[solid][line width=0.35pt] (A3) -- (A4) -- (B1);
	\draw[solid][line width=0.35pt] (A3) -- (A4) -- (B2);
	\draw[solid][line width=0.35pt] (B1) -- (A1) -- (B2) -- (A3) --cycle;
	\end{tikzpicture}
}

\newcommand{\Ccrosspoly}{C^{\octahedron}}
\newcommand{\Ccube}{C^{\mbox{\mancube}}}

\begin{document}

\title{\bfseries Angles and volumes of simplices, boxes, and crosspolytopes in constant curvature}

\author{Zakhar Kabluchko and Philipp Schange}

\date{}

\maketitle

\begin{abstract}
We determine, up to isometry, the Riemannian tangent and normal cones at arbitrary faces of boxes, acute orthocentric simplices, rectangular simplices, and asymmetric crosspolytopes in spaces of constant sectional curvature \(\kappa\in\mathbb R\). Using these descriptions, we derive explicit one-dimensional integral formulas for their internal and external angles in arbitrary dimension. In particular, our results cover regular simplices, cubes, and crosspolytopes, the three families of regular polytopes existing in every dimension. We also establish a correspondence between regular simplices and cubes with isometric local cones at corresponding faces, provided their edge-length parameters are suitably matched. Combining our angle formulas with the Poincaré relation, we obtain volume formulas for simplices and crosspolytopes in even dimensions. For boxes, we derive volume formulas valid in every dimension. Among the new results are angle formulas for spherical and hyperbolic cubes and crosspolytopes, volume formulas for asymmetric crosspolytopes, and a one-dimensional integral representation for the internal angles of regular crosspolytopes that appears to be new even in the Euclidean case. The formulas involve the standard normal distribution function \(\Phi(x)\) and its values at purely imaginary arguments \(\Phi(\ii x)\). For example, if \(d\geq 2\) is even, then the volume of the ideal regular simplex in the \(d\)-dimensional hyperbolic space of curvature \(\kappa=-1\) is given by the absolutely convergent integral
\[
\frac{\pi^{d/2}}
{\sqrt{2}\,\mathrm{i}^{d}\,
\Gamma\left(\frac{d+1}{2}\right)}
\int_{-\infty}^{\infty}
\left[
\Phi\left(\frac{\mathrm{i} y}{\sqrt d}\right)^{d+1}
+
\Phi\left(-\frac{\mathrm{i} y}{\sqrt d}\right)^{d+1}
\right]
\mathrm{e}^{-y^2/2}\,\mathrm{d}y.
\]
\end{abstract}


\tableofcontents

\section{Introduction and main results}
\subsection{Motivation and overview}
Angles at faces are among the basic local invariants of a convex polytope. They arise in sphere-packing bounds \cite{rogers_packing,coxeter_few_rogers_covering_equal,boeroeczky_packing_spaces_const_curv,kellerhals_ball_packing,rogers_book_packing_covering}, in the study of neighborliness of random polytopes~\cite{vershik_sporyshev,donoho_neighborliness_proportional,donoho_tanner2}, in the study of random projections and sections of polytopes~\cite{affentranger_schneider_1992,boroczky_henk,donoho_tanner2,donoho_tanner1,KabluchkoSeidel2022}, and in convex optimization with random data \cite{amelunxen_lotz_mccoy_tropp_living_on_the_edge,donoho_tanner_sparse_nonnegative_sol,donoho_tanner_observed_universality}. In spherical and hyperbolic geometry they have an additional global significance: in even dimension, the Poincar\'e relation expresses the Riemannian volume of a polytope in terms of an alternating sum of its internal angle sums \cite[Section~4.2]{camenga_thesis}. Basic properties of angles of convex polytopes, including the
Gram--Euler relation, are reviewed in
\cite[Section~14]{GruenbaumBook},
\cite[Sections~18E--18F]{McMullenConvexPolytopesPolyhedra},
and \cite[Section~15.2.3]{HenkRichterGebertZiegler2017}.

In this paper, we study internal and external angles of polytopes in \(d\)-dimensional spaces of constant sectional curvature \(\kappa\in\mathbb R\). Our main contributions are as follows.

\begin{itemize}
\item \emph{Angles and local cones for general classes of polytopes.}
We determine, up to isometry, the Riemannian tangent and normal cones at arbitrary faces of  boxes, acute orthocentric simplices, rectangular simplices, and asymmetric crosspolytopes. From these descriptions, we derive explicit formulas for the corresponding internal and external angles, valid in arbitrary dimension and across spherical, Euclidean, and hyperbolic geometry.
\item \emph{Regular polytopes.}
As special cases, we obtain explicit one-dimensional integral representations for the angles at faces of arbitrary dimension of regular simplices, cubes, and crosspolytopes, the three families of regular polytopes that exist in every dimension.
\item \emph{A simplex--cube correspondence.}
We show that, after a natural matching of their edge-length parameters, regular simplices and cubes have isometric tangent and normal cones at faces of equal dimension. In particular, the angles of the ideal hyperbolic cube and the ideal hyperbolic simplex at the faces of the same dimension are equal.
\item \emph{Volume formulas.}
We derive explicit formulas for the Riemannian volumes of 
boxes in every dimension. In even dimensions, the Poincar\'e relation
converts our internal-angle formulas into volume formulas for acute
orthocentric simplices, rectangular simplices, and asymmetric
crosspolytopes, including regular simplices and regular crosspolytopes
as special cases.
\end{itemize}

\subsection{Angles of regular polytopes}
\label{subsec:intro_angles_regular_polytopes}

Fix a dimension $d\geq 2$ and a curvature $\kappa\in\R$.
We consider polytopes in the $d$-dimensional space of constant sectional curvature
$\kappa$, which is spherical for $\kappa>0$, Euclidean for $\kappa=0$,
and hyperbolic for $\kappa<0$. We use the Klein model (see Section~\ref{subsec:klein_model} for details) in which the polytopes remain ordinary Euclidean polytopes, whereas the curvature $\kappa$ changes the tangent space inner product with respect to which their local cones are measured.
We write
$\Delta_\ell^d(\kappa)$, $\Box_\ell^d(\kappa)$, and
$\Diamond_\ell^d(\kappa)$ for, respectively, a regular simplex, cube,
and crosspolytope of geodesic edge length $\ell$ in this space, whenever
the corresponding polytope exists. If the center of the polytope is
placed at the origin of the Klein model, its Klein-model realization is
an ordinary Euclidean regular polytope, although the geodesic edge length \(\ell\) is in general different from the corresponding Euclidean edge length. For $\kappa<0$, we also write $\Delta_\infty^d(\kappa)$, $\Box_\infty^d(\kappa)$, and
$\Diamond_\infty^d(\kappa)$ for the corresponding ideal regular
polytopes, whose vertices lie on the ideal boundary of the hyperbolic space.

For a $k$-dimensional face $F$ of a polytope $P$, we denote the internal
and external angles at $F$ by $\beta_\kappa(F,P)$ and
$\gamma_\kappa(F,P)$, respectively. These are the solid angles of the
corresponding tangent and normal cones. The solid angle of a cone is
the normalized spherical measure of its intersection with the unit
sphere in its linear span, where both the unit sphere and its spherical
measure are induced by the Riemannian inner product on the tangent
space; see Sections~\ref{subsec:tangent_normal_cones} and~\ref{subsec:solid_angles} for details.

We now state formulas for $\beta_\kappa(F,P)$ and $\gamma_\kappa(F,P)$
for the three families of regular polytopes introduced above. By
symmetry, these angles depend only on $d$, $k$, $\kappa$, and $\ell$,
and not on the particular choice of the $k$-dimensional face $F$.

All formulas below will be expressed in terms of the complex-analytic extension
of the standard normal distribution function,
\begin{equation}\label{eq:intro_entire_normal_distribution_function}
\Phi(z)
=
\frac12+\frac1{\sqrt{2\pi}}
\int_0^z \eee^{-t^2/2}\,\dd t,
\qquad z\in\CC,
\end{equation}
where the integral may be taken along any contour joining $0$ to $z$.

\begin{convention*}
Let $\ii=\sqrt{-1}$.
Throughout the paper, square roots of negative real numbers are
understood according to the convention
\[
\sqrt{-a}:=\ii\sqrt a,
\qquad a>0.
\]
\end{convention*}

To state the spherical, Euclidean, and hyperbolic formulas in a common
form, we use the following parameter:
\begin{equation}\label{eq:intro_curvature_cosine}
C := C_\kappa(\ell)
:=
\begin{cases}
\cosh\bigl(\sqrt{-\kappa}\,\ell\bigr),&\kappa<0,\\
1,&\kappa=0,\\
\cos\bigl(\sqrt{\kappa}\,\ell\bigr),&\kappa>0.
\end{cases}
\end{equation}

In the formulas below, let $\xi\sim\mathrm N(0,1)$ be a standard
normal random variable.

\begin{theorem}[Angles of regular simplices and cubes]
\label{theo:intro_angles_regular_simplices_and_cubes}
Let $k\in\{0,\ldots,d-1\}$.  We first consider regular polytopes with finite geodesic edge length \(\ell\); the ideal hyperbolic case will be treated separately below.
\begin{enumerate}
\item Let $F$ be a $k$-dimensional face of
$\Delta_\ell^d(\kappa)$.  Then
\begin{align}
\gamma_\kappa\bigl(F,\Delta_\ell^d(\kappa)\bigr)
&=
\E\left[
\Phi\left(
\sqrt{\frac{C}{1+kC}}\,\xi
\right)^{d-k}
\right],
\label{eq:intro_regular_simplex_external_angle}
\\
\beta_\kappa\bigl(F,\Delta_\ell^d(\kappa)\bigr)
&=
\E\left[
\Phi\left(
\sqrt{-\frac{C}{1+dC}}\,\xi
\right)^{d-k}
\right].
\label{eq:intro_regular_simplex_internal_angle}
\end{align}

\item Let $F$ be a $k$-dimensional face of
$\Box_\ell^d(\kappa)$. Then
\begin{align}
\gamma_\kappa\bigl(F,\Box_\ell^d(\kappa)\bigr)
&=
\E\left[
\Phi\left(
\sqrt{\frac{C-1}{2+k(C-1)}}\,\xi
\right)^{d-k}
\right],
\label{eq:intro_regular_cube_external_angle}
\\
\beta_\kappa\bigl(F,\Box_\ell^d(\kappa)\bigr)
&=
\E\left[
\Phi\left(
\sqrt{-\frac{C-1}{2+d(C-1)}}\,\xi
\right)^{d-k}
\right].
\label{eq:intro_regular_cube_internal_angle}
\end{align}
\end{enumerate}
In the hyperbolic case, letting $\ell\to\infty$, equivalently
$C\to\infty$, yields the corresponding angle formulas for ideal
regular simplices and cubes. We record this special case separately
below.
\end{theorem}

\begin{example}[Ideal hyperbolic simplices and cubes]
\label{exa:intro_ideal_simplex_cube}
Let $\kappa<0$ and let $k\in\{1,\ldots,d-1\}$. If $F_\Delta$ and
$F_\Box$ are $k$-dimensional faces of the ideal regular simplex
$\Delta_\infty^d(\kappa)$ and the ideal regular cube
$\Box_\infty^d(\kappa)$, respectively, then
\[
\gamma_\kappa(F_\Delta,\Delta_\infty^d(\kappa))
=
\gamma_\kappa(F_\Box,\Box_\infty^d(\kappa))
=
\E\left[
\Phi\left(\frac{\xi}{\sqrt{k}}\right)^{d-k}
\right],
\]
and
\[
\beta_\kappa(F_\Delta,\Delta_\infty^d(\kappa))
=
\beta_\kappa(F_\Box,\Box_\infty^d(\kappa))
=
\E\left[
\Phi\left(\frac{\ii\xi}{\sqrt d}\right)^{d-k}
\right].
\]
Thus, ideal regular simplices and cubes have the same internal and
external angles at faces of the same positive dimension. This is not
merely a coincidence of the formulas: the corresponding tangent and
normal cones are isometric;  see Section~\ref{sec:cubes_simplices_isometric_local_cones} for more details.
\end{example}

The local cones of regular crosspolytopes have a different structure
from those of regular simplices and cubes.

\begin{theorem}[Angles of regular crosspolytopes]
\label{theo:intro_angles_regular_crosspolytopes}
Let $k\in\{0,\ldots,d-1\}$ and suppose that $\ell<\infty$.
Let $F$ be a $k$-dimensional face of
$\Diamond_\ell^d(\kappa)$. Then
\begin{align}
\gamma_\kappa\bigl(F,\Diamond_\ell^d(\kappa)\bigr)
&=
\frac12\,
\E\left[
\left(
2\Phi\left(
\sqrt{\frac{C}{1+kC}}\,|\xi|
\right)-1
\right)^{d-k-1}
\right],
\label{eq:intro_regular_crosspolytope_external_angle}
\\
\beta_\kappa\bigl(F,\Diamond_\ell^d(\kappa)\bigr)
&=
\frac12
-
\frac{2^{d-k-1}}{\sqrt{2\pi}}\,
\E\left[
\frac1{\xi}\,
\Im\left[
\Phi\left(
\ii\sqrt{\frac{C}{1+(d-1)C}}\,\xi
\right)^{d-k-1}
\right]
\right].
\label{eq:intro_regular_crosspolytope_internal_angle}
\end{align}
In the second formula, the integrand at $\xi=0$ is understood by
continuous extension. In the hyperbolic case, the corresponding formulas for faces of
positive dimension of the ideal regular crosspolytope are obtained by
letting $C\to\infty$.
\end{theorem}

The preceding formulas simultaneously cover spherical, Euclidean, and
hyperbolic geometry. Although some arguments of $\Phi$ are purely
imaginary, all the quantities displayed above are real.

\begin{remark}[Related literature]
Formulas for external angles of regular Euclidean simplices go back to
Ruben and Hadwiger \cite{ruben,hadwiger}, \cite[Section~15.2.3]{HenkRichterGebertZiegler2017}. A formula for the internal angle at a vertex of a regular Euclidean
simplex was obtained by Rogers
\cite[Section~4]{rogers_packing},
\cite[pp.~85--90]{rogers_book_packing_covering}, using a method which
he attributes to H.~E.~Daniels, while internal angles at faces of
arbitrary dimension were computed by Vershik and Sporyshev
\cite[Lemma~4]{vershik_sporyshev}, who also derived a formula for the
volumes of regular spherical simplices; see also~\citet[Section~2]{rogers_schlaefli} for a similar formula.  A special case of the Vershik--Sporyshev
formula, corresponding to geodesic edge lengths greater than $\pi/2$,
can be found in \cite[Satz~3, p.~283]{boehm_hertel_book}.   These results can be expressed in terms of the
function
\[
g_n(r):=
\E\left[\Phi\bigl(\sqrt r\,\xi\bigr)^n\right],
\qquad n\in \N, \quad r\geq -1/n,\quad (n,r)\neq(1,-1),
\]
following the notation of
\cite{kabluchko_zaporozhets_absorption}. More general formulas for
angles of orthocentric simplices and orthocentric cones were derived
in~\cite{kabluchko_schange_angles_orthocentric_simplices}. Volumes of
regular spherical simplices are described by the Schl\"afli function,
whose properties are discussed by Coxeter
\cite{coxeter_upper_bound_nonoverlap_spheres}; see also
\cite{shoom_schlaefli} for a recent account of this literature.
In view of these results, the formulas
\eqref{eq:intro_regular_simplex_external_angle} and
\eqref{eq:intro_regular_simplex_internal_angle} for the external and
internal angles of hyperbolic regular simplices are not surprising;
to the best of our knowledge, however, they have not appeared
explicitly in the literature.

For Euclidean regular crosspolytopes, external angles were computed by
Betke and Henk \cite{betke_henk}; see also~\cite{HenkHernandezCifre2008a,KabluchkoZaporozhets2019} for generalizations. Probabilistic representations for
both internal and external angles of Euclidean regular crosspolytopes were obtained in
\cite{KabluchkoSeidel2022}; in particular, the internal angle is
represented there in terms of a probability involving a sum of
absolute values of independent Gaussian random variables.  Applications to random sections and projections of regular polytopes are given in~\cite{affentranger_schneider_1992,boroczky_henk,KabluchkoSeidel2022}. We are not
aware of previous formulas for the Riemannian internal and external
angles at arbitrary-dimensional faces of regular spherical or
hyperbolic cubes and crosspolytopes. Even in the Euclidean case, the
one-dimensional representation
\eqref{eq:intro_regular_crosspolytope_internal_angle} for the internal
angle of a regular crosspolytope seems not to have been recorded
previously; in particular, it is not contained in~\cite{KabluchkoSeidel2022}.
\end{remark}

\subsection{Riemannian volumes}
\label{subsec:intro_riemannian_volumes}
Our angle formulas also lead to explicit formulas for Riemannian
volumes. In even dimensions, the link is provided by the Poincar\'e
relation, which expresses the volume of a polytope of nonzero constant
curvature in terms of an alternating sum of its internal angle sums;
see Section~\ref{subsec:poincare_relation}. For  boxes,
we obtain a volume formula directly, valid in every dimension.

We record here the resulting formulas for the ideal regular simplex,
cube, and crosspolytope in hyperbolic space of curvature $-1$.

\begin{theorem}[Volumes of ideal regular polytopes]
\label{theo:intro_volumes_ideal_regular_polytopes}
The following assertions hold.

\begin{enumerate}
\item If $d\geq2$ is even, then the volume of the ideal regular
$d$-simplex is
\begin{equation}
\label{eq:intro_volume_ideal_simplex}
\Vol_{d,-1}\bigl(\Delta_\infty^d(-1)\bigr)
=
\frac{\omega_{d+1}}{2\ii^d\sqrt{2\pi}}
\int_{-\infty}^{\infty}
\left[
\Phi\left(\frac{\ii y}{\sqrt d}\right)^{d+1}
+
\Phi\left(-\frac{\ii y}{\sqrt d}\right)^{d+1}
\right]
\eee^{-y^2/2}\,\dd y .
\end{equation}

\item For every integer $d\geq2$, the volume of the ideal regular $d$-cube is
\begin{equation}
\label{eq:intro_volume_ideal_cube}
\Vol_{d,-1}\bigl(\Box_\infty^d(-1)\bigr)
=
\frac{\omega_{d+1}}{\ii^d\sqrt{2\pi}}
\int_0^\infty
\left(
2\Phi\left(\frac{\ii y}{\sqrt d}\right)-1
\right)^d
\eee^{-y^2/2}\,\dd y .
\end{equation}

\item If $d\geq2$ is even, then the volume of the ideal regular
$d$-crosspolytope is
\begin{equation}
\label{eq:intro_volume_ideal_crosspolytope}
\Vol_{d,-1}\bigl(\Diamond_\infty^d(-1)\bigr)
=
\frac{\omega_{d+1}}{\ii^d}
\left[
\frac12
-
\frac{2^d}{\pi}
\int_0^\infty
\frac1y
\Im\left[
\Phi\left(\frac{\ii y}{\sqrt{d-1}}\right)^d
\right]
\eee^{-y^2/2}\,\dd y
\right].
\end{equation}
\end{enumerate}
Here
$
\omega_{d+1}
=
\frac{2\pi^{(d+1)/2}}
{\Gamma((d+1)/2)}
$
is the $d$-dimensional surface area of the unit sphere
$\mathbb S^d\subseteq\R^{d+1}$. Although the right-hand sides involve complex quantities, they are real. In \eqref{eq:intro_volume_ideal_crosspolytope}, the integrand at
$y=0$ is understood by continuous extension.
\end{theorem}

More generally, the results of the paper provide volume formulas for
 boxes in arbitrary dimension and, in even dimensions, for
acute orthocentric simplices, rectangular simplices, and asymmetric
crosspolytopes.

\begin{remark}[Related literature]
Volumes of regular hyperbolic simplices have been studied from several
points of view. Haagerup and Munkholm
\cite{haagerup_munkholm_simplices_max_vol_hyperbolic} proved that the
ideal regular hyperbolic simplex has maximal volume among hyperbolic
simplices. Marshall~\cite{Marshall1999Volumes} obtained high-dimensional
asymptotics for the volumes of regular spherical and hyperbolic
simplices. Kellerhals used regular hyperbolic simplices in deriving
lower volume bounds for hyperbolic manifolds
\cite{kellerhals_reg_simpl_vol_bounds} and studied the associated
simplicial density function in the context of ball packings in spaces
of constant curvature~\cite{kellerhals_ball_packing}. An explicit
formula for the volume of a regular hyperbolic simplex in arbitrary
dimension was obtained in~\cite{KabluchkoSchange2025Hyperbolic}.

For regular hyperbolic cubes, explicit volume formulas were obtained by
\citet{Marshall1998Cubes}; in particular,
\eqref{eq:intro_volume_ideal_cube} is equivalent to Marshall's formula
for the ideal cube. Related results were obtained independently by
Smith~\cite{Smith1988Thesis,Smith2000Simplexity}. Our results extend
Marshall's formula to  boxes and also to the spherical
setting. Smith also studied volumes of regular ideal crosspolytopes and
obtained their high-dimensional asymptotics
\cite{Smith1988Thesis}. To the best of our knowledge,
formula~\eqref{eq:intro_volume_ideal_crosspolytope} is new.
\end{remark}

\vspace*{5mm}

\section{Background}
\subsection{Klein model}\label{subsec:klein_model}
The underlying set of the Klein model of the geometry with
constant curvature $\kappa\in\R$ is given by
\[
\mathbb B^d(\kappa) := \R^d
\quad (\text{for } \kappa \geq 0),
\qquad
\mathbb B^d(\kappa) := \left\{x\in\mathbb{R}^d:\|x\|<\frac{1}{\sqrt{-\kappa}}\right\}
\quad (\text{for } \kappa<0).
\]
We also need the Euclidean closure of $\mathbb B^d(\kappa)$, which is
given by
\[
\bar{\mathbb B}^d(\kappa) := \R^d
\quad (\text{for } \kappa \geq 0),
\qquad
\bar{\mathbb B}^d(\kappa) := \left\{x\in\mathbb{R}^d:\|x\|\leq \frac{1}{\sqrt{-\kappa}}\right\}
\quad (\text{for } \kappa<0).
\]
The tangent space at $x\in\mathbb B^d(\kappa)$ is canonically identified
with $\R^d$, that is,
$T_x\mathbb B^d(\kappa)\simeq\R^d$.
 In the Klein model, the Riemannian metric at \(x\in \mathbb B^d(\kappa)\) is given by
\begin{equation}\label{eq:def_riemann_metric_g_x}
g_x(u,v)
=
\frac{\langle u,v\rangle}{1+\kappa\|x\|^2}
-
\frac{\kappa\,\langle x,u\rangle\langle x,v\rangle}
{\left(1+\kappa\|x\|^2\right)^2},
\qquad
u,v\in T_x\mathbb B^d(\kappa)\simeq \mathbb{R}^d.
\end{equation}

For $\kappa<0$, the Riemannian manifold
$(\mathbb B^d(\kappa),g)$ is the usual Beltrami--Klein model
of the $d$-dimensional hyperbolic space of constant sectional
curvature $\kappa$. For $\kappa=0$, it is the Euclidean space
$\R^d$, whereas for $\kappa>0$ it is isometric to an open
hemisphere of the $d$-dimensional sphere of constant sectional
curvature $\kappa$, equivalently, of radius $1/\sqrt{\kappa}$.

Here and below, $\delta_{ij}$ denotes the Kronecker delta, and
$I_d=(\delta_{ij})_{i,j=1}^d$ denotes the $d\times d$ identity matrix.
In the standard coordinates, the components of the Riemannian metric are
\[
g_{ij}(x)
=
\frac{\delta_{ij}}{1+\kappa\|x\|^2}
-
\frac{\kappa\,x_i x_j}
{\left(1+\kappa\|x\|^2\right)^2},
\qquad i,j=1,\ldots,d.
\]
The matrix representing the Riemannian metric at \(x\in \mathbb B^d(\kappa)\) is
\[
G_x
=
(g_{ij}(x))_{i, j = 1}^d
=
\frac{1}{1+\kappa\|x\|^2}I_d
-
\frac{\kappa}{(1+\kappa\|x\|^2)^2}xx^{\mathsf T}.
\]
The Riemannian volume of a Borel set $B\subseteq \bar{\mathbb B}^d(\kappa)$ is given by
\begin{equation}\label{eq:riemannian_volume_klein_model}
\Vol_{d,\kappa}(B)  = \int_{B\cap  \mathbb B^d(\kappa)} \frac{\dd y_1 \cdots \dd y_d}{(1+\kappa \| y \|^2)^{(d+1)/2}}\in [0,+\infty].
\end{equation}
The geodesic distance \(\rho_\kappa(x,y)\) between any two points \(x,y\in\mathbb B^d(\kappa)\) is given by
\begin{equation}\label{eq:geodesic_distance_in_B_d_kappa}
\rho_\kappa(x,y)
=
\begin{cases}
\frac{1}{\sqrt{\kappa}}\,
\arccos\!\left(
\frac{1+\kappa\langle x,y\rangle}
{\sqrt{\bigl(1+\kappa\|x\|^2\bigr)
       \bigl(1+\kappa\|y\|^2\bigr)}}
\right),
& \kappa>0,\\
\|x-y\|,
& \kappa=0,\\
\frac{1}{\sqrt{-\kappa}}\,
\operatorname{arcosh}\!\left(
\frac{1+\kappa\langle x,y\rangle}
{\sqrt{\bigl(1+\kappa\|x\|^2\bigr)
       \bigl(1+\kappa\|y\|^2\bigr)}}
\right),
& \kappa<0.
\end{cases}
\end{equation}
For background on hyperbolic geometry and, in particular, on the
Klein model of hyperbolic space, we refer to
Ratcliffe~\cite[Section~6.1]{ratcliffe_book}.

\subsection{Inverse of the metric matrix}

We need to compute the inverse of the matrix $G_x$. This will be done
using the next lemma, which is a special case of the Sherman--Morrison
formula.

\begin{lemma}[Sherman--Morrison]\label{lem:sherman_morrison}
Let $x\in\R^d$ and let $\alpha\in\R$. Suppose that $1+\alpha\|x\|^2\neq 0$.
Then the matrix $I_d+\alpha xx^{\mathsf T}$ is invertible and
\[
\left(I_d+\alpha xx^{\mathsf T}\right)^{-1}
=
I_d-\frac{\alpha}{1+\alpha\|x\|^2}xx^{\mathsf T}.
\]
\end{lemma}

\begin{proof}
We look for an inverse of the form
$
I_d+\beta xx^{\mathsf T}
$
for a suitable scalar $\beta\in\R$. Since
$
xx^{\mathsf T}xx^{\mathsf T}
=
x(x^{\mathsf T}x)x^{\mathsf T}
=
\|x\|^2xx^{\mathsf T},
$
we obtain
\begin{align*}
\left(I_d+\alpha xx^{\mathsf T}\right)
\left(I_d+\beta xx^{\mathsf T}\right)
=
I_d+(\alpha+\beta)xx^{\mathsf T}
+\alpha\beta\,xx^{\mathsf T}xx^{\mathsf T}
=
I_d+
\left(\alpha+\beta+\alpha\beta\|x\|^2\right)
xx^{\mathsf T}.
\end{align*}
Thus, $\beta$ must satisfy
$
\alpha+\beta+\alpha\beta\|x\|^2=0
$.
Since $1+\alpha\|x\|^2\neq0$, we may choose
$
\beta
=
-\frac{\alpha}{1+\alpha\|x\|^2}
$.
This proves the claim.
\end{proof}

\begin{lemma}[Inverse metric matrix]
\label{lem:inverse_metric_matrix_Klein_model}
Let $x\in\mathbb B^d(\kappa)$, where $\kappa\in\R$, and let $G_x$
denote the matrix of the Riemannian metric at $x$. Then
\[
G_x^{-1}
=
\bigl(1+\kappa\|x\|^2\bigr)
\left(I_d+\kappa xx^{\mathsf T}\right).
\]
\end{lemma}

\begin{proof}
Since $x\in\mathbb B^d(\kappa)$, we have
$
1+\kappa\|x\|^2>0
$.
We factor $G_x$ as
\[
G_x
=
\frac{1}{1+\kappa\|x\|^2}
\left(
I_d
-
\frac{\kappa}{1+\kappa\|x\|^2}xx^{\mathsf T}
\right).
\]
Set
$
\alpha
:=
-\frac{\kappa}{1+\kappa\|x\|^2}
$.
Then
$
1+\alpha\|x\|^2
=
1-\frac{\kappa\|x\|^2}{1+\kappa\|x\|^2}
=
\frac{1}{1+\kappa\|x\|^2}
>0
$.
The Sherman--Morrison formula, Lemma~\ref{lem:sherman_morrison},  therefore gives
\begin{align*}
\left(
I_d
-
\frac{\kappa}{1+\kappa\|x\|^2}xx^{\mathsf T}
\right)^{-1}
=
\left(I_d+\alpha xx^{\mathsf T}\right)^{-1}
=
I_d-\frac{\alpha}{1+\alpha\|x\|^2}xx^{\mathsf T}
=
I_d+\kappa xx^{\mathsf T}.
\end{align*}
Multiplying by $1+\kappa\|x\|^2$ yields the claimed formula for
$G_x^{-1}$.
\end{proof}

\subsection{Tangent and normal cones}\label{subsec:tangent_normal_cones}
Fix a curvature $\kappa\in\R$ and a dimension $d\geq 1$. By a polytope in the Klein model
we mean the Euclidean convex hull of finitely many points in
$\bar{\mathbb B}^d(\kappa)$. All notions of convexity, affine hull,
relative interior, dimension, and face used below are understood in
the Euclidean sense. This usage is
compatible with the geometry of the Klein model: geodesic segments are
represented by Euclidean line segments, and totally geodesic subspaces
are intersections of $\mathbb B^d(\kappa)$ with Euclidean affine
subspaces.

Let \(P\subseteq \bar{\mathbb  B}^d(\kappa)\) be a full-dimensional polytope.  As a Euclidean polytope, $P$ admits an H-representation
\begin{equation}\label{eq:polytope_H_representation}
P
=
\left\{
y\in \R^d:
\langle a_r,y\rangle \le b_r,\ r=1,\ldots,m
\right\},
\end{equation}
where \(a_r\in\mathbb{R}^d\setminus\{0\}\) and \(b_r\in\mathbb{R}\).
We assume that this H-representation is irredundant, meaning that no
inequality can be omitted without changing $P$. Since $P$ is
full-dimensional, the inequalities then define pairwise distinct
facets of $P$,
\[
F_r
:=
P\cap \{y\in \R^d:\langle a_r,y\rangle=b_r\},
\qquad
r=1,\ldots, m.
\]
(A facet of a $d$-dimensional polytope is a face of dimension $d-1$.) Note that $a_r$ is a Euclidean outward normal vector to $F_r$.

Now let  \(F\) be a face of \(P\) with $F\neq P$.  (By convention, faces are nonempty.)   Then $F$ can be represented as an intersection of all facets containing $F$, that is $F= \bigcap_{r\in I(F)} F_r$, where \(I(F)\) is the set of indices \(r\) such that the corresponding
facet \(F_r\) contains \(F\), i.e.
\[
I(F)
:=
\{r\in\{1,\ldots,m\}:F\subseteq F_r\}
=
\left\{
r\in\{1,\ldots,m\}:
\langle a_r,y\rangle=b_r
\text{ for every } y\in F
\right\}.
\]

We now define the Riemannian tangent and normal cones of $P$ along a
face $F$, based at a point
$x\in\operatorname{relint}(F)\cap\mathbb B^d(\kappa)$.
\begin{definition}[Tangent and normal cones]
Let \(P\subseteq \bar{\mathbb  B}^d(\kappa)\) be a full-dimensional polytope with irredundant representation~\eqref{eq:polytope_H_representation}, and let $F\neq P$ be a face of $P$ such that $\operatorname{relint}(F)\cap\mathbb B^d(\kappa)\neq\varnothing$.
For \(x\in \relint(F) \cap \mathbb B^d(\kappa)\),  the \emph{tangent cone} of \(P\) at \(x\) is
\begin{equation}\label{eq:tangent_cone_def_Klein_model}
T_x(F,P)
=
\left\{
u\in\mathbb{R}^d:
\langle a_r,u\rangle \le 0
\text{ for all } r\in I(F)
\right\} \subseteq T_x \mathbb  B^d(\kappa)\simeq \mathbb{R}^d.
\end{equation}
The \emph{normal cone} of \(P\) at \(x\) is defined as the polar cone of $T_x(F,P)$ with respect to the inner product $g_x$, i.e.,
\begin{equation}\label{eq:normal_cone_def_Klein_model}
N_x(F, P)
:=
\left\{
v\in \R^d:
g_x(v,u)\le 0
\text{ for all }u\in T_x(F,P)
\right\} \subseteq T_x \mathbb  B^d(\kappa)\simeq \mathbb{R}^d.
\end{equation}
\end{definition}

For $\kappa\geq0$, the condition
$
\operatorname{relint}(F)\cap\mathbb B^d(\kappa)\neq\varnothing
$
is automatic. For $\kappa<0$, it fails precisely when $F$ is an ideal
vertex, that is, a vertex of $P$ lying on
$\partial\mathbb B^d(\kappa)$.

Note that, as a subset of $\R^d$, $T_x(F,P)$ coincides with the usual
Euclidean tangent cone of $P$ at $x$. Since
$x\in\operatorname{relint}(F)$, this cone depends only on the face $F$
and not on the choice of $x$. We regard both $T_x(F,P)$ and
$N_x(F,P)$ as cones in the ambient inner product space
$
\bigl(T_x\mathbb B^d(\kappa),g_x\bigr)
\simeq
\bigl(\R^d,g_x\bigr).
$
In particular, the polarity in~\eqref{eq:normal_cone_def_Klein_model}
is taken with respect to $g_x$. Therefore, as a subset of $\R^d$, the
normal cone $N_x(F,P)$ may depend on $x$ and $\kappa$. The next
proposition shows that, for fixed $\kappa$, its \emph{isometry type} does not
depend on the choice of
$x\in\operatorname{relint}(F)\cap\mathbb B^d(\kappa)$. In general,
however, its isometry type may depend on $\kappa$.

The next proposition will be our main tool for identifying the isometry
types of the normal and tangent cones. For vectors $v_1,\ldots,v_k\in\R^d$, we write
\[
\pos\{v_1,\ldots,v_k\}
:=
\left\{
\sum_{j=1}^k\lambda_jv_j:
\lambda_1,\ldots,\lambda_k\geq0
\right\}
\]
for their positive hull.

\begin{proposition}[Representations of the normal cone]\label{prop:normal_cone_in_constant_curvature}
Let $\kappa\in \R$ and $d\geq 1$. Let $P\subseteq\bar{\mathbb B}^d(\kappa)$ be a full-dimensional polytope
given by an irredundant H-representation
\eqref{eq:polytope_H_representation}, and let $F\neq P$ be a face of
$P$ such that $\operatorname{relint}(F)\cap\mathbb B^d(\kappa)\neq\varnothing$.
Then the following assertions hold for every $x\in\operatorname{relint}(F)\cap\mathbb B^d(\kappa)$.
\begin{itemize}
\item[(a)] One has $N_x(F, P)
=
\pos\{n_r(x):r\in I(F)\}$,
where
\begin{equation}\label{eq:riemann_outward_normal}
n_r(x)
:=
G_x^{-1}a_r
=
\bigl(1+\kappa\|x\|^2\bigr)
\left(a_r+\kappa\langle a_r,x\rangle x\right), \qquad r\in I(F).
\end{equation}
Moreover, for every $r\in I(F)$, the vector
$n_r(x)\in T_x\mathbb B^d(\kappa)$ is uniquely characterized by
\begin{equation}\label{eq:riemann_outward_normal_duality}
g_x\bigl(u,n_r(x)\bigr)
=
\langle a_r,u\rangle
\qquad
\text{ for all }
u\in T_x\mathbb B^d(\kappa).
\end{equation}
Finally, the linear hull of $N_x(F,P)$ has dimension $d-\dim F$.
\item[(b)] One has
$
N_x(F, P)
=
\pos\{\tilde n_r(x):r\in I(F)\},
$
where
$$
\tilde n_r(x)
:=
\bigl(1+\kappa\|x\|^2\bigr)^{-1/2} n_r(x)
=
\bigl(1+\kappa\|x\|^2\bigr)^{1/2}
\left(a_r+\kappa\langle a_r,x\rangle x\right),
\qquad r\in I(F).
$$
\item[(c)] The entries of the Gram matrix of the generator system $\{\tilde n_r(x):r\in I(F)\}$ are given by
$$
g_x(\tilde n_r(x), \tilde n_s(x))
=
\langle a_r,a_s\rangle+\kappa b_rb_s,
\qquad
r,s\in I(F).
$$
\item[(d)]
For every \(y\in \operatorname{relint}(F)\cap \mathbb B^d(\kappa)\) there is a linear isometry $\Phi_{x,y}: (T_x \mathbb  B^d(\kappa), g_x) \to (T_y \mathbb  B^d(\kappa), g_y)$ such that
\[
\Phi_{x,y}\bigl(\tilde n_r(x)\bigr)
=
\tilde n_r(y),
\qquad
r\in I(F).
\]
Consequently, $\Phi_{x,y}\bigl(N_x(F,P)\bigr) =N_y(F,P)$
and
$\Phi_{x,y}\bigl(T_x(F,P)\bigr) =T_y(F,P)$. In particular, for fixed $\kappa$, the isometry types of the normal
and tangent cones do not depend on the choice of
$x\in\operatorname{relint}(F)\cap\mathbb B^d(\kappa)$.
\end{itemize}
\end{proposition}
\begin{remark}
The vectors $n_r(x)$ and $\tilde n_r(x)$ are outward Riemannian normal vectors at $x$ to the
geodesic hyperplane
$
\mathbb B^d(\kappa)
\cap
\left\{
y\in\R^d:
\langle a_r,y\rangle=b_r
\right\}
$
supporting the facet $F_r$. This follows from~\eqref{eq:riemann_outward_normal_duality}. In general, they need not be unit normal vectors.
\end{remark}
\begin{remark}[Curvature as a rank-one perturbation of the Gram matrix]
\label{rem:curvature_rank_one_perturbation}
Part~(c) of Proposition~\ref{prop:normal_cone_in_constant_curvature} is one
of the basic observations used throughout the paper. It states that the Gram matrix of the generators of the Riemannian normal cone at $F$ is
\[
\bigl(
g_x(\widetilde n_r(x),\widetilde n_s(x))
\bigr)_{r,s\in I(F)}
=
\bigl(
\langle a_r,a_s\rangle
\bigr)_{r,s\in I(F)}
+
\kappa\,
(b_rb_s)_{r,s\in I(F)}.
\]
Thus, in Klein coordinates, curvature modifies the Gram
matrix $(
\langle a_r,a_s\rangle)_{r,s\in I(F)}$ of the Euclidean active facet normals $a_r$, $r\in I(F)$, by the rank-at-most-one perturbation $\kappa\,
(b_rb_s)_{r,s\in I(F)}$.
Consequently, the Riemannian normal cone at $F$ is isometric to a
Euclidean cone whose generators have an explicit Gram matrix, while
the tangent cone is isometric to its Euclidean polar. Thus, the
computation of the corresponding Riemannian internal and external angles reduces
to the computation of the solid angles of these Euclidean cones.
\end{remark}

\begin{proof}[Proof of Proposition~\ref{prop:normal_cone_in_constant_curvature}]
Let \(x\in \operatorname{relint}(F)\cap \mathbb B^d(\kappa)\).
\emph{Proof of~(a).}
For \(r\in I(F)\), put $n_r(x):=G_x^{-1}a_r$. The explicit inverse
metric formula given in
Lemma~\ref{lem:inverse_metric_matrix_Klein_model} gives
$
n_r(x)
=
(1+\kappa\|x\|^2)
(a_r+\kappa\langle a_r,x\rangle x)
$.
Since $g_x(u,v)=\langle u,G_xv\rangle$, we have
\[
g_x(u,n_r(x))=\langle a_r,u\rangle
\qquad
\text{for all }u\in T_x\mathbb B^d(\kappa)\simeq\R^d.
\]
Using~\eqref{eq:tangent_cone_def_Klein_model}, the tangent cone of
\(P\) at \(x\) can therefore be written as
\[
T_x(F,P)
=
\left\{
u\in\R^d:
g_x(u,n_r(x))\leq 0
\text{ for all }r\in I(F)
\right\}.
\]
By the standard polarity formula for polyhedral cones, it follows that
\[
N_x(F,P)
=
\pos\{n_r(x):r\in I(F)\}
=
\pos\{G_x^{-1}a_r:r\in I(F)\}.
\]

It remains to prove that the dimension of the linear hull of
$N_x(F,P)$ equals $d-\dim F$. The linear hull of $N_x(F,P)$ is the
linear space generated by the vectors $G_x^{-1}a_r$, $r\in I(F)$.
Since $G_x^{-1}$ is nonsingular, this space has the same dimension as
the linear space generated by the vectors $a_r$, $r\in I(F)$. The
orthogonal complement of the latter space with respect to the
Euclidean scalar product is
$
\{
u\in\R^d:
\langle a_r,u\rangle=0
\text{ for all }r\in I(F)
\},
$
which is the linear hull of $F-x$. Hence, this orthogonal complement
has dimension $\dim F$, and the claim follows.

\vspace*{2mm}
\noindent
\emph{Proof of~(b).}
Since $x\in\mathbb B^d(\kappa)$, we have
$
1+\kappa\|x\|^2>0
$.
Since multiplication of each generator by a positive scalar does not change the positive
hull, part~(b) follows from part~(a).

\vspace*{2mm}
\noindent
\emph{Proof of~(c).}
For \(r,s\in I(F)\),
$$
g_x(n_r(x),n_s(x))
=
g_x(G_x^{-1} a_r, G_x^{-1} a_s)
=
\langle G_x^{-1} a_r, G_x G_x^{-1} a_s\rangle
=
\langle  n_r(x), a_s\rangle.
$$
Using formula~\eqref{eq:riemann_outward_normal} for $n_r(x)$
we get
$$
g_x(n_r(x),n_s(x))
=
\bigl(1+\kappa\|x\|^2\bigr)
\left(
\langle a_r,a_s\rangle+\kappa\langle a_r,x\rangle
\langle a_s,x\rangle
\right)
=
\bigl(1+\kappa\|x\|^2\bigr)
\left(
\langle a_r,a_s\rangle+\kappa b_rb_s
\right),
$$
because $x\in F$ and hence $\langle a_r,x\rangle=b_r$ for every \(r\in I(F)\). In view of $\tilde n_r(x)
=
(1+\kappa\|x\|^2)^{-1/2} n_r(x)$, this gives
$$
g_x(\tilde n_r(x),\tilde n_s(x))
=
\langle a_r,a_s\rangle+\kappa b_rb_s,
$$
which is precisely part~(c). Note that the right-hand side does not depend on $x$.

\vspace*{2mm}
\noindent
\emph{Proof of~(d).}
Let $x,y\in\operatorname{relint}(F)\cap\mathbb B^d(\kappa)$.
By part~(c), the systems
$\{\tilde n_r(x):r\in I(F)\}$ and
$\{\tilde n_r(y):r\in I(F)\}$ have the same Gram matrix. Hence, the
assignment
$
\tilde n_r(x)\mapsto\tilde n_r(y)$, $r\in I(F)$,
extends to a linear isometry between the linear spaces generated by
these systems. Extending this isometry to their orthogonal complements,
we obtain a linear isometry
$
\Phi_{x,y}:
(T_x\mathbb B^d(\kappa),g_x)
\longrightarrow
(T_y\mathbb B^d(\kappa),g_y)
$
such that
$
\Phi_{x,y}\bigl(\tilde n_r(x)\bigr)
=
\tilde n_r(y)
$
for all
$r\in I(F)$.
Taking positive hulls yields
$
\Phi_{x,y}(N_x(F,P))=N_y(F,P).
$
Since linear isometries preserve polar cones, it also follows that
$
\Phi_{x,y}(T_x(F,P))=T_y(F,P).
$
\end{proof}

\subsection{Solid angles}\label{subsec:solid_angles}
Let $C\subseteq\R^d$ be a polyhedral cone, let $L$ be its linear hull,
and put $\ell:=\dim L$. The \emph{Euclidean solid angle} of $C$ is defined by
\[
\alpha(C)
:=
\frac{
\Vol_{\ell}\bigl(C\cap\mathbb B^d\bigr)
}{
\Vol_{\ell}\bigl(L\cap\mathbb B^d\bigr)
},
\]
where $\Vol_{\ell}$ denotes the $\ell$-dimensional Euclidean volume
on $L$, and $\mathbb B^d:=\{u\in\R^d:\|u\|\leq1\}$ is the Euclidean unit ball. For $\ell=0$, we use the convention
$\Vol_0(\{0\})=1$, so that $\alpha(\{0\})=1$.

More generally, let $(V,g)$ be a $d$-dimensional real inner product space, and let
$C\subseteq V$ be a polyhedral cone. Choose a linear isometry $\Psi: (V,g)\to (\R^d,\langle\cdot,\cdot\rangle)$.
The \emph{solid angle of $C$ with respect to $g$} is defined by
\[
\alpha(C;g)
:=
\alpha\bigl(\Psi(C)\bigr).
\]
This definition does not depend on the choice of $\Psi$, since any two
such isometries differ by composition with an orthogonal transformation
of $\R^d$, and orthogonal transformations preserve Euclidean solid
angles.

We are now able to define internal and external angles of polytopes in $\bar{\mathbb B}^d(\kappa)$. For $x\in\mathbb B^d(\kappa)$ and a polyhedral cone
$C\subseteq T_x\mathbb B^d(\kappa)$, we write
\[
\alpha_x(C)
:=
\alpha(C;g_x)
\]
for the solid angle of $C$ with respect to the Riemannian inner product
$g_x$.
Now let $P\subseteq\bar{\mathbb B}^d(\kappa)$ be a full-dimensional polytope, and let $F\neq P$ be a face of $P$ such that $\operatorname{relint}(F)\cap\mathbb B^d(\kappa)\neq\varnothing$.
Choose
$x\in\operatorname{relint}(F)\cap\mathbb B^d(\kappa)$.
The \emph{internal angle} of $P$ at $F$ is defined by
\[
\beta_\kappa(F,P)
:=
\alpha_x\bigl(T_x(F,P)\bigr) = \alpha\bigl(T_x(F,P); g_x\bigr).
\]
The \emph{external angle} of $P$ at $F$ is defined by
\[
\gamma_\kappa(F,P)
:=
\alpha_x\bigl(N_x(F,P)\bigr) = \alpha\bigl(N_x(F,P); g_x\bigr).
\]
These definitions do not depend on the choice of $x$. Indeed, for any
$x,y\in\operatorname{relint}(F)\cap\mathbb B^d(\kappa)$, part~(d) of
Proposition~\ref{prop:normal_cone_in_constant_curvature} provides
a linear isometry between  $(T_x \mathbb  B^d(\kappa), g_x)$ and $(T_y \mathbb  B^d(\kappa), g_y)$ that maps $T_x(F,P)$ onto $T_y(F,P)$ and
$N_x(F,P)$ onto $N_y(F,P)$, and solid angles are
invariant under linear isometries.

These definitions apply to every face $F\neq P$ except an ideal vertex. We extend the definition of the internal angle by setting
$
\beta_\kappa(P,P):=1.
$
Moreover, if $d\geq2$, $\kappa<0$, and $v$ is an ideal vertex of $P$,
we set
\[
\beta_\kappa(\{v\},P):=0.
\]

\subsection{Poincar\'e relation for angle sums}\label{subsec:poincare_relation}

The next result gives an alternating relation among the internal angle
sums of a polytope. In even dimensions, it relates this alternating sum
to the Riemannian volume; if, in addition, $\kappa\neq0$, it yields an
expression for the volume in terms of the angle sums.

\begin{theorem}[Poincar\'e relation]
\label{theo:poincare_relation_curvature_kappa}
Let $d\geq2$ and $\kappa\in\R$, and let
$P\subseteq\bar{\mathbb B}^{d}(\kappa)$ be a full-dimensional
polytope. For $j=0,\ldots,d$, let $\mathcal F_j(P)$ denote the set of
$j$-dimensional faces of $P$, and define the $j$th internal angle sum
of $P$ by
\[
\sigma_{j,\kappa}(P)
:=
\sum_{F\in\mathcal F_j(P)}
\beta_\kappa(F,P).
\]
In particular, $\sigma_{d,\kappa}(P)=1$, and ideal vertices contribute
zero to $\sigma_{0,\kappa}(P)$. Then
\[
\sum_{j=0}^d(-1)^j\sigma_{j,\kappa}(P)
=
\begin{cases}
\displaystyle
\frac{2\kappa^{d/2}}{\omega_{d+1}}
\Vol_{d,\kappa}(P),
& \text{if $d$ is even},\\
0,
& \text{if $d$ is odd},
\end{cases}
\]
where
$
\omega_{d+1}
=
\frac{2\pi^{(d+1)/2}}
{\Gamma((d+1)/2)}
$
is the $d$-dimensional surface area of the unit sphere
$\mathbb S^d\subseteq\R^{d+1}$.
\end{theorem}

A unified formulation for Euclidean, spherical, and hyperbolic
polytopes is given in~\cite[Section~4.2, Theorem~4.2.3]{camenga_thesis}.
For the hyperbolic case, a proof can be found in
\cite[pp.~120--121]{alekseevski_vinberg_solodovnikov_book}.
For hyperbolic polytopes, including polytopes with ideal
vertices, see also
\cite[Section~2, Corollary]{heckman_volume_coxeter}.
For hyperbolic simplices, including the extension to simplices with
ideal vertices, see
\cite[Section~11.3, Theorem~11.3.1 and Lemma~2]{ratcliffe_book}
and
\cite[Section~3.1, Theorem~3.1 and Corollary~3.2]
{kellerhals_zehrt_gauss_bonnet}.
For $\kappa=0$, the relation reduces to the Gram relation; see
\cite[Section~14.1]{GruenbaumBook} or
\cite[p.~199]{perles_shephard_angle_sums}.
For $\kappa=1$, it is Sommerville's relation; see
\cite[Section~5, Theorem~(37)]{perles_shephard_angle_sums}
and the original paper~\cite{sommerville_angle_sums_volume}.
A conic formulation is given in~\cite[Section~2.3 and Theorem~4.1]{AmelunxenLotzDCG17}.

\subsection{Orthocentric cones}
In this subsection, we recall the class of orthocentric cones introduced
in~\cite{kabluchko_schange_angles_orthocentric_simplices}. These polyhedral cones
will arise below as tangent and normal cones of polytopes in geometries
of constant sectional curvature, in particular, of boxes and regular
simplices.

\begin{definition}[Orthocentric cones]\label{def:ortho_cone_intro}
Let $m\in\N$, let
$\lambda_0,\lambda_1,\ldots,\lambda_m\in\R\setminus\{0\}$, and let
$\eps_1,\ldots,\eps_m\in\{\pm1\}$. Suppose that, for some $N\geq m$,
the linearly independent vectors $v_1,\ldots,v_m\in\R^N$ satisfy
\begin{equation}\label{eq:v_i_scalar_prod}
\langle v_i,v_j\rangle
=
\frac{1}{\lambda_0}+\frac{\delta_{ij}}{\lambda_i},
\qquad i,j\in\{1,\ldots,m\}.
\end{equation}
Then the polyhedral cone
$
\pos\{\eps_1v_1,\ldots,\eps_mv_m\}
$
is called an \emph{orthocentric cone} with parameters
$\lambda_0,\lambda_1,\ldots,\lambda_m$ and signs
$\eps_1,\ldots,\eps_m$. We denote any cone obtained in this way by
$$
C_m(\lambda_0;\lambda_1,\ldots,\lambda_m;
\eps_1,\ldots,\eps_m).
$$
\end{definition}
It has been shown in~\cite[Lemma~3.2]{kabluchko_schange_angles_orthocentric_simplices} that an orthocentric cone
$
C_m(\lambda_0;\lambda_1,\ldots,\lambda_m;
\eps_1,\ldots,\eps_m)
$
exists if and only if either
$
\lambda_0,\lambda_1,\ldots,\lambda_m>0,
$
or exactly one of the numbers
$\lambda_0,\lambda_1,\ldots,\lambda_m$ is negative, all the others are
positive, and
$
\lambda_0+\lambda_1+\cdots+\lambda_m<0.
$
Under either of these conditions, such a cone can be realized in
$\R^N$ for every $N\geq m$.

We shall use the following duality property of orthocentric cones
from~\cite[Theorem~4.2]{kabluchko_schange_angles_orthocentric_simplices}.

\begin{theorem}[The polar of an orthocentric cone]
\label{theo:orthocentric_cone_dual}
In the setting of Definition~\ref{def:ortho_cone_intro}, the polar in
$\R^N$ of the cone
$
\pos\{\eps_1v_1,\ldots,\eps_mv_m\}
$
is isometric to a direct orthogonal sum of $\R^{N-m}$ and
\begin{equation}\label{eq:C_d_cone_theo:orthocentric_cone_dual}
C_m\left(
-\lambda_0-\lambda_1-\cdots-\lambda_m;
\lambda_1,\ldots,\lambda_m;
\eps_1\sgn(\lambda_1),\ldots,\eps_m\sgn(\lambda_m)
\right).
\end{equation}
\end{theorem}

For our applications, we require a closed-form expression for the
Euclidean solid angles of orthocentric cones with
$\lambda_1,\ldots,\lambda_m>0$. Only orthocentric cones satisfying this
condition will occur below. Their solid angles can be expressed in terms
of the function $\mathbf{g}_m$ defined as follows; see
~\cite[Theorem~5.2]{kabluchko_schange_angles_orthocentric_simplices}
combined with
~\cite[Theorem~7.5]{kabluchko_schange_angles_orthocentric_simplices}.

\begin{definition}[Function $\mathbf{g}_m$]
\label{def:explicit_g_d_formula_intro}
Let $m\geq0$ and $\lambda_1,\ldots,\lambda_m>0$, and suppose that
either
$
\lambda_0>0
$
or
$\lambda_0<-(\lambda_1+\cdots+\lambda_m)$.
For $\eps_1,\ldots,\eps_m\in\{\pm1\}$, define
\begin{equation}\label{eq:g_d_formula_intro_positive}
\mathbf{g}_m(\lambda_0;\lambda_1,\ldots,\lambda_m;
\eps_1,\ldots,\eps_m)
=
\frac{1}{\sqrt{2\pi}}
\int_{-\infty}^{\infty}
\prod_{j=1}^m
\Phi\left(
\eps_j\sqrt{\frac{\lambda_j}{\lambda_0}}\,x
\right)
\eee^{-x^2/2}\dd x.
\end{equation}
Here, $\Phi$ denotes the analytic continuation to $\CC$ of the standard
normal distribution function; see~\eqref{eq:intro_entire_normal_distribution_function}.
Here and below, for $a>0$, we use the convention $\sqrt{-a}:=\ii\sqrt{a}$.
For $m=0$, the product in~\eqref{eq:g_d_formula_intro_positive} is
understood to be equal to $1$, so that
$
\mathbf{g}_0(\lambda_0)=1.
$
\end{definition}

\begin{theorem}[Solid angles of orthocentric cones]
\label{theo:ortho_cone_angle}
Let $m\in\N$ and $\lambda_1,\ldots,\lambda_m>0$, and suppose that
either
$
\lambda_0>0
$
or
$\lambda_0<-(\lambda_1+\cdots+\lambda_m)$.
Then, for every $\eps_1,\ldots,\eps_m\in\{\pm1\}$, the Euclidean
solid angle of the orthocentric cone
$
C
=
C_m(\lambda_0;\lambda_1,\ldots,\lambda_m;
\eps_1,\ldots,\eps_m)
$
is given by
\begin{equation}\label{eq:solid_angle_orthocentric_cone}
\alpha(C)
=
\mathbf{g}_m\left(
-\lambda_0-\lambda_1-\ldots-\lambda_m;
\lambda_1,\ldots,\lambda_m;
\eps_1,\ldots,\eps_m
\right).
\end{equation}
\end{theorem}

\section{Angles and volumes of boxes}
In this section, we shall derive explicit one-dimensional integral representations for internal and external angles of boxes in geometries of constant curvature, and for their volumes.
Fix a dimension $d\geq 2$ and a curvature $\kappa \in \R$. For parameters $\tau_1^-,\ldots,\tau_d^- >0$ and $\tau_1^+, \ldots, \tau_d^+>0$ consider the box
\begin{equation}\label{eq:box_rectangular_def}
P
:=
[-\tau_1^-,\tau_1^+]\times\cdots\times[-\tau_d^-,\tau_d^+]
=
\left\{
x\in\mathbb{R}^d:
-\tau_i^- \le x_i \le \tau_i^+,\ i=1,\ldots,d
\right\}
.
\end{equation}
We assume that \(P\) is contained in the closure of the Klein model $\bar{\mathbb B}^d(\kappa)$. This is always the case for $\kappa \geq 0$. For $\kappa< 0$, this is equivalent to requiring that every vertex of \(P\) lies in \(\bar{\mathbb B}^d(\kappa) \) and hence the  containment condition is
\[
\sum_{i=1}^d \left(\max\{\tau_i^-,\tau_i^+\}\right)^2
\leq
-\frac{1}{\kappa}\qquad   \qquad (\kappa <0).
\]

\subsection{Internal and external angles of boxes}
Our aim is to derive explicit one-dimensional integral representations for the internal and
external angles of $P$ at its faces.  For \(1\le p\le d\), consider the \((d-p)\)-dimensional face
\begin{equation}\label{eq:box_rectangular_face_def}
F
:=
\{x\in P:x_1=\tau_1^+,\ldots,x_p=\tau_p^+\}
=
\{\tau_1^+\}\times\cdots\times\{\tau_p^+\}
\times
[-\tau_{p+1}^-,\tau_{p+1}^+]\times\cdots\times[-\tau_d^-,\tau_d^+].
\end{equation}
After permuting the coordinates and, if necessary, reflecting some
coordinate axes, there is no loss of generality in considering a face
of this form. The next theorem describes the tangent and normal cones at $F$ in terms
of orthocentric cones with explicit parameters and provides closed-form
expressions for the corresponding angles. We note that the ambient space
of the tangent and normal cones is the inner product space
$(T_x\mathbb B^d(\kappa),g_x)\simeq(\R^d,g_x)$, whereas the orthocentric
cones appearing below are regarded as cones in $\R^d$ endowed with the
standard scalar product $\langle\cdot,\cdot\rangle$.

\begin{theorem}[Internal and external angles of boxes] \label{theo:hyperbolic_angles_boxes}
Let $d\geq 2$ be an integer, $\kappa\in \R\backslash\{0\}$ and let $P$ be a box defined by~\eqref{eq:box_rectangular_def} for some  $\tau_1^-,\ldots,\tau_d^- >0$ and $\tau_1^+, \ldots, \tau_d^+>0$. Assume that $P \subseteq \bar{\mathbb B}^d(\kappa)$. For $1\leq p\leq d$, let $F$ be the face of $P$ defined
by~\eqref{eq:box_rectangular_face_def} and suppose that, when $\kappa<0$,
the face $F$ is not an ideal vertex. Let
$x\in\operatorname{relint}(F)\cap\mathbb B^d(\kappa)$.   Then the following assertions hold.
\begin{itemize}
\item[(a)] The normal cone $N_x(F,P)\subseteq (T_x \mathbb  B^d(\kappa), g_x)$ of $P$ at $F$ is isometric to the orthocentric cone
\begin{equation}\label{eq:box_const_curv_normal_cone_as_orthocentric}
C_p\left(\frac 1 \kappa; (\tau_1^+)^2,\ldots, (\tau_p^+)^2; 1,\ldots, 1\right).
\end{equation}
\item[(b)] The external angle at $F$ is given by
\begin{align*}
\alpha_x(N_x(F,P))
&=
\mathbf{g}_p \left(-\frac 1 {\kappa} - (\tau_1^+)^2 - \ldots - (\tau_p^+)^2; (\tau_1^+)^2,\ldots, (\tau_p^+)^2; 1,\ldots, 1\right) \\
&=
\frac{1}{\sqrt{2 \pi}} \int_{-\infty}^{\infty} \prod_{j=1}^p \Phi \left(\frac{\tau_j^+\, z }{\sqrt{-\frac 1 {\kappa} - (\tau_1^+)^2 - \ldots - (\tau_p^+)^2}} \right) \eee^{-z^2/2} \dd z.
\end{align*}
\item[(c)] The tangent cone $T_x(F,P)\subseteq (T_x \mathbb  B^d(\kappa), g_x)$ of $P$ at $F$ is isometric to the direct orthogonal sum of $\R^{d-p}$ and the orthocentric cone
\begin{equation}\label{eq:box_const_curv_tangent_cone_as_orthocentric}
C_p\left(-\frac 1\kappa - (\tau_1^+)^2 - \ldots - (\tau_p^+)^2; (\tau_1^+)^2,\ldots, (\tau_p^+)^2; 1,\ldots, 1\right).
\end{equation}
\item[(d)] The internal angle at $F$ is given by
\begin{align*}
\alpha_x(T_x(F,P))
=
\mathbf{g}_p \left(\frac 1 {\kappa}; (\tau_1^+)^2,\ldots, (\tau_p^+)^2; 1,\ldots, 1\right)
=
\frac{1}{\sqrt{2 \pi}} \int_{-\infty}^{\infty} \prod_{j=1}^p \Phi \left(\tau_j^+ \sqrt {\kappa}\, z \right) \eee^{-z^2/2} \dd z.
\end{align*}
\end{itemize}
\end{theorem}
\begin{proof}
\emph{Proof of (a).}
The irredundant H-representation of $P$ is
$$
P = \{x\in \R^d: \langle x, e_i \rangle\leq \tau_i^+, \langle x, -e_i \rangle\leq \tau_i^-, i=1,\ldots, d\}.
$$
The facets of $P$ have the form $P\cap \{x_i = \tau_i^+\}$ and $P\cap \{x_i = -\tau_i^-\}$, for $i=1,\ldots, d$. The facets of $P$ containing $F$ are precisely $P \cap \{x_i = \tau_i^+\}$ for $i=1,\ldots, p$. Their Euclidean outward normal vectors can be chosen to be $a_1=e_1,\ldots, a_p = e_p$, with corresponding $b_1 = \tau_1^+, \ldots, b_p = \tau_p^+$.

Now let $x\in\operatorname{relint}(F)\cap\mathbb B^d(\kappa)$. By Proposition~\ref{prop:normal_cone_in_constant_curvature}, the normal cone $N_x(F,P)$ is positively spanned by the vectors $\tilde n_1(x),\ldots, \tilde n_p(x)$ satisfying
$$
g_x(\tilde n_r(x), \tilde n_s(x))
=
\langle a_r,a_s\rangle+\kappa b_rb_s = \delta_{rs} + \kappa \tau_r^+ \tau_s^+, \qquad r,s = 1,\ldots, p.
$$
The Riemannian scalar products of the vectors
$
n_1'(x):=\tilde n_1(x)/\tau_1^+,
\ldots,
n_p'(x):=\tilde n_p(x)/\tau_p^+
$
are given by
\begin{align*}
g_x(n_r'(x),n_s'(x))
=
\frac{
g_x(\tilde n_r(x),\tilde n_s(x))
}{
\tau_r^+\tau_s^+
}
=
\frac{
\delta_{rs}+\kappa\tau_r^+\tau_s^+
}{
\tau_r^+\tau_s^+
}
=
\frac{\delta_{rs}}{(\tau_r^+)^2}+\kappa,
\qquad r,s=1,\ldots,p.
\end{align*}
By Definition~\ref{def:ortho_cone_intro}, the normal cone $N_x(F,P) = \pos (n_1'(x),\dots, n_p'(x))$ is isometric to the orthocentric cone~\eqref{eq:box_const_curv_normal_cone_as_orthocentric}, proving (a).

\vspace*{2mm}
\noindent
\emph{Proof of (c).} The tangent cone $T_x(F,P)$ is the polar cone of $N_x(F,P)$ in $(T_x\mathbb B^d(\kappa), g_x)$.
Let $I:T_x\mathbb B^d(\kappa) \to \R^d$ be an isometry between $(T_x\mathbb B^d(\kappa), g_x)$ and the Euclidean space $\R^d$ endowed with the standard scalar product  $\langle\cdot, \cdot \rangle$. Then $IT_x(F,P)$ is the polar cone of $IN_x(F,P)$ in the Euclidean space $\R^d$. Now, $IN_x(F,P)$ is isometric to~\eqref{eq:box_const_curv_normal_cone_as_orthocentric} by part~(a) and the dimension of its linear hull is $d-(d-p) = p$ by Proposition~\ref{prop:normal_cone_in_constant_curvature}, part~(a). By Theorem~\ref{theo:orthocentric_cone_dual},  the polar cone of $IN_x(F,P)$ is isometric to a direct orthogonal sum of $\R^{d-p}$ and the orthocentric cone~\eqref{eq:box_const_curv_tangent_cone_as_orthocentric}. This proves part~(c).

\vspace*{2mm}
\noindent
\emph{Proof of (b) and (d).}
Since taking a direct orthogonal sum with a Euclidean space does not
change the solid angle of a cone, both parts follow from the
characterizations of the normal and tangent cones in~(a) and~(c),
combined with Theorem~\ref{theo:ortho_cone_angle} and Definition~\ref{def:explicit_g_d_formula_intro}. In deriving the
integral representation in~(b), the square root obtained directly from
Definition~\ref{def:explicit_g_d_formula_intro} may differ from the
displayed one by a common factor $-1$ in all arguments of $\Phi$. This
does not change the integral, by the substitution $z\mapsto-z$.
\end{proof}

In Theorem~\ref{theo:hyperbolic_angles_boxes}, we excluded the Euclidean
case.
For $\kappa=0$, the situation is much simpler. For every
$1\leq p\leq d$ and every $(d-p)$-dimensional face $F$ of $P$, the
tangent cone is isometric to
$
\R^{d-p}\oplus[0,\infty)^p
$.
Hence, both the internal and external angles at $F$ are equal to
$2^{-p}$.

\subsection{Volumes of boxes}
The next result states an explicit one-dimensional integral representation for the Riemannian volume of boxes.
\begin{theorem}[Volume of boxes]\label{theo:vol_boxes}
Let $d\geq 2$, $\kappa\in \R\backslash\{0\}$ and let $P=[-\tau_1^-, \tau_1^+] \times \ldots \times [-\tau_d^-, \tau_d^+]$ for some  $\tau_1^-,\ldots,\tau_d^- >0$ and $\tau_1^+, \ldots, \tau_d^+>0$. Assume that $P \subseteq \bar{\mathbb B}^d(\kappa)$. Then the Riemannian volume of $P$ is given by
\begin{equation}\label{eq:vol_boxes}
\Vol_{d,\kappa}(P)
=
\frac{\omega_{d+1}}{\kappa^{d/2}\sqrt{2\pi}} \int_0^\infty \prod_{j=1}^d \left(\Phi\left(\tau_j^+ \sqrt\kappa\, z\right) -\Phi\left(-\tau_j^-\sqrt\kappa\, z\right) \right) \eee^{-z^2/2}\,\dd z,
\end{equation}
where
$
\omega_{d+1}
=
\frac{2\pi^{(d+1)/2}}
{\Gamma((d+1)/2)}
$
is the $d$-dimensional surface area of the unit sphere
$\mathbb S^d\subseteq\R^{d+1}$ and, for $\kappa<0$, we understand
$
\kappa^{d/2}:=(\sqrt{\kappa})^d
=\ii^d(-\kappa)^{d/2}
$.
\end{theorem}

It is straightforward to check that, as $\kappa\to0$, the right-hand
side of~\eqref{eq:vol_boxes} converges to
$
\prod_{j=1}^d(\tau_j^++\tau_j^-)
$,
which is the Euclidean volume of $P$.

\begin{proof}[Proof of Theorem~\ref{theo:vol_boxes}]
By~\eqref{eq:riemannian_volume_klein_model}, the Riemannian volume is given by
$$
\Vol_{d,\kappa}(P) = \int_{P\cap \mathbb B^d(\kappa)} (1+\kappa \|x\|^2)^{-(d+1)/2} \dd x.
$$
The definition of the Gamma function entails that, for $a>0$ and $u>0$,
$$
u^{-a}
=
\frac{1}{\Gamma(a)} \int_0^{\infty} t^{a-1} \eee^{-u t}\, \dd t.
$$
For  all $x\in P\cap \mathbb B^d(\kappa)$ we have $1+\kappa \|x\|^2 > 0$, and therefore obtain
$$
(1+\kappa \|x\|^2)^{-(d+1)/2}
=
\frac{1}{\Gamma \left( \frac{d+1}{2} \right)} \int_0^{\infty} t^{(d-1)/2} \eee^{-t} \eee^{-\kappa t \|x\|^2}\, \dd t.
$$
Since
$
P\setminus\mathbb B^d(\kappa)
\subseteq\partial\mathbb B^d(\kappa)
$
has Lebesgue measure zero, we may replace
$P\cap\mathbb B^d(\kappa)$ by $P$ in the integral. Tonelli's theorem
then yields
\begin{align*}
\Vol_{d,\kappa}(P)
=
\frac{1}{\Gamma \left( \frac{d+1}{2} \right)} \int_0^{\infty} t^{(d-1)/2} \eee^{-t} \prod_{j=1}^d \bigg( \int_{-\tau_j^-}^{\tau_j^+} \eee^{-\kappa t x_j^2}\, \dd x_j \bigg) \, \dd t.
\end{align*}
For $t>0$, differentiation of the entire function $\Phi$ gives
\[
\frac{\dd}{\dd x}
\left[
\frac{\sqrt{2\pi}}{\sqrt{2\kappa t}}\,
\Phi\left(\sqrt{2\kappa t}\,x\right)
\right]
=
\eee^{-\kappa t x^2}.
\]
Consequently,
\[
\int_{-\tau_j^-}^{\tau_j^+}\eee^{-\kappa t x_j^2}\dd x_j
=
\frac{\sqrt{2\pi}}{\sqrt{2\kappa t}}
\left(
\Phi\left(\sqrt{2\kappa t}\,\tau_j^+\right)
-
\Phi\left(-\sqrt{2\kappa t}\,\tau_j^-\right)
\right).
\]
Hence,
\begin{align*}
\Vol_{d,\kappa}(P)
&=
\frac{\pi^{d/2}}{\Gamma\left(\frac{d+1}{2}\right) \kappa^{d/2}} \int_0^{\infty} t^{-1/2} \eee^{-t} \prod_{j=1}^d \left( \Phi\left( \sqrt{2 \kappa t} \tau_j^+ \right) - \Phi\left( - \sqrt{2 \kappa t} \tau_j^- \right) \right) \, \dd t\\
&=
\frac{\omega_{d+1}}{\kappa^{d/2} \sqrt{2 \pi}} \int_0^{\infty} \prod_{j=1}^d \left( \Phi\left( \sqrt{\kappa} \tau_j^+ z \right) - \Phi\left( - \sqrt{\kappa} \tau_j^- z \right) \right) \eee^{-z^2/2}\, \dd z,
\end{align*}
using the substitution $z=\sqrt{2t}$ and noticing that $\omega_{d+1}=\frac{2\pi^{(d+1)/2}}{\Gamma((d+1)/2)}$.
\end{proof}

\begin{remark}[Relation to earlier work on hyperbolic cubes]
For regular hyperbolic cubes, closely related volume formulas were
obtained earlier by Marshall~\cite{Marshall1998Cubes}; related results
were obtained independently by Smith~\cite{Smith1988Thesis,Smith2000Simplexity}.
In particular, the specialization of Theorem~\ref{theo:vol_boxes} to
regular hyperbolic cubes is equivalent to Marshall's formula.
Theorem~\ref{theo:vol_boxes} extends this approach to general boxes and
also to the spherical setting.
For the regular ideal hyperbolic cube, Marshall
\cite[Corollary~2]{Marshall1998Cubes} further proved
\[
\Vol_{d,-1}\bigl(\Box_\infty^d(-1)\bigr)
\sim
\sqrt{2}\,L^d d^{-d/2},
\qquad
L=2.52304\ldots,
\qquad d\to\infty.
\]
Smith independently obtained the corresponding, weaker exponential
growth rate~\cite{Smith1988Thesis,Smith2000Simplexity}.
\end{remark}

Assuming that the dimension $d$ is even, it is possible to prove Theorem~\ref{theo:vol_boxes} by combining Poincar\'e's relation (Theorem~\ref{theo:poincare_relation_curvature_kappa}) with Theorem~\ref{theo:hyperbolic_angles_boxes}.   We omit the details.

\subsection{Example: Angles and volumes of cubes}\label{subsec:example_angles_vols_cubes}
In this subsection, we specialize the preceding results on the angles
and volumes of boxes to cubes and express them in terms of the geodesic
edge length of the cube.
Fix a dimension $d\geq2$ and a curvature $\kappa\in\R$. For $\tau>0$,
consider the Euclidean cube $[-\tau,\tau]^d\subseteq\R^d$. We regard $[-\tau,\tau]^d$ as a cube in the Klein model whenever
$[-\tau,\tau]^d\subseteq \bar{\mathbb B}^d(\kappa)$.
Since every vertex of the cube has Euclidean norm $\sqrt d \,  \tau$, this
condition is equivalent to
$
1+\kappa d\tau^2\geq 0.
$
Thus, there is no restriction on $\tau>0$ when $\kappa\geq 0$, whereas
for $\kappa<0$ one must have
$
0<\tau\leq 1/\sqrt{-\kappa d}.
$
In the hyperbolic case $\kappa<0$, the cube is called \emph{ideal} when
$
\tau=1/\sqrt{-\kappa d}
$,
in which case all its vertices lie on $\partial\mathbb B^d(\kappa)$.

\begin{lemma}[Geodesic edge length of a cube] \label{lem:side_length_cube}
Let $d\geq2$, $\kappa\in\R$, and $\tau>0$. Suppose that $[-\tau,\tau]^d\subseteq \mathbb B^d(\kappa)$, equivalently, $1+\kappa d\tau^2> 0$.
Then all edges of the cube have the same geodesic length, denoted by
$\ell_\kappa(\tau)$, and
\[
\ell_\kappa(\tau)
=
\begin{cases}
\displaystyle
\frac{1}{\sqrt{\kappa}}\,
\arccos\left(
\frac{1+\kappa(d-2)\tau^2}
{1+\kappa d\tau^2}
\right),
& \kappa>0,\\
2\tau,
& \kappa=0,\\
\displaystyle
\frac{1}{\sqrt{-\kappa}}\,
\operatorname{arcosh}\left(
\frac{1+\kappa(d-2)\tau^2}
{1+\kappa d\tau^2}
\right),
& \kappa<0.
\end{cases}
\]
As $\tau$ varies over the admissible values, the possible edge
lengths are $0 < \ell_\kappa(\tau)  < \frac{2}{\sqrt{\kappa}} \arcsin \frac{1}{\sqrt d}$ for $\kappa > 0$ and $\ell_\kappa(\tau)\in (0,\infty)$ for $\kappa \leq 0$.
\end{lemma}

\begin{proof}
Coordinate permutations and sign changes are isometries of the Klein
model and act transitively on the edges of the cube. Hence, all edges
have the same geodesic length.
Consider the adjacent vertices $u=(-\tau,\tau,\ldots,\tau)$ and $v=(\tau,\tau,\ldots,\tau)$ of the cube.
Then $\|u\|^2=\|v\|^2=d\tau^2$ and $\langle u,v\rangle=(d-2)\tau^2$,
so the formula for $\ell_\kappa(\tau)$ follows from the expression for
$\rho_\kappa(u,v)$ given in~\eqref{eq:geodesic_distance_in_B_d_kappa}.  The range follows by letting $\tau$ vary over its
admissible interval and observing that $\ell_\kappa(\tau)$ is strictly
increasing.
\end{proof}

\begin{corollary}[Angles and volumes of cubes]\label{cor:angles_volumes_cubes_in_const_curvature}
Let $d\geq2$ and $\kappa\in\R\setminus\{0\}$. Let
$$
\Box_\ell^d(\kappa)=[-\tau,\tau]^d
\subseteq\mathbb B^d(\kappa)
$$
be the cube of geodesic edge length $\ell$.  Such a cube exists for every $\ell\in(0,\infty)$ if $\kappa<0$, whereas
if $\kappa>0$, it exists precisely for
$
0<\ell<\frac{2}{\sqrt{\kappa}}\arcsin\frac{1}{\sqrt d}.
$
Put
\[
C := C_\kappa(\ell)
:=
\begin{cases}
\cosh(\sqrt{-\kappa}\,\ell),&\kappa<0,\\
\cos(\sqrt{\kappa}\,\ell),&\kappa>0.
\end{cases}
\]
Then, for every $k\in\{0,\ldots,d-1\}$, every $k$-dimensional face
$F$ of $\Box_\ell^d(\kappa)$, and every
$x\in\operatorname{relint}(F)$, the external and internal angles at
$F$ are given by
\begin{align}
\alpha_x\bigl(N_x(F,\Box_\ell^d(\kappa))\bigr)
&=
\frac{1}{\sqrt{2\pi}}
\int_{-\infty}^{\infty}
\Phi\left(
\sqrt{\frac{C-1}{2+k(C-1)}}\,y
\right)^{d-k}
\eee^{-y^2/2}\dd y,
\label{eq:external_angles_cube}\\
\alpha_x\bigl(T_x(F,\Box_\ell^d(\kappa))\bigr)
&=
\frac{1}{\sqrt{2\pi}}
\int_{-\infty}^{\infty}
\Phi\left(
\sqrt{-\frac{C-1}{2+d(C-1)}}\,y
\right)^{d-k}
\eee^{-y^2/2}\dd y.
\label{eq:internal_angles_cube}
\end{align}
The Riemannian volume of
$\Box_\ell^d(\kappa)$ is given by
\begin{equation}\label{eq:volume_cube_const_curvature}
\Vol_{d,\kappa}\bigl(\Box_\ell^d(\kappa)\bigr)
=
\frac{\omega_{d+1}}{\kappa^{d/2}\sqrt{2\pi}}
\int_0^\infty
\left(
2\Phi\left(
\sqrt{-\frac{C-1}{2+d(C-1)}}\,z
\right)-1
\right)^d
\eee^{-z^2/2}\dd z.
\end{equation}
For $\kappa<0$, let
$
\Box_{\infty}^d(\kappa)
:=
[-\frac{1}{\sqrt{-\kappa d}},
\frac{1}{\sqrt{-\kappa d}}]^d \subseteq \bar{\mathbb B}^d(\kappa)
$
denote the ideal cube.
Its angle and volume formulas correspond to the limiting values
$\ell\to\infty$ and $C\to\infty$.  More precisely, for every
$k\in\{1,\ldots,d-1\}$, every $k$-dimensional face $F$ of
$\Box_{\infty}^d(\kappa)$, and every
$x\in\operatorname{relint}(F)\cap\mathbb B^d(\kappa)$, one has
\begin{align}
\alpha_x\bigl(N_x(F,\Box_\infty^d(\kappa))\bigr)
&=
\frac{1}{\sqrt{2\pi}}
\int_{-\infty}^{\infty}
\Phi\left(
\frac y{\sqrt k}
\right)^{d-k}
\eee^{-y^2/2}\dd y,
\label{eq:external_angles_cube_ideal}\\
\alpha_x\bigl(T_x(F,\Box_\infty^d(\kappa))\bigr)
&=
\frac{1}{\sqrt{2\pi}}
\int_{-\infty}^{\infty}
\Phi\left(\frac {\ii \, y} {\sqrt d}\right)^{d-k}
\eee^{-y^2/2}\dd y.
\label{eq:internal_angles_cube_ideal}
\end{align}
The Riemannian volume of the ideal cube $\Box_\infty^d(\kappa)$ is given by
\begin{equation}\label{eq:volume_ideal_cube_const_curvature_Phi}
\Vol_{d,\kappa}\bigl(\Box_\infty^d(\kappa)\bigr)
=
\frac{\omega_{d+1}}{\kappa^{d/2}\sqrt{2\pi}}
\int_0^\infty
\left(
2\Phi\left(\frac{\ii z}{\sqrt d}\right)-1
\right)^d
\eee^{-z^2/2}\dd z.
\end{equation}
\end{corollary}

\begin{proof}
Put $p:=d-k$. For a non-ideal cube,
$
1+\kappa d\tau^2>0.
$
In the ideal hyperbolic case, equality holds; when considering its
angles, we assume that $k\geq1$.
By permuting the
coordinates and reflecting coordinate axes, we may assume that
$
F=\{\tau\}^p\times[-\tau,\tau]^k.
$
In particular, $F$ is not an ideal vertex. Hence,
Theorem~\ref{theo:hyperbolic_angles_boxes}, applied with
$\tau_i^-=\tau_i^+=\tau$, yields
\begin{align*}
\alpha_x\bigl(N_x(F,[-\tau,\tau]^d)\bigr)
&=
\frac{1}{\sqrt{2\pi}}
\int_{-\infty}^{\infty}
\Phi\left(
\frac{\tau y}{\sqrt{-1/\kappa-p\tau^2}}
\right)^p
\eee^{-y^2/2}\dd y,\\
\alpha_x\bigl(T_x(F,[-\tau,\tau]^d)\bigr)
&=
\frac{1}{\sqrt{2\pi}}
\int_{-\infty}^{\infty}
\Phi\left(\tau\sqrt{\kappa}\,y\right)^p
\eee^{-y^2/2}\dd y.
\end{align*}
Also, Theorem~\ref{theo:vol_boxes} gives
\begin{equation*}
\Vol_{d,\kappa}\bigl([-\tau,\tau]^d\bigr)
=
\frac{\omega_{d+1}}{\kappa^{d/2}\sqrt{2\pi}}
\int_0^\infty
\left(
2\Phi\left(\tau\sqrt{\kappa}\,z\right)-1
\right)^d
\eee^{-z^2/2}\dd z.
\end{equation*}
Suppose first that $1+\kappa d\tau^2>0$. Since  the geodesic edge length of $[-\tau, \tau]^d$ is $\ell$,
Lemma~\ref{lem:side_length_cube} gives
\[
C=
\frac{1+\kappa(d-2)\tau^2}
     {1+\kappa d\tau^2}.
\]
Solving for $\kappa\tau^2$ and using that $\tau>0$ and $p=d-k$ gives
\[
\kappa\tau^2
=
-\frac{C-1}{2+d(C-1)},
\qquad
\tau\sqrt{\kappa}
=
\sqrt{-\frac{C-1}{2+d(C-1)}},
\qquad
\frac{\tau^2}{-1/\kappa-p\tau^2}
=
\frac{C-1}{2+k(C-1)}.
\]
Substitution into the preceding integral formulas proves
\eqref{eq:external_angles_cube}, \eqref{eq:internal_angles_cube} and~\eqref{eq:volume_cube_const_curvature}.
For the external-angle formula, the square roots determined by the
third identity may differ by a common factor $-1$ in all arguments of
$\Phi$. Such a simultaneous change of sign does not alter the integral,
by the substitution $y\mapsto-y$.

Finally, suppose that $\kappa<0$ and  $\tau=1/\sqrt{-\kappa d}$.
Then
\[
\frac{\tau}{\sqrt{-1/\kappa-p\tau^2}}=\frac1{\sqrt k},
\qquad
\tau\sqrt{\kappa}=\frac{\ii}{\sqrt d}.
\]
The same integral formulas therefore yield
\eqref{eq:external_angles_cube_ideal}, \eqref{eq:internal_angles_cube_ideal} and~\eqref{eq:volume_ideal_cube_const_curvature_Phi}.
\end{proof}

\begin{example}[The Euclidean case $\kappa=0$]
Let
$
\Box_\ell^d(0):=
[-\frac{\ell}{2},\frac{\ell}{2}]^d
$
be the Euclidean cube of edge length $\ell>0$. For every
$k\in\{0,\ldots,d-1\}$, every $k$-dimensional face $F$ of
$\Box_\ell^d(0)$, and every $x\in\operatorname{relint}(F)$, one has
\[
\alpha_x\bigl(N_x(F,\Box_\ell^d(0))\bigr)
=
\alpha_x\bigl(T_x(F,\Box_\ell^d(0))\bigr)
=
\frac1{2^{d-k}}.
\]
Thus, formulas~\eqref{eq:external_angles_cube}
and~\eqref{eq:internal_angles_cube} give the correct values for
$\kappa=0$ if one  sets $C=1$. Moreover,
$
\Vol_{d,0}(\Box_\ell^d(0))=\ell^d.
$
\end{example}

\section{Angles and volumes of acute orthocentric simplices}
In this section, we study a family of simplices called acute
orthocentric simplices, which includes regular simplices as a special
case. We determine their internal and external angles by
characterizing their tangent and normal cones, up to isometry, in terms
of orthocentric cones. In even dimensions, we also obtain a formula for
their Riemannian volumes. Finally, we specialize these results to regular
simplices.

\subsection{Setting}\label{subsec:acute_orthocentric_simplices_setting}
Let us define the family of simplices we are interested in. Fix a dimension $d\geq2$. Let $\tau_0,\ldots,\tau_d>0$ be parameters and put $s=\tau_0^2+\ldots+\tau_d^2$. Consider points $v_0,\ldots,v_d\in \R^{d}$ such that
\begin{equation}\label{eq:def_orthocentric_position}
\lan v_j, v_k \ran = -\frac 1s, \quad 0\leq j \neq k \leq d,
\qquad
\|v_j\|^2 = -\frac1s + \frac{1}{\tau_j^2}, \quad 0\leq j \leq d.
\end{equation}
The existence of such points, as well as their affine independence, is
established in~\cite[Section~3.1]{KabluchkoSchange2025Hyperbolic}.
We are interested in the simplex
$$
P := \conv(v_0, \ldots, v_d) \subset \R^d.
$$
The following facts are easy to verify  directly or can be found in~\cite[Section~3.1]{KabluchkoSchange2025Hyperbolic}. As a Euclidean simplex, $P$ is isometric to the simplex $\conv(e_0/\tau_0, \ldots, e_d/\tau_d) \subseteq \R^{d+1}$, where $e_0,\ldots, e_{d}$ form the  standard orthonormal basis of $\R^{d+1}$.  Moreover, the lines
$\lin(v_0),\ldots,\lin(v_d)$ are the altitudes of $P$ and meet at the
origin. Thus, $P$ is orthocentric with orthocenter $0$. Furthermore,
$0$ lies in the interior of $P$. We therefore refer to $P$ as an
\emph{acute orthocentric simplex}.

Fix a curvature $\kappa\in\R$ and assume that
$P\subseteq\bar{\mathbb B}^d(\kappa)$. Since
$\bar{\mathbb B}^d(\kappa)$ is convex, this is equivalent to $v_0,\ldots,v_d\in\bar{\mathbb B}^d(\kappa)$.  By~\cite[Lemma~3.4]{KabluchkoSchange2025Hyperbolic}, for $v_0, \ldots, v_d \in \R^d$ satisfying~\eqref{eq:def_orthocentric_position}, we have
\begin{equation}\label{eq:containment_condition_simplex}
v_0, \ldots, v_d \in \bar{\mathbb B}^d(\kappa)
\quad \Longleftrightarrow \quad
\kappa
\geq
- \min_{0\leq j \leq d} \frac{1}{1/\tau_j^2 - 1/s}.
\end{equation}

Although we shall not use this fact, the geodesics represented in the
Klein model by the intersections
$\lin(v_j)\cap\mathbb B^d(\kappa)$ are also perpendicular, with respect
to the Riemannian metric, to the corresponding opposite facets of $P$.
Indeed, let $x_j \in\mathbb B^d(\kappa)$ be the intersection point of the Euclidean altitude
$\lin(v_j)$ with the facet opposite $v_j$, and let $u$ be tangent to
this facet. Then
$
\langle v_j,u\rangle=0.
$
Since $x_j\in\lin(v_j)$, we also have
$\langle x_j,u\rangle=0$. Hence, both terms in
\eqref{eq:def_riemann_metric_g_x} vanish when computing
$g_{x_j}(v_j,u)$, so the Euclidean orthogonality is also Riemannian
orthogonality. Consequently, $P$ is orthocentric also in the geometry
of curvature $\kappa$.

\subsection{Internal and external angles of acute orthocentric simplices}
The next result characterizes the tangent and normal cones of $P$ at
its faces and provides one-dimensional integral representations for its internal and
external angles. After relabeling the vertices and the corresponding
parameters, there is no loss of generality in considering a face of the
form $F:=\conv(v_0,\ldots,v_k)$, $k\in\{0,\ldots,d-1\}$.

\begin{theorem}[Internal and external angles of acute orthocentric simplices] \label{theo:hyperbolic_angles_acute_orthocentric_simplices}
Let $\kappa\in \R$, $d\geq 2$, $\tau_0,\ldots,\tau_d>0$, and let $v_0, \ldots, v_d \in \bar {\mathbb B}^d(\kappa)$ be points  satisfying~\eqref{eq:def_orthocentric_position}. Consider the simplex $P := \conv(v_0, \dots, v_d)\subset \bar {\mathbb B}^d(\kappa)$.
Let $F:=\conv(v_0,\ldots,v_k)$ for some  $k\in\{0,\ldots,d-1\}$,
and let $x\in\operatorname{relint}(F)\cap\mathbb B^d(\kappa)$.
If $\kappa\neq s$, then the following assertions hold.
\begin{itemize}
\item[(a)] The normal cone $N_x(F,P)$ of $P$ at $F$ is isometric to the orthocentric cone
\begin{equation}\label{eq:acute_orthocentr_simpl_const_curv_normal_cone_as_orthocentric}
C_{d-k}\left( \frac{s^2}{\kappa-s}; \tau_{k+1}^2, \ldots, \tau_d^2; 1,\ldots, 1 \right).
\end{equation}
\item[(b)] The external angle at $F$ is given by
\begin{align*}
\alpha_x(N_x(F,P))
&=
\mathbf{g}_{d-k}\left( -\frac{s^2}{\kappa-s}-\tau_{k+1}^2 -\ldots -\tau_d^2; \tau_{k+1}^2,\ldots, \tau_d^2; 1,\ldots, 1 \right)\\
&=
\frac{1}{\sqrt{2\pi}} \int_{-\infty}^{\infty} \prod_{j=k+1}^d \Phi\left( \frac{\tau_j}{\sqrt{-\frac{s^2}{\kappa-s}- \tau_{k+1}^2-\ldots-\tau_d^2}} y \right) \eee^{-y^2/2}\,\dd y.
\end{align*}
\item[(c)] The tangent cone $T_x(F,P)$ of $P$ at $F$ is isometric to the direct orthogonal sum of $\R^k$ and the orthocentric cone
\begin{equation}\label{eq:acute_orthocentr_simpl_const_curv_tangent_cone_as_orthocentric}
C_{d-k}\left( -\frac{s^2}{\kappa-s}-\tau_{k+1}^2 -\ldots -\tau_d^2; \tau_{k+1}^2,\ldots, \tau_d^2; 1,\ldots, 1 \right).
\end{equation}
\item[(d)] The internal angle at $F$ is given by
\begin{align*}
\alpha_x(T_x(F,P))
&=
\mathbf{g}_{d-k}\left( \frac{s^2}{\kappa-s}; \tau_{k+1}^2, \ldots, \tau_d^2; 1,\ldots, 1 \right)
\\
&=
\frac{1}{\sqrt{2\pi}} \int_{-\infty}^{\infty} \prod_{j=k+1}^d \Phi\left( \frac{\tau_j\sqrt{\kappa-s}}{s}\, y \right) \eee^{-y^2/2}\,\dd y.
\end{align*}
\end{itemize}

If $\kappa=s$, then $N_x(F,P)$ is isometric to the $(d-k)$-dimensional orthant and $T_x(F,P)$ is isometric to the direct orthogonal sum of $\R^k$ and the $(d-k)$-dimensional orthant. Their angles are
$$
\alpha_x(N_x(F,P)) = \alpha_x(T_x(F,P)) = \frac{1}{2^{d-k}}.
$$
The integral formulas in parts~(b) and~(d) remain valid for
$\kappa=s$ by continuous extension, with all arguments of $\Phi$
understood to be equal to $0$.
\end{theorem}

\begin{remark}
Taking $\kappa = 0$ in Theorem~\ref{theo:hyperbolic_angles_acute_orthocentric_simplices} gives precisely~\cite[Theorem~5.5]{kabluchko_schange_angles_orthocentric_simplices}.
\end{remark}

\begin{proof}[Proof of Theorem~\ref{theo:hyperbolic_angles_acute_orthocentric_simplices}]
\emph{Proof of (a).}
The simplex $P$ has the irredundant H-representation
$$
P
=
\left\{ y\in\R^d:\langle -v_j,y \rangle\leq \frac1s,\ j=0,\ldots,d \right\}
$$
and for each $j\in \{0,\ldots, d\}$, the facet opposite to $v_j$ is
$$
F_j
:=
\conv(v_0,\ldots,\widehat v_j,\ldots,v_d)
=
P \cap \left\{y\in\R^d:\langle -v_j,y\rangle=\frac1s\right\}.
$$
Indeed, $\lan -v_j,v_i \ran = 1/s$ for $i\neq j$, while
$$
\lan -v_j,v_j\ran
=
\frac1s - \frac1{\tau_j^2}
<
\frac1s.
$$
Now let $x\in \relint(F) \cap {\mathbb B}^d(\kappa)$.
By Proposition~\ref{prop:normal_cone_in_constant_curvature},
parts~(b) and~(c),  the normal cone $N_x(F,P)$ is positively spanned by the vectors $\tilde n_{k+1}(x), \ldots, \tilde n_d(x)$ satisfying
$$
g_x(\tilde n_j(x), \tilde n_\ell(x))
=
\langle -v_j,-v_\ell\rangle+\frac{\kappa}{s^2}
=
-\frac1s + \frac{\delta_{j\ell}}{\tau_j^2}+\frac{\kappa}{s^2}
=
\frac{\kappa-s}{s^2} + \frac{\delta_{j\ell}}{\tau_j^2},
\qquad j,\ell\in\{k+1,\ldots,d\}.
$$

If $\kappa\neq s$, then Definition~\ref{def:ortho_cone_intro} shows that
$
N_x(F,P)
=
\pos\{\tilde n_{k+1}(x),\ldots,\tilde n_d(x)\}
$
is isometric to the orthocentric cone
in~\eqref{eq:acute_orthocentr_simpl_const_curv_normal_cone_as_orthocentric}.
If $\kappa=s$, the generators
$\tilde n_{k+1}(x),\ldots,\tilde n_d(x)$ are pairwise orthogonal, and
hence $N_x(F,P)$ is isometric to the $(d-k)$-dimensional orthant.

\vspace*{0.2cm}
\noindent
\emph{Proof of (c).}
The tangent cone $T_x(F,P)$ is the polar cone of $N_x(F,P)$ in
$(T_x\mathbb B^d(\kappa),g_x)$. Let
$
I:(T_x\mathbb B^d(\kappa),g_x)
\to
(\R^d,\langle\cdot,\cdot\rangle)
$
be a linear isometry. Then $IT_x(F,P)$ is the Euclidean polar cone of
$IN_x(F,P)$. By part~(a), the cone $IN_x(F,P)$ is isometric to the
orthocentric cone displayed
in~\eqref{eq:acute_orthocentr_simpl_const_curv_normal_cone_as_orthocentric}.
Moreover, its linear hull has dimension $d-k$, by
Proposition~\ref{prop:normal_cone_in_constant_curvature}, part~(a).
Hence, Theorem~\ref{theo:orthocentric_cone_dual} shows that
$IT_x(F,P)$, and therefore $T_x(F,P)$, is isometric to the direct
orthogonal sum of $\R^k$ and the orthocentric cone displayed
in~\eqref{eq:acute_orthocentr_simpl_const_curv_tangent_cone_as_orthocentric}.

If $\kappa=s$, the asserted description of the tangent cone follows
directly: the Euclidean polar of $IN_x(F,P)$ is isometric to the direct
orthogonal sum of $\R^k$ and a $(d-k)$-dimensional orthant.

\vspace*{0.2cm}
\noindent
\emph{Proof of (b) and (d).} Both parts follow from the characterization of the tangent and normal cones proved in (a) and (c), combined with the formula for the angles of orthocentric cones stated in Theorem~\ref{theo:ortho_cone_angle}. In deriving the integral representation in~(b), the square roots
obtained directly from
Definition~\ref{def:explicit_g_d_formula_intro} may differ from the
displayed ones by a common factor $-1$ in all arguments of $\Phi$.
This does not change the integral, by the substitution $y\mapsto-y$.  If $\kappa=s$, use that the angle of the $(d-k)$-dimensional orthant is $1/2^{d-k}$.
\end{proof}

\subsection{Volumes of orthocentric simplices in even dimension}
Combining
Theorem~\ref{theo:hyperbolic_angles_acute_orthocentric_simplices}
with the Poincar\'e relation
(Theorem~\ref{theo:poincare_relation_curvature_kappa}), we can derive a one-dimensional integral representation for the Riemannian volume of acute orthocentric simplices provided the dimension $d$ is even.

\begin{theorem}[Volume in even dimension]
\label{theo:vol_orthocentric_simplex_acute_even_d}
Suppose that $d\geq 2$ is even and $\kappa\neq0$. Let $P=\conv(v_0,\ldots,v_d)\subseteq
\bar{\mathbb B}^d(\kappa)$ be an acute orthocentric simplex with parameters $\tau_0,\ldots, \tau_d >0$, as in Section~\ref{subsec:acute_orthocentric_simplices_setting}, and put $s=\tau_0^2+\ldots+\tau_d^2$. Then
\begin{align}
\Vol_{d,\kappa}(P)
&=
\frac{\omega_{d+1}}
{2\kappa^{d/2}\sqrt{2\pi}}
\int_{-\infty}^{\infty}
\Bigg[
\prod_{j=0}^d
\Phi\left(
\frac{\tau_j\sqrt{\kappa-s}}{s}x
\right)
+
\prod_{j=0}^d
\Phi\left(
-\frac{\tau_j\sqrt{\kappa-s}}{s}x
\right)
\Bigg]
\eee^{-x^2/2}\dd x.
\label{eq:volume_acute_orthocentric_simplex_even}
\end{align}
\end{theorem}

\begin{proof}
Since $d$ is even, Poincar\'e's relation, Theorem~\ref{theo:poincare_relation_curvature_kappa}, yields
$$
\Vol_{d,\kappa}(P)
=
\frac{\omega_{d+1}}{2\kappa^{d/2}} \sum_{k=0}^d (-1)^k \sigma_{k,\kappa}(P)
=
\frac{\omega_{d+1}}{2\kappa^{d/2}} \sum_{k=0}^d (-1)^k \sigma_{d-k, \kappa}(P).
$$
We first assume that $P\subseteq\mathbb B^d(\kappa)$, so that $P$ has
no ideal vertices. By the formula in part~(d) of
Theorem~\ref{theo:hyperbolic_angles_acute_orthocentric_simplices},
which remains valid for $\kappa=s$, we have
$$
\sigma_{d-k, \kappa}(P)
=
\sum_{\substack{J\subseteq\{0,\ldots,d\}\\ |J|=k}} \frac1{\sqrt{2\pi}} \int_{-\infty}^{\infty} \prod_{j\in J} \Phi\left(\frac{\tau_j\sqrt{\kappa-s}}{s}x\right) \eee^{-x^2/2}\,\dd x,
$$
for every $k\in\{1,\ldots,d\}$.
The same formula also holds for $k=0$ since the internal angle of $P$ at the face $P$ is equal to $1$. Hence,
\begin{align*}
\Vol_{d,\kappa}(P)
&=
\frac{\omega_{d+1}}{2\kappa^{d/2}\sqrt{2\pi}} \int_{-\infty}^{\infty} \eee^{-x^2/2} \sum_{\substack{J\subseteq\{0,\ldots,d\}\\ J\neq\{0,\ldots,d\}}} (-1)^{|J|} \prod_{j\in J} \Phi\left(\frac{\tau_j\sqrt{\kappa-s}}{s}x\right) \,\dd x \\
&=
\frac{\omega_{d+1}}{2\kappa^{d/2}\sqrt{2\pi}} \int_{-\infty}^{\infty} \eee^{-x^2/2} \left[ \prod_{j=0}^d \left( 1 - \Phi\left(\frac{\tau_j \sqrt{\kappa-s}}{s}x\right) \right) +\prod_{j=0}^d \Phi\left(\frac{\tau_j \sqrt{\kappa-s}}{s} x \right) \right]
\,\dd x \\
&=
\frac{\omega_{d+1}}{2\kappa^{d/2}\sqrt{2\pi}} \int_{-\infty}^{\infty} \eee^{-x^2/2} \left[ \prod_{j=0}^d \Phi\left(-\frac{\tau_j\sqrt{\kappa-s}}{s}x\right) + \prod_{j=0}^d \Phi\left(\frac{\tau_j\sqrt{\kappa-s}}{s}x\right) \right]
\,\dd x,
\end{align*}
where we used $1-\Phi(z)=\Phi(-z)$ and the fact that, since $d+1$ is
odd, the omitted full-set term in the binomial expansion has sign $-1$.
This is precisely the claimed
formula when $P$ has no ideal vertices.

It remains to consider the case where $P$ has ideal vertices, which
can occur only when $\kappa<0$. For $r\in(0,1)$, put $P_r:=rP=\conv(rv_0,\ldots,rv_d)$.
Then $P_r\subseteq\mathbb B^d(\kappa)$ has no ideal vertices and $P_r$ is an acute
orthocentric simplex with parameters
\[
\tau_{j,r}:=\frac{\tau_j}{r},
\qquad
s_r:=\frac{s}{r^2}.
\]
Applying the formula already proved to $P_r$ and letting $r\uparrow1$
gives the result. Indeed, the left-hand side converges by monotone
convergence,  whereas the right-hand side converges by dominated
convergence.
To pass to the limit on the right-hand side, fix $r_0\in(0,1)$ and
consider $r\in[r_0,1]$. Put
\[
b_{j,r}:=\frac{\tau_j\sqrt{s-\kappa r^2}}{s},
\qquad j=0,\ldots,d.
\]
Thus, the integrand corresponding to $P_r$ can be written as
\begin{equation}\label{eq:integrand_for_P_r_for_dominated_conv_if_ideal_vertices_exist}
\eee^{-x^2/2}
\sum_{\substack{J\subseteq\{0,\ldots,d\}\\
J\neq\{0,\ldots,d\}}}
(-1)^{|J|}
\prod_{j\in J}\Phi(\ii b_{j,r}x).
\end{equation}
The elementary estimate
$
|\Phi(\ii t)|
\leq
C \, \frac{\eee^{t^2/2}}{1+|t|}
$,
$t\in\R$,
implies, uniformly in $r\in[r_0,1]$, that
\[
\eee^{-x^2/2}
\prod_{j\in J}|\Phi(\ii b_{j,r}x)|
\leq
\frac{C_J}{(1+|x|)^{|J|}}
\exp\left(
-\frac{1-\sum_{j\in J}b_{j,r}^2}{2}x^2
\right),
\]
for every $J\subseteq \{0,\ldots, d\}$. Moreover, by~\eqref{eq:containment_condition_simplex}, the assumption
$P\subseteq\bar{\mathbb B}^d(\kappa)$ implies
\[
\sum_{j\neq i}b_{j,1}^2
=
\frac{(s-\tau_i^2)(s-\kappa)}{s^2}
\leq1,
\qquad i=0,\ldots,d.
\]
The potentially most dangerous case $|J|= d+1$ is absent in~\eqref{eq:integrand_for_P_r_for_dominated_conv_if_ideal_vertices_exist}. If $|J|\leq d-1$, then $J$ omits at least two distinct indices,
say $i$ and $\ell$. Hence,
$
\sum_{j\in J}b_{j,r}^2
\leq
\sum_{j\in J}b_{j,1}^2
\leq
\sum_{j\neq i}b_{j,1}^2-b_{\ell,1}^2
\leq
1-b_{\ell,1}^2
<1
$.
Since there are only finitely many such subsets $J$, the last
inequality is uniform over all of them and over $r\in[r_0,1]$.
The corresponding terms
in~\eqref{eq:integrand_for_P_r_for_dominated_conv_if_ideal_vertices_exist}
therefore admit a common exponentially decaying majorant.  Finally,  if $|J|=d$, then
$\sum_{j\in J}b_{j,r}^2\leq1$, and the corresponding term in~\eqref{eq:integrand_for_P_r_for_dominated_conv_if_ideal_vertices_exist} is bounded
by a constant multiple of $(1+|x|)^{-d}$. Since $d\geq2$ and there
are only finitely many subsets $J$, the integrands~\eqref{eq:integrand_for_P_r_for_dominated_conv_if_ideal_vertices_exist} admit a common
integrable majorant. Dominated convergence therefore applies.
\end{proof}

\begin{remark}[Simplification for $\kappa >0$]
If $\kappa>0$, each of the two summands in the integrand in
Theorem~\ref{theo:vol_orthocentric_simplex_acute_even_d} is absolutely
integrable. Indeed, if $\kappa\geq s$, then all arguments of $\Phi$ are real, so
$0\leq\Phi(t)\leq1$ shows that each summand is bounded in absolute
value by the integrable function $\eee^{-x^2/2}$. On the other hand, for $\kappa \in (0, s)$ the asymptotic formula
$$
\Phi(\ii t) \sim \frac{\ii \eee^{t^2/2}}{\sqrt{2\pi} \, t}, \qquad t\to \pm \infty,
$$
gives
\begin{equation}\label{eq:asympt_e_x^2_times_product_Phis}
\eee^{-x^2/2}
\prod_{j=0}^{d}
\Phi\left(
\eta \frac{\tau_j\sqrt{\kappa-s}}{s}\,x
\right)
\sim
\frac{
(\ii s \eta / x)^{d+1}
\,
\eee^{
-\frac{\kappa}{2s}x^2}
}{
(2\pi)^{(d+1)/2}
(s-\kappa)^{(d+1)/2}
\prod_{j=0}^{d}\tau_j
}
,
\qquad x\to\pm\infty,
\end{equation}
for $\eta \in \{\pm 1\}$, which implies integrability for $\kappa \in (0,s)$.
The substitution $x\mapsto-x$ shows that the integrals of the individual summands are equal, and hence the volume formula in
Theorem~\ref{theo:vol_orthocentric_simplex_acute_even_d} simplifies to
\[
\Vol_{d,\kappa}(P)
=
\frac{\omega_{d+1}}{\kappa^{d/2}\sqrt{2\pi}}
\int_{-\infty}^{\infty}
\prod_{j=0}^d
\Phi\left(
\frac{\tau_j\sqrt{\kappa-s}}{s}x
\right)
\eee^{-x^2/2}\dd x,
\qquad \kappa>0.
\]
For $\kappa<0$, however, the integrals corresponding to the two
products generally do not converge separately---this again follows from~\eqref{eq:asympt_e_x^2_times_product_Phis}; only their sum exhibits
the cancellation needed for convergence.
\end{remark}

A different formula for the volume of $P$, valid in every dimension,
was obtained in~\cite[Theorem~3.6]{KabluchkoSchange2025Hyperbolic}.
That formula involves a conditionally convergent, highly oscillatory
integral.  By contrast, the symmetrized integral in
Theorem~\ref{theo:vol_orthocentric_simplex_acute_even_d} is absolutely
convergent; if $P\subseteq\mathbb B^d(\kappa)$, its integrand decays
exponentially fast.

\subsection{Example: Angles and volumes of regular simplices}\label{subsec:angles_volumes_regular_simplices}
In this subsection, we specialize the preceding results to regular simplices and express them in terms of the geodesic
edge length of the simplex.
\begin{corollary}[Angles and volumes of regular simplices]
\label{cor:angles_regular_simplices_constant_curvature}
Let $d\geq2$ and $\kappa\in\R$. Let
$\Delta_\ell^d(\kappa)\subseteq\mathbb B^d(\kappa)$ be a $d$-dimensional regular
simplex of geodesic edge length $\ell$. Such a simplex exists for every
$\ell\in(0,\infty)$ if $\kappa\leq0$, whereas for $\kappa>0$ it exists
precisely when $0<\ell<\frac1{\sqrt{\kappa}}\arccos(-\frac1d)$.
Put
\[
C:= C_\kappa(\ell) :=
\begin{cases}
\cosh(\ell\sqrt{-\kappa}),&\kappa<0,\\
1,&\kappa=0,\\
\cos(\ell\sqrt{\kappa}),&\kappa>0.
\end{cases}
\]
Then, for every $k\in\{0,\ldots,d-1\}$, every $k$-dimensional face
$F$ of $\Delta_\ell^d(\kappa)$, and every $x\in\relint(F)$, the
external and internal angles at $F$ are given by
\begin{align}
\alpha_x\bigl(N_x(F,\Delta_\ell^d(\kappa))\bigr)
&=
\frac1{\sqrt{2\pi}}
\int_{-\infty}^{\infty}
\Phi\left(
\sqrt{\frac{C}{1+kC}}\,y
\right)^{d-k}
\eee^{-y^2/2}\dd y,
\label{eq:regular_simplex_external_angle}\\
\alpha_x\bigl(T_x(F,\Delta_\ell^d(\kappa))\bigr)
&=
\frac1{\sqrt{2\pi}}
\int_{-\infty}^{\infty}
\Phi\left(
\sqrt{-\frac{C}{1+dC}}\,y
\right)^{d-k}
\eee^{-y^2/2}\dd y.
\label{eq:regular_simplex_internal_angle}
\end{align}
If $d$ is even and $\kappa \neq 0$, then
\begin{equation}\label{eq:volume_regular_simplex_curvature_kappa_even_d}
\Vol_{d,\kappa}\bigl(\Delta_\ell^d(\kappa)\bigr)
=
\frac{\omega_{d+1}}
{2\kappa^{d/2}\sqrt{2\pi}}
\int_{-\infty}^{\infty}
\left(
\Phi\left(
\sqrt{-\frac{C}{1+dC}}\,y
\right)^{d+1}
+
\Phi\left(
-\sqrt{-\frac{C}{1+dC}}\,y
\right)^{d+1}
\right)
\eee^{-y^2/2}\dd y.
\end{equation}

For $\kappa<0$, let $\Delta_\infty^d(\kappa)\subseteq \bar{\mathbb B}^d(\kappa)$ denote the $d$-dimensional regular ideal
simplex whose vertices lie on
$\partial\mathbb B^d(\kappa)$. Its angle and volume formulas correspond to the limiting values
$\ell\to\infty$ and $C\to\infty$.  More
precisely, for every $k\in\{1,\ldots,d-1\}$, every $k$-dimensional
face $F$ of $\Delta_\infty^d(\kappa)$, and every
$x\in\relint(F)\cap\mathbb B^d(\kappa)$, one has
\begin{align}
\alpha_x\bigl(N_x(F,\Delta_\infty^d(\kappa))\bigr)
&=
\frac1{\sqrt{2\pi}}
\int_{-\infty}^{\infty}
\Phi\left(\frac{y}{\sqrt{k}}\right)^{d-k}
\eee^{-y^2/2}\dd y, \label{eq:angle_ideal_reg_simplex_external}
\\
\alpha_x\bigl(T_x(F,\Delta_\infty^d(\kappa))\bigr)
&=
\frac1{\sqrt{2\pi}}
\int_{-\infty}^{\infty}
\Phi\left(\frac{\ii y}{\sqrt d}\right)^{d-k}
\eee^{-y^2/2}\dd y. \label{eq:angle_ideal_reg_simplex_internal}
\end{align}
If, in addition, $d$ is even, then
\begin{equation}\label{eq:volume_regular_simplex_ideal_curvature_kappa_even_d}
\Vol_{d,\kappa}\bigl(\Delta_\infty^d(\kappa)\bigr)
=
\frac{\omega_{d+1}}
{2\kappa^{d/2}\sqrt{2\pi}}
\int_{-\infty}^{\infty}
\left(
\Phi\left(\frac{\ii y}{\sqrt d}\right)^{d+1}
+
\Phi\left(-\frac{\ii y}{\sqrt d}\right)^{d+1}
\right)
\eee^{-y^2/2}\dd y.
\end{equation}
\end{corollary}
\begin{proof}
Take $\tau_0=\cdots=\tau_d=\tau>0$ in the setting of
Section~\ref{subsec:acute_orthocentric_simplices_setting}. Then
$s=(d+1)\tau^2$, and the resulting orthocentric simplex $P = \conv(v_0, \ldots, v_d)$ is a regular simplex centered at $0$. Since any two regular simplices of the same geodesic edge length in the
space of curvature $\kappa$ are isometric, it suffices to prove the
formulas for this centered realization.  By~\eqref{eq:def_orthocentric_position},  $\|v_j\|^2 = d/((d+1)\tau^2)$ and thus the condition that the simplex $P$ be contained in
$\mathbb B^d(\kappa)$ is equivalent to $\tau\in (0,\infty)$ for $\kappa\geq0$ and $\tau > \sqrt{-\kappa d/(d+1)}$ for $\kappa<0$. Also, for $\kappa <0$, the ideal regular simplex corresponds to $\tau = \sqrt{-\kappa d/(d+1)}$.

Let $\ell = \rho_\kappa(v_i,v_j)$, $i\neq j$,  be the geodesic edge length of $P$ and define $C$ as above.
The distance formula~\eqref{eq:geodesic_distance_in_B_d_kappa} in the Klein model
gives
\begin{equation}\label{eq:relation_C_tau_for_reg_simpl}
C
=
\frac{(d+1)\tau^2-\kappa}
{(d+1)\tau^2+d\kappa}.
\end{equation}
Together with the admissible range of $\tau$, \eqref{eq:relation_C_tau_for_reg_simpl} implies that $C$ ranges over
$(1,\infty)$ if $\kappa<0$ and over $(-1/d,1)$ if $\kappa>0$. The  relation between $C$ and $\ell$
then gives  the stated range of possible edge lengths $\ell$ for $\kappa \neq 0$. For $\kappa = 0$,  $\ell=\sqrt2 /\tau$ ranges in $(0,\infty)$.

Let $F$ be a $k$-dimensional face. Since all parameters $\tau_j$ are equal, the
angle formulas in Theorem~\ref{theo:hyperbolic_angles_acute_orthocentric_simplices}
contain $d-k$ identical factors. Moreover, for $\kappa \neq s$,  \eqref{eq:relation_C_tau_for_reg_simpl} implies that
\begin{equation}\label{eq:relation_C_tau_for_reg_simpl_for_parameters_of_Phi}
\frac{\tau^2}
{-\dfrac{s^2}{\kappa-s}-(d-k)\tau^2}
=
\frac{C}{1+kC},
\qquad
\frac{\tau^2(\kappa-s)}{s^2}
=
-\frac{C}{1+dC}.
\end{equation}
The identities in~\eqref{eq:relation_C_tau_for_reg_simpl_for_parameters_of_Phi} extend
continuously to $\kappa=s$, with each side equal to $0$.
Substitution into parts~(b) and~(d) of Theorem~\ref{theo:hyperbolic_angles_acute_orthocentric_simplices}  gives the stated angle formulas~\eqref{eq:regular_simplex_external_angle} and~\eqref{eq:regular_simplex_internal_angle}.  If $d$ is even and $\kappa\neq0$,  Theorem~\ref{theo:vol_orthocentric_simplex_acute_even_d} yields~\eqref{eq:volume_regular_simplex_curvature_kappa_even_d}.

Finally, suppose that $\kappa<0$. The regular ideal simplex corresponds
to $\tau=\sqrt{-\kappa d/(d+1)}$.
For $k\geq1$, the right-hand sides of
\eqref{eq:relation_C_tau_for_reg_simpl_for_parameters_of_Phi} then take their limiting values as $C\to\infty$, namely $1/k$ and
$-1/d$. Substitution into parts~(b) and~(d) of Theorem~\ref{theo:hyperbolic_angles_acute_orthocentric_simplices}  gives the angle formulas~\eqref{eq:angle_ideal_reg_simplex_external} and~\eqref{eq:angle_ideal_reg_simplex_internal}.  If $d$ is even, Theorem~\ref{theo:vol_orthocentric_simplex_acute_even_d} yields the volume formula~\eqref{eq:volume_regular_simplex_ideal_curvature_kappa_even_d}.
\end{proof}

\begin{remark}[Simplification for $\kappa > 0$]
Suppose that $d$ is even. If $\kappa>0$, the two summands in the
integrand in~\eqref{eq:volume_regular_simplex_curvature_kappa_even_d}
are individually integrable. Their integrals are equal by the
substitution $y\mapsto-y$, and hence
\[
\Vol_{d,\kappa}\bigl(\Delta_\ell^d(\kappa)\bigr)
=
\frac{\omega_{d+1}}
{\kappa^{d/2}\sqrt{2\pi}}
\int_{-\infty}^{\infty}
\Phi\left(
\sqrt{-\frac{C}{1+dC}}\,y
\right)^{d+1}
\eee^{-y^2/2}\dd y,
\qquad \kappa>0.
\]
If $\kappa<0$, however, the integrals of the two summands
diverge separately; only their sum exhibits the cancellation required
for convergence. Thus, in negative curvature the symmetrized form
in~\eqref{eq:volume_regular_simplex_curvature_kappa_even_d} and~\eqref{eq:volume_regular_simplex_ideal_curvature_kappa_even_d} must be
retained.
\end{remark}

\begin{example}[The Euclidean case $\kappa=0$]
Let $\Delta_\ell^d(0)$ be a $d$-dimensional Euclidean regular simplex
of edge length $\ell>0$. Then, for every
$k\in\{0,\ldots,d-1\}$, every $k$-dimensional face $F$ of
$\Delta_\ell^d(0)$, and every $x\in\relint(F)$, the external and
internal angles at $F$ are given by
\begin{align*}
\alpha_x\bigl(N_x(F,\Delta_\ell^d(0))\bigr)
&=
\frac1{\sqrt{2\pi}}
\int_{-\infty}^{\infty}
\Phi\left(\frac{y}{\sqrt{k+1}}\right)^{d-k}
\eee^{-y^2/2}\dd y,
\\
\alpha_x\bigl(T_x(F,\Delta_\ell^d(0))\bigr)
&=
\frac1{\sqrt{2\pi}}
\int_{-\infty}^{\infty}
\Phi\left(\frac{\ii y}{\sqrt{d+1}}\right)^{d-k}
\eee^{-y^2/2}\dd y.
\end{align*}
The Euclidean volume of $\Delta_\ell^d(0)$ is well known to be
\[
\Vol_{d,0}\bigl(\Delta_\ell^d(0)\bigr)
=
\frac{\sqrt{d+1}}{d!\,2^{d/2}}\,\ell^d.
\]
If $d$ is even, this value can also be obtained by letting
$\kappa\to0$ in
\eqref{eq:volume_regular_simplex_curvature_kappa_even_d}.
\end{example}

\section{Simplices and cubes with isometric local cones}
\label{sec:cubes_simplices_isometric_local_cones}

\subsection{The simplex--cube correspondence}
\label{subsec:cube_simplex_correspondence}

The preceding results reveal a correspondence between regular
cubes and regular simplices:  after a simple matching of their edge-length parameters, their
normal and tangent cones at faces of the same dimension $k$ are isometric for all $k$ simultaneously.
Let $d\geq 2$ and consider
non-ideal regular polytopes
\[
\Delta:=\Delta_{\ell_\Delta}^d(\kappa_\Delta) \subseteq \mathbb B^d (\kappa_\Delta),
\qquad
\Box:=\Box_{\ell_\Box}^d(\kappa_\Box)\subseteq \mathbb B^d (\kappa_\Box)
\]
in possibly different curvatures  $\kappa_\Delta,\kappa_\Box\in\R$ and with geodesic edge lengths $\ell_\Delta$ and $\ell_\Box$. Put
\[
C_\Delta:=C_{\kappa_\Delta}(\ell_\Delta),
\qquad
C_\Box:=C_{\kappa_\Box}(\ell_\Box),
\quad
\text{ where }
\quad
C_\kappa(\ell):=
\begin{cases}
\cosh(\sqrt{-\kappa}\,\ell),&\kappa<0,\\
1,&\kappa=0,\\
\cos(\sqrt{\kappa}\,\ell),&\kappa>0.
\end{cases}
\]
Recall from Corollaries~\ref{cor:angles_volumes_cubes_in_const_curvature}
and~\ref{cor:angles_regular_simplices_constant_curvature} and their proofs that the
admissible ranges of these parameters are $-\frac1d<C_\Delta<\infty$ and  $\frac{d-2}{d}<C_\Box<\infty$.
The transformation
\begin{equation}
C_\Box=1+2C_\Delta
\label{eq:cube_simplex_matching_parameters}
\end{equation}
maps the first of these intervals bijectively onto the second one. We call a pair $(\Delta,\Box)$ satisfying~\eqref{eq:cube_simplex_matching_parameters} a \emph{matched simplex--cube
pair}.

\begin{proposition}[Local cones of matched simplices and cubes]
\label{prop:cubes_simplices_isometric_local_cones}
Suppose that \eqref{eq:cube_simplex_matching_parameters} holds. Then, for every $k\in\{0,\ldots,d-1\}$, the following
assertions hold.
\begin{itemize}
\item[(a)] If $C_\Delta\neq 0$, the Riemannian normal cones at any
$k$-dimensional face of $\Delta$ and at any $k$-dimensional face of
$\Box$ are isometric to
\[
C_{d-k}\left(
-d-\frac1{C_\Delta};
1,\ldots,1;
1,\ldots,1
\right).
\]
If $C_\Delta=0$, all these normal cones are isometric to the
$(d-k)$-dimensional orthant $[0,\infty)^{d-k}$.

\item[(b)] If $C_\Delta\neq 0$, the Riemannian tangent cones at any
$k$-dimensional face of $\Delta$ and at any $k$-dimensional face of
$\Box$ are isometric to
\[
\R^k\oplus
C_{d-k}\left(
k+\frac1{C_\Delta};
1,\ldots,1;
1,\ldots,1
\right).
\]
If $C_\Delta=0$, all these tangent cones are isometric to $\R^k\oplus[0,\infty)^{d-k}$.

\item[(c)] Consequently, the external and internal angles at
$k$-dimensional faces of $\Delta$ and $\Box$ coincide. Their common
values are, respectively,
\begin{equation*}
\frac{1}{\sqrt{2\pi}}
\int_{-\infty}^{\infty}
\Phi\left(
\sqrt{\frac{C_\Delta}{1+kC_\Delta}}\,y
\right)^{d-k}
\eee^{-y^2/2}\,\dint y
\;\;\text{and}\;\;
\frac{1}{\sqrt{2\pi}}
\int_{-\infty}^{\infty}
\Phi\left(
\sqrt{-\frac{C_\Delta}{1+dC_\Delta}}\,y
\right)^{d-k}
\eee^{-y^2/2}\,\dint y.
\end{equation*}
\end{itemize}
\end{proposition}

\begin{proof}
Suppose first that $C_\Delta\neq 0$. By Theorem~\ref{theo:hyperbolic_angles_boxes} and the identities derived
in the proof of
Corollary~\ref{cor:angles_volumes_cubes_in_const_curvature}, the normal
and tangent cones at a $k$-dimensional face of $\Box$ are isometric,
respectively, to
\[
C_{d-k}\left(
-d-\frac{2}{C_\Box-1};
1,\ldots,1;
1,\ldots,1
\right)
\quad\text{and}\quad
\R^k\oplus
C_{d-k}\left(
k+\frac{2}{C_\Box-1};
1,\ldots,1;
1,\ldots,1
\right).
\]
Here we used the invariance of an orthocentric cone under a common
positive rescaling of the parameters $\lambda_0,\lambda_1,\ldots,\lambda_m$:
for every $c>0$,
\begin{equation}\label{eq:orthocentric_cone_invariance_under_multiplication_of_param}
C_m(c\lambda_0;c\lambda_1,\ldots,c\lambda_m;
\varepsilon_1,\ldots,\varepsilon_m)
\text{ is isometric to }
C_m(\lambda_0;\lambda_1,\ldots,\lambda_m;
\varepsilon_1,\ldots,\varepsilon_m).
\end{equation}

Similarly, Theorem~\ref{theo:hyperbolic_angles_acute_orthocentric_simplices}
and the identities derived in the proof of
Corollary~\ref{cor:angles_regular_simplices_constant_curvature} show
that the normal and tangent cones at a $k$-dimensional face of $\Delta$
are isometric, respectively, to
\[
C_{d-k}\left(
-d-\frac{1}{C_\Delta};
1,\ldots,1;
1,\ldots,1
\right)
\quad\text{and}\quad
\R^k\oplus
C_{d-k}\left(
k+\frac{1}{C_\Delta};
1,\ldots,1;
1,\ldots,1
\right).
\]
By~\eqref{eq:cube_simplex_matching_parameters}, we have
$
\frac{2}{C_\Box-1}
=
\frac{1}{C_\Delta}$,
and the asserted cone isometries follow.

If $C_\Delta= 0$, then $C_\Box=1$. Thus, $\Box$ is Euclidean, and its normal
and tangent cones have the isometry types stated in parts~(a) and~(b).
For the simplex, the case $C_\Delta=0$ corresponds to the orthant case
in Theorem~\ref{theo:hyperbolic_angles_acute_orthocentric_simplices},
and hence gives the same isometry types.

The equality of the external and internal angles follows from the
isometries of the corresponding cones. The integral representations
follow from either
Corollary~\ref{cor:angles_volumes_cubes_in_const_curvature} or
Corollary~\ref{cor:angles_regular_simplices_constant_curvature}, after
using~\eqref{eq:cube_simplex_matching_parameters}.
\end{proof}

\begin{remark}[Geometric meaning of the matching relation]
Let $\alpha_\Delta$ denote the vertex angle of a triangular
two-dimensional face of $\Delta$, and let $\alpha_\Box$ denote the
vertex angle of a square two-dimensional face of $\Box$. Applying the spherical and hyperbolic laws of cosines and sines to the
two-dimensional faces gives
\[
\cos\alpha_\Delta
=
\frac{C_\Delta}{1+C_\Delta},
\qquad
\cos\alpha_\Box
=
\frac{C_\Box-1}{C_\Box+1}.
\]
Consequently, $C_\Box=1+2C_\Delta$ if and only if  $\alpha_\Box=\alpha_\Delta$.
Thus, the matching relation means precisely that the triangular
two-dimensional faces of the simplex and the square two-dimensional
faces of the cube have the same vertex angles.
This also gives a geometric interpretation of the isometry between
the tangent cones at vertices. Indeed, the tangent cone at a vertex
of either polytope is generated by the directions of the $d$ incident
edges. Any two such edges are contained in a triangular face in the
simplex and in a square face in the cube. Hence, for a matched pair, the corresponding systems of unit tangent
vectors along the incident edges have the same Gram matrix.
\end{remark}
\begin{remark}[Curvature types of matched pairs]
The matching relation may connect regular polytopes belonging to
different geometries. The possibilities can be
summarized as follows:
\[
\begin{array}{c|c|c|c}
\text{range of }C_\Delta
&
\Delta
&
\Box
&
\text{range of }C_\Box
\\
\hline
-\dfrac1d<C_\Delta<0
&
\text{spherical simplex, }\alpha_\Delta>\dfrac{\pi}{2}
&
\text{spherical cube}
&
\dfrac{d-2}{d}<C_\Box<1
\\[2mm]
C_\Delta=0
&
\text{spherical orthant, }\alpha_\Delta=\dfrac{\pi}{2}
&
\text{Euclidean cube}
&
C_\Box=1
\\[2mm]
0<C_\Delta<1
&
\text{spherical simplex, }\alpha_\Delta<\dfrac{\pi}{2}
&
\text{hyperbolic cube}
&
1<C_\Box<3
\\[2mm]
C_\Delta=1
&
\text{Euclidean simplex}
&
\text{hyperbolic cube}
&
C_\Box=3
\\[2mm]
1<C_\Delta<\infty
&
\text{hyperbolic simplex}
&
\text{hyperbolic cube}
&
3<C_\Box<\infty
\\[2mm]
C_\Delta\to+\infty
&
\text{ideal hyperbolic simplex}
&
\text{ideal hyperbolic cube}
&
C_\Box\to+\infty.
\end{array}
\]
For example, in the case $C_\Delta=0$, the simplex has edge length $\ell_\Delta=\frac{\pi}{2\sqrt{\kappa_\Delta}}$
and is isometric to the intersection of the sphere with a Euclidean
orthant.
\end{remark}

\subsection{The ideal case and horospherical vertex figures}
\label{subsec:ideal_simplices_cubes_isometric_local_cones}
The ideal limit stated in the last row of the preceding correspondence has an additional property: the local cones of an ideal regular cube
and an ideal regular simplex can be identified with those of a
Euclidean regular simplex one dimension lower.

\begin{proposition}[Local cones of ideal simplices and cubes] \label{prop:matching_ideal_simplices_cubes}
Let $d\geq 2$ and $\kappa<0$. Then, for every $k\in\{1,\ldots,d-1\}$, the following
assertions hold.
\begin{itemize}
\item[(a)] The Riemannian normal cone at any $k$-dimensional face of the ideal
regular cube $\Box_\infty^d(\kappa)$ is isometric to the Riemannian
normal cone at any $k$-dimensional face of the ideal regular simplex
$\Delta_\infty^d(\kappa)$. More precisely, all these normal cones are
isometric to
$
C_{d-k}(-d;1,\ldots,1;1,\ldots,1).
$
\item[(b)] Likewise, the Riemannian tangent cones at such faces of the ideal
regular cube and the ideal regular simplex are all isometric to
$
\mathbb R^k\oplus
C_{d-k}(k;1,\ldots,1;1,\ldots,1).
$
\item[(c)] Moreover, the normal cone at any $(k-1)$-dimensional face of a
Euclidean regular $(d-1)$-simplex is isometric to $C_{d-k}(-d;1,\ldots,1;1,\ldots,1)$,
whereas its tangent cone is isometric to
$\mathbb R^{k-1}\oplus C_{d-k}(k;1,\ldots,1;1,\ldots,1)$.
\item[(d)] Consequently, the external and internal angles at
$k$-dimensional faces of $\Box_\infty^d(\kappa)$ and
$\Delta_\infty^d(\kappa)$ coincide with the corresponding angles at
$(k-1)$-dimensional faces of a Euclidean regular $(d-1)$-simplex.
Their common external and internal angles are, respectively,
\[
\frac{1}{\sqrt{2\pi}}
\int_{-\infty}^{\infty}
\Phi\left(\frac{y}{\sqrt{k}}\right)^{d-k}
\eee^{-y^2/2}\,\dint y
\quad\text{and}\quad
\frac{1}{\sqrt{2\pi}}
\int_{-\infty}^{\infty}
\Phi\left(\frac{\ii y}{\sqrt d}\right)^{d-k}
\eee^{-y^2/2}\,\dint y.
\]
\end{itemize}
\end{proposition}

\begin{proof}
Consider first the ideal regular cube. In its centered realization from
Section~\ref{subsec:example_angles_vols_cubes}, one has $\tau^2=-\frac{1}{\kappa d}$.
Applying Theorem~\ref{theo:hyperbolic_angles_boxes} with $p=d-k$
shows that the normal and tangent cones at any $k$-dimensional face
are isometric, respectively, to
\[
C_{d-k}\left(
\frac{1}{\kappa};
\tau^2,\ldots,\tau^2;
1,\ldots,1
\right)
\quad\text{and}\quad
\R^k\oplus
C_{d-k}\left(
-\frac{1}{\kappa}-(d-k)\tau^2;
\tau^2,\ldots,\tau^2;
1,\ldots,1
\right).
\]
Simultaneously dividing all $\lambda$-parameters of an orthocentric cone by the
positive number $\tau^2$ does not change its isometry type; see~\eqref{eq:orthocentric_cone_invariance_under_multiplication_of_param}. Thus,
the normal and tangent cones are isometric, respectively, to
\[
C_{d-k}(-d;1,\ldots,1;1,\ldots,1)
\quad\text{and}\quad
\R^k\oplus C_{d-k}(k;1,\ldots,1;1,\ldots,1).
\]

Consider next the ideal regular $d$-simplex. In the notation of Section~\ref{subsec:angles_volumes_regular_simplices}, all parameters are equal to some $\tau>0$ such that $\tau^2=-\frac{\kappa d}{d+1}$ and $s=(d+1)\tau^2=-\kappa d$.
By
Theorem~\ref{theo:hyperbolic_angles_acute_orthocentric_simplices},
the normal and tangent cones at any $k$-dimensional face are isometric,
respectively, to
\[
C_{d-k}\left(
\frac{s^2}{\kappa-s};
\tau^2,\ldots,\tau^2;
1,\ldots,1
\right)
\quad\text{and}\quad
\R^k\oplus
C_{d-k}\left(
-\frac{s^2}{\kappa-s}-(d-k)\tau^2;
\tau^2,\ldots,\tau^2;
1,\ldots,1
\right).
\]
Dividing all parameters by $\tau^2$, see~\eqref{eq:orthocentric_cone_invariance_under_multiplication_of_param}, and using the relation $\frac{s^2}{\kappa-s}=-d\tau^2$ therefore gives the same two
isometry types,
\[
C_{d-k}(-d;1,\ldots,1;1,\ldots,1)
\quad\text{and}\quad
\R^k\oplus C_{d-k}(k;1,\ldots,1;1,\ldots,1).
\]
This proves parts~(a) and~(b).

If $d=2$, then $k=1$, and parts~(c) and~(d) follow immediately.  We may therefore assume in the remainder of the proof that $d\geq 3$. Consider a Euclidean regular $(d-1)$-simplex.  In the setting
of Section~\ref{subsec:acute_orthocentric_simplices_setting}, take
all its parameters to be equal to some $\tau>0$. Since there are $d$
vertices, one has $s=d\tau^2$.
Applying
Theorem~\ref{theo:hyperbolic_angles_acute_orthocentric_simplices}
with curvature $0$, ambient dimension $d-1$, and face dimension
$k-1$, we see that the normal and tangent cones at any
$(k-1)$-dimensional face are isometric, respectively, to
\[
C_{d-k}\left(
-d\tau^2;
\tau^2,\ldots,\tau^2;
1,\ldots,1
\right)
\quad\text{and}\quad
\R^{k-1}\oplus
C_{d-k}\left(
k\tau^2;
\tau^2,\ldots,\tau^2;
1,\ldots,1
\right).
\]
Dividing all $\lambda$-parameters by $\tau^2$ gives the same isometry types as above,  which proves part~(c).

Taking solid angles proves the asserted equality of the corresponding
internal and external angles. The integral
representations in part~(d) follow from
Corollary~\ref{cor:angles_regular_simplices_constant_curvature},
applied in dimension $d-1$ to a face of dimension $k-1$, with
curvature $0$ and hence with $C=1$.
\end{proof}

\begin{remark}[Horospherical interpretation]
The preceding result also has a direct geometric explanation. Let $P$
be either an ideal regular $d$-cube or an ideal regular $d$-simplex,
and let $F$ be one of its $k$-dimensional faces. Choose an ideal vertex
$v$ of $F$ and realize the hyperbolic space in the upper half-space
model so that $v$ is the point at infinity. Choose a horosphere $H$ centred at $v$ and sufficiently close to $v$.
In the upper half-space model, $H$ is represented by a horizontal
Euclidean hyperplane. The $d$ edges of $P$ emanating from $v$ are
represented by vertical geodesic rays and therefore meet $H$
orthogonally. Their intersection points with $H$ are precisely the
vertices of the horospherical vertex figure $P\cap H$.

The horospherical vertex figure $P\cap H$, endowed with the intrinsic
Euclidean metric of the horosphere, is a Euclidean regular
$(d-1)$-simplex. Indeed, the $d$
vertices of $P\cap H$ are obtained by intersecting $H$ with the $d$
edges of $P$ emanating from $v$. For both the cube and the simplex, the stabilizer of $v$ in the symmetry group of $P$ is finite and
permutes the $d$ incident edges as the full symmetric group. It
therefore preserves every horosphere centred at $v$, and its
restriction to $H$ acts by Euclidean isometries. Hence, the vertices of $P\cap H$ form a Euclidean regular
simplex.

Note that $F\cap H$ is a $(k-1)$-dimensional face of $P\cap H$. Choose
$x\in\operatorname{relint}(F\cap H)$; then
$x\in\operatorname{relint}(F)$. Every facet of $P$ containing $F$ also
contains $v$. Its supporting hyperplane is therefore vertical in the
upper half-space model, and its intersection with $H$ is a supporting
hyperplane of $P\cap H$ containing $F\cap H$. Since the hyperbolic
metric is conformal to the Euclidean metric, the corresponding normal
directions at $x$ agree up to a common positive scale. Hence the
Riemannian normal cone $N_x(F,P)$ is contained in $T_xH$ and is
isometric to the Euclidean normal cone
$N_x(F\cap H,P\cap H)$ of $P\cap H$ at $F\cap H$.
Taking polar cones in the full hyperbolic tangent space on the one hand
and in $T_xH$ on the other, and using that the orthogonal complement of
$T_xH$ is one-dimensional, shows that
$T_x(F,P)$ is isometric to $\R\oplus T_x(F\cap H,P\cap H)$.
This explains geometrically both the dimension shift
$(d,k)\mapsto(d-1,k-1)$ and the equality of the corresponding internal
and external angles.
\end{remark}

\section{Angles and volumes of rectangular simplices}
Fix a dimension $d\geq 2$, parameters $\tau_1, \ldots, \tau_d > 0$ and consider the \emph{rectangular simplex}
$$
P
:=
\conv\left(0, \frac{e_1}{\tau_1}, \ldots, \frac{e_d}{\tau_d} \right).
$$
We assume that the curvature $\kappa\in \R$ is such that $P$ is contained in $\bar {\mathbb B}^d(\kappa)$, equivalently,
\begin{equation}\label{eq:rect_simplex_kappa_range}
\kappa \geq  - \min(\tau_1^2, \ldots, \tau_d^2).
\end{equation}
In this section, we determine the internal and external angles of $P$.   In even dimensions, we also obtain a formula for
the Riemannian volume of $P$.

\subsection{Internal and external angles of rectangular simplices}
We distinguish faces that contain the vertex $0$, and those which do not. Let $k\in \{0, \ldots, d-1\}$. After relabeling the coordinates and the corresponding parameters,
there is no loss of generality in considering faces of the following
two forms:
$$
F
=
\conv\left(0, \frac{e_1}{\tau_1}, \ldots, \frac{e_k}{\tau_k} \right)
\qquad \text{ and } \qquad
G
=
\conv\left(\frac{e_1}{\tau_1}, \ldots, \frac{e_{k+1}}{\tau_{k+1}} \right).
$$

\begin{theorem}[Internal and external angles of rectangular simplices]\label{theo:hyperbolic_angles_rectangular_orthocentric_simplices}
Let $d\geq2$, let $\tau_1,\ldots,\tau_d>0$, and suppose that~\eqref{eq:rect_simplex_kappa_range} holds.
Put $P:=\conv(0,e_1/\tau_1, \ldots, e_d/\tau_d)$. Then $P\subseteq\bar{\mathbb B}^d(\kappa)$.

\begin{itemize}
\item[(i)] For $k\in \{0, \ldots, d-1\}$, let $F=\conv(0,e_1/\tau_1, \ldots, e_k/\tau_k)$, and let $x\in \relint(F) \cap {\mathbb B}^d(\kappa)$. Then the following assertions hold.
\begin{itemize}
\item[(a)] The normal cone $N_x(F,P)$ is isometric to a $(d-k)$-dimensional orthant.
\item[(b)] The external angle at $F$ is given by $\alpha_x(N_x(F,P)) = 1/2^{d-k}$.
\item[(c)] The tangent cone $T_x(F,P)$ is isometric to the direct orthogonal sum of  $\R^k$ and a $(d-k)$-dimensional orthant.
\item[(d)] The internal angle at $F$ is given by $\alpha_x(T_x(F,P)) = 1/2^{d-k}$.
\end{itemize}

\item[(ii)] For $k\in \{0, \ldots, d-1\}$ let $G=\conv(e_1/\tau_1, \ldots, e_{k+1}/\tau_{k+1})$, and let $x\in \relint(G) \cap {\mathbb B}^d(\kappa)$. Then, the following assertions hold.
\begin{itemize}
\item[(a)] The normal cone $N_x(G,P)$ is isometric to a cone $\pos(w_0, w_{k+2}, \ldots, w_d) \subseteq \R^d$ which is  generated by vectors whose Gram matrix is
{\small
	\begin{equation}\label{eq:rectangular_simplex_normal_cone_Gram_matrix_with_Euclidean_product}
		\left( \left\langle w_i,w_j \right\rangle \right)_{i,j = 0,k+2, \ldots, d}
		=
		\begin{pmatrix}
			\kappa + \tau_1^2 + \ldots + \tau_d^2 & -\tau_{k+2} & -\tau_{k+3} & \cdots & -\tau_d\\
			-\tau_{k+2} & 1 & 0 & \cdots & 0\\
			\vdots & 0 & \ddots & \ddots & \vdots\\
			\vdots& \vdots & \ddots & \ddots & 0\\
			-\tau_d & 0 & \cdots & 0 & 1
		\end{pmatrix}.
	\end{equation}
}

	\item[(b)] The external angle at $G$ is given by
\[
\alpha_x(N_x(G,P))
=
\frac{1}{\sqrt{2\pi}}
\int_0^\infty
\prod_{j=k+2}^d
\Phi\left(
\frac{\tau_j}
{\sqrt{\kappa+\tau_1^2+\ldots+\tau_{k+1}^2}}\,y
\right)
\eee^{-y^2/2}\dd y.
\]

\item[(c)] The tangent cone $T_x(G,P)$ is isometric to the direct orthogonal sum of $\R^k$ and a cone $\pos(v_0, v_{k+2}, \ldots, v_d) \subseteq \R^d$ generated by vectors whose Gram matrix is
{\small
	\begin{multline}\label{eq:rectangular_simplex_tangent_cone_Gram_matrix_with_Euclidean_product}
		\left( \left\langle v_i, v_j \right\rangle \right)_{i,j = 0,k+2, \ldots, d}\\
		=
		\begin{pmatrix}
			1 & \tau_{k+2} & \tau_{k+3} &\cdots & \tau_d\\
			\tau_{k+2} & \kappa + \tau_1^2 + \ldots + \tau_{k+1}^2 + \tau_{k+2}^2 & \tau_{k+2} \tau_{k+3} &\cdots &\tau_{k+2}\tau_d\\
			\tau_{k+3}& \tau_{k+3}\tau_{k+2} & \ddots & \ddots & \vdots\\
			\vdots & \vdots & \ddots & \ddots & \tau_{d-1}\tau_d\\
			\tau_d & \tau_d\tau_{k+2} & \cdots & \tau_d\tau_{d-1} & \kappa + \tau_1^2 + \ldots + \tau_{k+1}^2 + \tau_d^2
		\end{pmatrix}.
	\end{multline}
}

\item[(d)] The internal angle at $G$ is given by
	\begin{align*}
		\alpha_x(T_x(G,P))
		&=
		\frac{1}{2^{d-k}} + \frac{1}{\pi} \int_0^{\infty} \frac{\eee^{-(\kappa + \tau_1^2+\ldots + \tau_d^2)y^2/2}}{y} \Im\left( \prod_{j=k+2}^d \Phi\left(-\ii \tau_j y\right) \right) \dd y.
	\end{align*}
\end{itemize}
\end{itemize}
\end{theorem}
\begin{remark}
With the notation $\mathbf{g}_m$ introduced
in~\cite{kabluchko_schange_angles_orthocentric_simplices}, the angle
formulas in parts~(ii)(b) and~(ii)(d) can equivalently be written as
the following limits; see
\cite[Theorem~8.5]{kabluchko_schange_angles_orthocentric_simplices}:
\begin{align*}
\alpha_x(N_x(G,P))
&=
\lim_{\lambda_1\to\pm\infty}
\mathbf{g}_{d-k}\left(
\kappa+\tau_1^2+\ldots+\tau_{k+1}^2;
\lambda_1,\tau_{k+2}^2,\ldots,\tau_d^2;
1,\ldots,1
\right),
\\
\alpha_x(T_x(G,P))
&=
\lim_{\lambda_1\to\pm\infty}
\mathbf{g}_{d-k}\left(
-(\kappa+\tau_1^2+\ldots+\tau_d^2)-\lambda_1;
\lambda_1,\tau_{k+2}^2,\ldots,\tau_d^2;
\sgn(\lambda_1),1,\ldots,1
\right).
\end{align*}
\end{remark}

\begin{proof}[Proof of Theorem~\ref{theo:hyperbolic_angles_rectangular_orthocentric_simplices}]
The irredundant H-representation of $P$ is
$$P
=
\left\{
y\in\R^d:
\left\langle \tau_1e_1+\ldots+\tau_de_d,y\right\rangle\leq1,
\quad
\langle-e_i,y\rangle\leq0,\ i=1,\ldots,d
\right\}.
$$
The facets of $P$ are
$$
F_0
:=
\left\{ x\in P: \lan \tau_1e_1 + \ldots + \tau_de_d, x \ran = 1 \right\},
\quad
F_i
:=
\left\{x\in P: \lan -e_i, x \ran = 0\right\},
\quad
i=1,\ldots, d.
$$
\emph{Proof of (i).}
Note that $F_i \supseteq F$ exactly for $i=k+1,\ldots, d$. Hence, Proposition~\ref{prop:normal_cone_in_constant_curvature}, parts~(b) and~(c), implies  that the normal cone at $F$ is $N_x(F,P) = \pos(\tilde n_{k+1}(x), \ldots, \tilde n_d(x))$, and the generating vectors satisfy
$$
g_x(\tilde n_i(x), \tilde n_j(x))
=
\lan -e_i, -e_j \ran + \kappa \cdot 0 \cdot 0
=
\delta_{ij},
\qquad
i,j \in \{k+1, \ldots, d\}.
$$
Hence, the  normal cone $N_x(F,P)$ is isometric to the $(d-k)$-dimensional orthant. This implies all assertions of~(i), since the polar cone of an orthant is the direct orthogonal sum of the orthogonal complement of its linear span and a $(d-k)$-dimensional orthant, and the solid angle of a $(d-k)$-dimensional orthant is $1/2^{d-k}$.

\vspace*{0.2cm}
\noindent
\emph{Proof of (ii), part~(a).} Note that $F_i \supseteq G$ exactly for $i=0,k+2,\ldots, d$. Hence, Proposition~\ref{prop:normal_cone_in_constant_curvature}, parts~(b) and (c), yields $N_x(G,P)=\pos(\tilde n_0(x),\tilde n_{k+2}(x),\ldots,\tilde n_d(x))$ for generators $\tilde n_0(x), \tilde n_{k+2}(x), \ldots, \tilde n_d(x)$ satisfying
\begin{align*}
g_x(\tilde n_0(x), \tilde n_0(x))
&=
\lan \tau_1e_1 + \ldots + \tau_de_d, \tau_1e_1 + \ldots + \tau_de_d \ran + \kappa \cdot 1 \cdot 1
=
\kappa + \tau_1^2 + \ldots + \tau_d^2,\\
g_x(\tilde n_0(x), \tilde n_j(x))
&=
\lan \tau_1e_1 + \ldots + \tau_de_d, -e_j \ran + \kappa \cdot 1 \cdot 0
=
-\tau_j,
\qquad  j\in \{k+2, \ldots, d\},
\\
g_x(\tilde n_j(x),\tilde n_\ell(x))
&=
\langle-e_j,-e_\ell\rangle+\kappa\cdot0\cdot0
=
\delta_{j\ell},
\qquad
j,\ell\in\{k+2,\ldots,d\}.
\end{align*}
This proves part (a) of (ii).

\vspace*{0.2cm}
\noindent
\emph{Proof of (ii), part~(c).}
The tangent cone $T_x(G,P)$ is the polar cone of $N_x(G,P)$ in $(T_x{\mathbb B}^d(\kappa), g_x)$. By part (a), the normal cone is isometric to $\pos(w_0, w_{k+2}, \ldots, w_d) \subseteq \R^d$ with scalar products~\eqref{eq:rectangular_simplex_normal_cone_Gram_matrix_with_Euclidean_product}, and the dimension of its linear hull is $d-k$, by Proposition~\ref{prop:normal_cone_in_constant_curvature} (a). Hence, $T_x(G,P)$ is isometric to the direct orthogonal sum of $\R^k$ and the polar cone of $\pos(w_0, w_{k+2}, \ldots, w_d)$ considered as a cone in its linear hull. To complete the proof of (c), we need to show that this cone has generators with the scalar products~\eqref{eq:rectangular_simplex_tangent_cone_Gram_matrix_with_Euclidean_product}.
To this end, let
$$
u
=
\tau_1e_1 + \ldots + \tau_ke_k + \sqrt{\tau_{k+1}^2 + \kappa}\,  e_{k+1}\in \R^d.
$$
Then $\|u\|>0$. Indeed, \eqref{eq:rect_simplex_kappa_range} implies $\tau_{k+1}^2+\kappa\geq 0$, and the only case in which $\|u\|=0$ could occur is $k=0$ and $\kappa=-\tau_1^2$, which is excluded by $x\in\relint(G)\cap {\mathbb B}^d(\kappa)$.
Consider the vectors
$$
w_0=u+\tau_{k+2}e_{k+2}+\ldots+\tau_de_d \in \R^d,
\qquad
w_i=-e_i\in \R^d,\quad i=k+2,\ldots,d.
$$
The  Gram matrix of this system is precisely the matrix in~\eqref{eq:rectangular_simplex_normal_cone_Gram_matrix_with_Euclidean_product}. It is enough to consider this realization, since the cone and its polar, up to isometry, are determined by the Gram matrix of the generators.
Now define
$$
v_0=-\frac{u}{\|u\|}\in \R^d,
\qquad
v_i=-\frac{\tau_i}{\|u\|}u+\|u\|e_i \in \R^d,
\quad i=k+2,\ldots,d.
$$
Since
$u=w_0+\sum_{i=k+2}^d \tau_i w_i$
and  $e_i=-w_i$,
we have \(v_0,v_{k+2},\ldots,v_d\in\lin(w_0,w_{k+2},\ldots,w_d)\). Moreover,
$$
\langle v_i,w_j\rangle=-\|u\|\delta_{ij},
\qquad
i,j\in\{0,k+2,\ldots,d\}.
$$
Hence the polar cone of
$
\pos(w_0,w_{k+2},\ldots,w_d)
$
inside its linear hull is precisely
$
\pos(v_0,v_{k+2},\ldots,v_d).
$
Finally, a computation shows that the scalar products of the $v_i$'s are exactly as displayed in \eqref{eq:rectangular_simplex_tangent_cone_Gram_matrix_with_Euclidean_product}. This completes the proof of part~(c) of~(ii).

\vspace*{0.2cm}
\noindent
\emph{Proof of (ii), parts~(b) and (d).}
Note that the case $\kappa = 0$ is the Euclidean case, and for Euclidean rectangular simplices, the angles are known from~\cite[Theorem~8.5]{kabluchko_schange_angles_orthocentric_simplices}. Let $\kappa\neq 0$. We shall reduce this case to the Euclidean one. To this end,
we choose $\tau_1',\ldots,\tau_{k+1}'>0$ such that
$$
(\tau_1')^2+\ldots+(\tau_{k+1}')^2
=
\kappa+\tau_1^2+\ldots+\tau_{k+1}^2.
$$
This is possible, because the right-hand side is strictly positive. Indeed, from~\eqref{eq:rect_simplex_kappa_range}  we get $\kappa+\tau_1^2+\ldots+\tau_{k+1}^2\geq 0$. Equality can occur only if $k=0$ and $\kappa=-\tau_1^2$ but this contradicts the assumption $x\in \relint(G)\cap {\mathbb B}^d(\kappa)$. Thus the right-hand side is strictly positive.
Put  $\tau_j'=\tau_j$ for $j=k+2,\ldots,d$, and consider the Euclidean rectangular simplex
$$
P'
=
\conv(0,e_1/\tau_1',\ldots,e_d/\tau_d')
$$
with the face
$$
G'
=
\conv(e_1/\tau_1',\ldots,e_{k+1}/\tau_{k+1}').
$$
By construction,  the Gram matrices in \eqref{eq:rectangular_simplex_normal_cone_Gram_matrix_with_Euclidean_product} and \eqref{eq:rectangular_simplex_tangent_cone_Gram_matrix_with_Euclidean_product} coincide with the corresponding Gram matrices of the Euclidean normal and tangent cones of $P'$ at $G'$. Therefore, by parts~(ii)(a) and~(ii)(c), which have already been
proved, the cones $N_x(G,P)$ and $T_x(G,P)$ are isometric to the corresponding Euclidean cones of the Euclidean rectangular simplex $P'$.

The asserted formulas for the external and internal angles now follow directly from the Euclidean formulas for rectangular simplices given in~\cite[Theorem~8.5]{kabluchko_schange_angles_orthocentric_simplices}.
\end{proof}

\subsection{Volumes of rectangular simplices in even dimension}
Combining Theorem~\ref{theo:hyperbolic_angles_rectangular_orthocentric_simplices} with the Poincar\'e relation (Theorem~\ref{theo:poincare_relation_curvature_kappa}) we can derive a one-dimensional integral representation for the Riemannian volume of rectangular simplices provided the dimension $d$ is even.

\begin{theorem}[Volume in even dimension]\label{theo:volume_rectangular_simplex}
Let $d\geq2$ be even, let $\kappa\neq0$, let
$\tau_1,\ldots,\tau_d>0$, and suppose
that~\eqref{eq:rect_simplex_kappa_range} holds.
Then, the Riemannian volume of $P=\conv(0,e_1/\tau_1, \ldots, e_d/\tau_d)$ is given by
$$
\Vol_{d,\kappa}(P) = \frac{\omega_{d+1}}{\kappa^{d/2}}
\bigg[ \frac 1{2^{d+1}} + \frac{\ii}{2 \pi} \int_0^{\infty} \eee^{-(\kappa + \tau_1^2 + \ldots + \tau_d^2)y^2/2} \frac1y \bigg( \prod_{j=1}^d \Phi(\ii \tau_j y) - \prod_{j=1}^d \Phi(-\ii \tau_j y) \bigg) \dd y \bigg].
$$
\end{theorem}
\begin{proof}
We first assume that $P\subseteq\mathbb B^d(\kappa)$, so that $P$
has no ideal vertices.
Since $d$ is even, Poincar\'e's relation, Theorem~\ref{theo:poincare_relation_curvature_kappa},  yields
$$
\Vol_{d,\kappa}(P)
=
\frac{\omega_{d+1}}{2\kappa^{d/2}} \sum_{k=0}^d (-1)^k \sigma_{k,\kappa}(P)
=
\frac{\omega_{d+1}}{2\kappa^{d/2}} \sum_{k=0}^d (-1)^k \sigma_{d-k, \kappa}(P).
$$
For $k\in\{1,\ldots,d\}$, parts~(i)(d) and (ii)(d) of Theorem~\ref{theo:hyperbolic_angles_rectangular_orthocentric_simplices}  entail

\begin{align*}
\sigma_{d-k, \kappa}(P)
&=
\sum_{\substack{J\subseteq \{1,\ldots, d\}\\ |J|=k}}
\beta_\kappa\left(
  \conv\left(0,\frac{e_i}{\tau_i}:i\in J^c \right),P
\right)
+
\sum_{\substack{J\subseteq \{1,\ldots, d\}\\ |J|=k-1}}
\beta_\kappa\left(
  \conv\left(\frac{e_i}{\tau_i}:i\in J^c \right),P
\right)\\
&=
\sum_{\substack{J\subseteq \{1,\ldots, d\} \\ |J|=k}} \frac1{2^k} + \sum_{\substack{J\subseteq \{1,\ldots, d\} \\ |J|=k-1}} \left[ \frac1{2^k} + \frac1\pi \int_0^\infty \frac{\eee^{-(\kappa + \tau_1^2 + \ldots + \tau_d^2)y^2/2}}{y} \Im\left( \prod_{j\in J}\Phi(-\ii\tau_j y)\right)\,\dd y \right] \\
&=
\binom dk\frac1{2^k} + \binom d{k-1}\frac1{2^k}+ \frac1\pi \sum_{\substack{J\subseteq \{1,\ldots, d\} \\ |J|=k-1}} \int_0^\infty \frac{\eee^{-Ay^2/2}}{y} \Im \left( \prod_{j\in J}\Phi(-\ii\tau_j y) \right) \,\dd y,
\end{align*}
where we set $A:= \kappa+\tau_1^2+\ldots+\tau_d^2$. Observe that $A>0$ by~\eqref{eq:rect_simplex_kappa_range}.
For $k=0$, the same formula holds if we use the convention
$\binom d{-1}=0$ and interpret sums over subsets of cardinality $-1$
as zero, since $\sigma_{d,\kappa}(P)=1$ by convention.
Hence, we obtain
\begin{align*}
&\Vol_{d,\kappa}(P)
=
\frac{\omega_{d+1}}{2\kappa^{d/2}} \sum_{k=0}^d (-1)^k \sigma_{d-k, \kappa}(P)\\
&=
\frac{\omega_{d+1}}{2\kappa^{d/2}} \sum_{k=0}^d (-1)^k \left[ \binom dk \frac1{2^k} +\binom d{k-1}\frac1{2^k}+ \frac1\pi \sum_{\substack{J\subseteq \{1,\ldots, d\} \\ |J|=k-1}} \int_0^\infty \frac{\eee^{-Ay^2/2}}{y} \Im\left( \prod_{j\in J}\Phi(-\ii\tau_j y) \right) \,\dd y \right] \\
&=
\frac{\omega_{d+1}}{2\kappa^{d/2}} \left[ \frac1{2^d} + 0 -\frac1\pi\int_0^\infty \frac{\eee^{-Ay^2/2}}{y} \Im\left( \sum_{\substack{J\subseteq \{1,\ldots, d\} \\J\neq \{1,\ldots, d\} }} (-1)^{|J|} \prod_{j\in J}\Phi(-\ii\tau_j y) \right) \,\dd y \right].
\end{align*}
Since $d$ is even and $1-\Phi(z) = \Phi(-z)$ for $z\in \CC$, we have
$$
\sum_{\substack{J\subseteq \{1,\ldots, d\} \\J\neq \{1,\ldots, d\} }}
(-1)^{|J|}
\prod_{j\in J}\Phi(-\ii\tau_j y)
=
\prod_{j=1}^d\left(1-\Phi(-\ii\tau_j y)\right)
-
\prod_{j=1}^d\Phi(-\ii\tau_j y)
=
\prod_{j=1}^d\Phi(\ii\tau_j y)
-
\prod_{j=1}^d\Phi(-\ii\tau_j y).
$$
The two products on the right-hand side are complex conjugates, hence the difference is purely imaginary. Consequently,
\begin{align*}
\Vol_{d,\kappa}(P)
&=
\frac{\omega_{d+1}}{2\kappa^{d/2}} \left[\frac1{2^d} + \frac{\ii}{\pi} \int_0^\infty \frac{\eee^{-Ay^2/2}}{y} \left(\prod_{j=1}^d\Phi(\ii\tau_j y) - \prod_{j=1}^d\Phi(-\ii\tau_j y)
\right)\,\dd y \right],
\end{align*}
which is the stated identity.

If $P$ has ideal vertices, then necessarily $\kappa<0$. For
$r\in(0,1)$, put $P_r:=rP$. Then
$P_r\subseteq\mathbb B^d(\kappa)$ is a rectangular simplex with
parameters $\tau_{j,r}=\tau_j/r$. Applying the formula already proved
to $P_r$ and making the substitution $y=rz$, we obtain the same
expression as in the statement, with $\kappa$ replaced in the
exponential factor by $\kappa r^2$. Letting $r\uparrow1$ completes the
proof by monotone convergence on the left and dominated convergence
on the right.
\end{proof}

\section{Cones over asymmetric crosspolytopes and boxes}
\subsection{Setting}
For  $d\geq 2$ let $\tau_1^+,\ldots,\tau_d^+>0$ and $\tau_1^-,\ldots,\tau_d^->0$.  Let
$e_1,\ldots,e_d$ be the standard orthonormal basis of
$\mathbb R^{d}$. The \emph{asymmetric crosspolytope} is a polytope of the form
$$
P := \conv \left\{\frac{e_1}{\tau_1^+},\ldots,  \frac{e_d}{\tau_d^+}, -\frac{e_1}{\tau_1^-},\ldots, -\frac{e_d}{\tau_d^-} \right\}\subseteq \R^d.
$$
Our aim is to compute the angles of $P$ viewed as a polytope in the geometry of constant curvature $\kappa$. The first step is to identify its tangent and normal cones, up to isometry. This will be done in terms of two classes of cones that we introduce in the next definition.

\begin{definition}[Cones over asymmetric crosspolytopes and boxes]
Let $d\geq1$ be an integer, and let
$\tau_1^+,\ldots,\tau_d^+,\tau_1^-,\ldots,\tau_d^->0$. Put $\vec \tau:= ((\tau_1^+,\tau_1^-),\ldots,(\tau_d^+,\tau_d^-))\in (\R^2)^d$. Let
$e_0,e_1,\ldots,e_d$ be the standard orthonormal basis of
$\mathbb R^{d+1}$. Define the cones
\begin{align*}
\Ccrosspoly_d(\vec \tau)
&
:=\pos\left\{
e_0+\frac{e_i}{\tau_i^+},\,
e_0-\frac{e_i}{\tau_i^-} :i=1,\ldots,d
\right\}\subseteq \R^{d+1},
\\
\Ccube_d (\vec\tau)
&
:=
\pos \left\{
e_0+\sum_{i=1}^d \eps_i \tau_i^{[\eps_i]} e_i
:
\eps_1,\ldots, \eps_d\in \{-1,+1\}
\right\}\subseteq \R^{d+1},
\end{align*}
where $\tau_i^{[+1]}=\tau_i^+$ and $\tau_i^{[-1]}=\tau_i^-$.
\end{definition}

Note that the intersection of $\Ccrosspoly_d(\vec\tau)$, respectively, $\Ccube_d (\vec\tau)$, with the affine hyperplane $\{x_0 = 1\}\subseteq \R^{d+1}$ is an asymmetric crosspolytope, respectively, a box in that hyperplane. More precisely,
\begin{align*}
&\Ccrosspoly_d(\vec \tau) \cap \{x_0= 1\} =  e_0 +   \conv \left\{\frac{e_1}{\tau_1^+},\ldots,  \frac{e_d}{\tau_d^+}, -\frac{e_1}{\tau_1^-},\ldots, -\frac{e_d}{\tau_d^-} \right\},
\\
&\Ccube_d(\vec \tau) \cap \{x_0= 1\} =  e_0 + \left[-\tau_1^-, \tau_1^+\right]\times\cdots\times\left[-\tau_d^-, \tau_d^+\right].
\end{align*}

\begin{figure}[t]
\centering
\resizebox{0.75\textwidth}{!}{%
\begin{tikzpicture}[
    x={(4.2cm,0cm)},        
    y={(0.35cm,0.75cm)},    
    z={(0.35cm,-0.75cm)},   
    line join=round,
    line cap=round,
    >=Stealth,
    visible/.style={thick},
    hidden/.style={densely dotted, thin}
]

\pgfmathsetmacro{\tpone}{0.7}  
\pgfmathsetmacro{\tmone}{0.65} 
\pgfmathsetmacro{\tptwo}{1.2}  
\pgfmathsetmacro{\tmtwo}{1.7}  

\pgfmathsetmacro{\ipone}{1/\tpone}
\pgfmathsetmacro{\imone}{-1/\tmone}
\pgfmathsetmacro{\iptwo}{1/\tptwo}
\pgfmathsetmacro{\imtwo}{-1/\tmtwo}

\pgfmathsetmacro{\nmtone}{-\tmone}
\pgfmathsetmacro{\nmtwo}{-\tmtwo}

\coordinate (O) at (0,0,0);

\coordinate (A1) at (1,\ipone,0);
\coordinate (A2) at (1,0,\iptwo);
\coordinate (A3) at (1,\imone,0);
\coordinate (A4) at (1,0,\imtwo);

\coordinate (B1) at (-1,\tpone,\tptwo);
\coordinate (B2) at (-1,\tpone,\nmtwo);
\coordinate (B3) at (-1,\nmtone,\nmtwo);
\coordinate (B4) at (-1,\nmtone,\tptwo);

\def\planeside{1.85}
\def\middleplaneside{2.35} 

\coordinate (Ppa) at (1,-\planeside,-\planeside);
\coordinate (Ppb) at (1, \planeside,-\planeside);
\coordinate (Ppc) at (1, \planeside, \planeside);
\coordinate (Ppd) at (1,-\planeside, \planeside);

\coordinate (Pma) at (-1,-\planeside,-\planeside);
\coordinate (Pmb) at (-1, \planeside,-\planeside);
\coordinate (Pmc) at (-1, \planeside, \planeside);
\coordinate (Pmd) at (-1,-\planeside, \planeside);

\coordinate (P0a) at (0,-\middleplaneside,-\middleplaneside);
\coordinate (P0b) at (0, \middleplaneside,-\middleplaneside);
\coordinate (P0c) at (0, \middleplaneside, \middleplaneside);
\coordinate (P0d) at (0,-\middleplaneside, \middleplaneside);


\fill[gray!50, fill opacity=0.10]
    (Ppa) -- (Ppb) -- (Ppc) -- (Ppd) -- cycle;
\fill[gray!50, fill opacity=0.20]
    (P0a) -- (P0b) -- (P0c) -- (P0d) -- cycle;
\fill[gray!50, fill opacity=0.10]
    (Pma) -- (Pmb) -- (Pmc) -- (Pmd) -- cycle;

\draw[gray!70, visible] (Ppa) -- (Ppb);
\draw[gray!70, visible] (Ppa) -- (Ppd);

\draw[gray!70, visible] (P0a) -- (P0b);
\draw[gray!70, visible] (P0a) -- (P0d);

\draw[gray!70, visible] (Pma) -- (Pmb);
\draw[gray!70, visible] (Pma) -- (Pmd);

\draw[gray!75, visible] (Ppb) -- (Ppc) -- (Ppd);
\draw[gray!75, visible] (P0b) -- (P0c) -- (P0d);
\draw[gray!75, visible] (Pmb) -- (Pmc) -- (Pmd);


\draw[gray!75,->] (-1.45,0,0) -- (1.45,0,0)
    node[below right] {$x_0$};

\draw[gray!75,->] (0,-1.95,0) -- (0,1.95,0)
    node[above left] {$x_1$};

\draw[gray!75,->] (0,0,-1.95) -- (0,0,1.95)
    node[below left] {$x_2$};


\draw[gray!85,thin] (1,-1.55,0) -- (1,0,0);
\draw[gray!85,->,thin] (1,0,0) -- (1,1.55,0);

\draw[gray!85,thin] (1,0,-1.55) -- (1,0,0);
\draw[gray!85,->,thin] (1,0,0) -- (1,0,1.55);

\draw[gray!85, thin] (-1,-1.55,0) -- (-1,0,0);
\draw[gray!85,->,thin] (-1,0,0) -- (-1,1.55,0);
\draw[gray!85, thin] (-1,0,-1.55) -- (-1,0,0);
\draw[gray!85,->,thin] (-1,0,0) -- (-1,0,1.55);


\filldraw[red!35, fill opacity=0.20, draw=none]
    (O) -- (B1) -- (B2) -- cycle;
\filldraw[red!35, fill opacity=0.20, draw=none]
    (O) -- (B2) -- (B3) -- cycle;
\filldraw[red!35, fill opacity=0.20, draw=none]
    (O) -- (B3) -- (B4) -- cycle;
\filldraw[red!35, fill opacity=0.20, draw=none]
    (O) -- (B4) -- (B1) -- cycle;

\draw[gray!80, hidden] (Pma) -- (Pmb);
\draw[gray!80, hidden] (Pma) -- (Pmd);
\draw[gray!85, visible] (Pmb) -- (Pmc) -- (Pmd);

\draw[gray!85, hidden] (-1,-1.55,0) -- (-1,0,0);
\draw[gray!85,->,thin] (-1,0,0) -- (-1,1.55,0);

\draw[gray!85, hidden] (-1,0,-1.55) -- (-1,0,0);
\draw[gray!85,->,thin] (-1,0,0) -- (-1,0,1.55);

\draw[red!65, hidden] (O) -- (B3);

\draw[red!80, visible] (O) -- (B1);
\draw[red!80, visible] (O) -- (B2);
\draw[red!80, visible] (O) -- (B4);

\draw[red!90, very thick]
    (B1) -- (B2) -- (B3) -- (B4) -- cycle;


\filldraw[blue!35, fill opacity=0.27, draw=none]
    (O) -- (A1) -- (A2) -- cycle;
\filldraw[blue!35, fill opacity=0.27, draw=none]
    (O) -- (A2) -- (A3) -- cycle;
\filldraw[blue!35, fill opacity=0.27, draw=none]
    (O) -- (A3) -- (A4) -- cycle;
\filldraw[blue!35, fill opacity=0.27, draw=none]
    (O) -- (A4) -- (A1) -- cycle;

\draw[blue!65, hidden] (A3) -- (A4);
\draw[blue!65, hidden] (A1) -- (A4);
\draw[blue!65, hidden] (O) -- (A4);

\draw[blue!80, visible] (A1) -- (A2);
\draw[blue!80, visible] (A2) -- (A3);

\draw[blue!80, visible] (O) -- (A1);
\draw[blue!80, visible] (O) -- (A2);
\draw[blue!80, visible] (O) -- (A3);


\node[gray!85] at (0.8,1.95,3.95) {$x_0=1$};
\node[gray!85] at (-0.1,1.95,3.95) {$x_0=0$};
\node[gray!85] at (-1.15,1.95,3.95) {$x_0=-1$};

\fill (O) circle (1.2pt);
\node[below] at (O) {$0$};

\end{tikzpicture}%
}
\caption[The cone Ccrosspoly_d(tau) and its polar cone -Ccube_d(tau*) for d=2]%
{The cone
\texorpdfstring{$\protect\Ccrosspoly_d(\vec \tau)$}{Ccrosspoly\_d(tau)}
and its polar cone
\texorpdfstring{$-\protect\Ccube_d(\vec \tau^{*})$}{-Ccube\_d(tau*)}
for $d=2$; see Proposition~\ref{prop:polar_cones_over_box_and_crosspolytope}.}
 \label{fig:cross_cone_dual_cube_cone}
\end{figure}

\subsection{Polarity}
In the next result we show that the cones $\Ccrosspoly_d(\vec \tau)$ and $-\Ccube_d(\vec \tau^*)$ are mutually polar, where $\vec \tau^*$ is obtained from $\vec \tau$ by swapping $\tau_i^+$ and $\tau_i^-$ for all $i$ simultaneously. See Figure~\ref{fig:cross_cone_dual_cube_cone} for an illustration.    We recall that for a polyhedral cone $C\subseteq \R^{d+1}$, the polar cone $C^{\circ}$ is defined as
\[
C^\circ
:=
\{y\in\mathbb R^{d+1}:\langle y,x\rangle\le 0
\text{ for all }x\in C\}.
\]
The bipolar theorem states that $C^{\circ\circ} = C$.

\begin{proposition}\label{prop:polar_cones_over_box_and_crosspolytope}
Let $d\geq1$ be an integer, and let
$\tau_1^+,\ldots,\tau_d^+,\tau_1^-,\ldots,\tau_d^->0$. Define
$$
\vec \tau:= ((\tau_1^+,\tau_1^-),\ldots,(\tau_d^+,\tau_d^-)),
\qquad
\vec \tau^* := ((\tau_1^-,\tau_1^+),\ldots,(\tau_d^-,\tau_d^+)).
$$
\begin{itemize}
\item[(a)] The polar of $\Ccube_d(\vec \tau)$ is $-\Ccrosspoly_d(\vec \tau^{*})$.
\item[(b)] The polar of $\Ccrosspoly_d(\vec\tau)$ is $-\Ccube_d(\vec \tau^{*})$.
\end{itemize}
\end{proposition}
\begin{proof}
By the bipolar theorem it suffices to prove (a). Note that the cone $\Ccube_d(\vec \tau)$ can be written as
$$
\Ccube_d(\vec \tau)
=
\left\{ x\in \R^{d+1}: x_0\geq 0, \,  -\tau_1^- x_0 \leq x_1 \leq \tau_1^+ x_0, \, \ldots, \, - \tau_d^- x_0 \leq x_d \leq \tau_d^+ x_0  \right\}.
$$
This cone has $2d$ facets naturally indexed by $i\in \{1,\ldots, d\}$ and $\eps_i \in \{\pm 1\}$, so that the supporting hyperplane of the corresponding facet is
$$
\left\{x\in \R^{d+1}: x_i = \eps_i  \tau_i^{[\eps_i]} x_0 \right\}
=
\left\{x\in \R^{d+1}: \left\lan x,-e_0+\frac{\eps_i e_i}{\tau_i^{[\eps_i]}} \right\ran = 0 \right\}.
$$

In particular, $-e_0+ \eps_i e_i / \tau_i^{[\eps_i]}$ is an outer normal vector of $\Ccube_d(\vec \tau)$ at this facet.
By conic duality, the polar cone is the positive hull of the outer
normal vectors to all facets. Hence,
\begin{align*}
(\Ccube_d(\vec \tau))^{\circ}
&=
\pos\left(-e_0+\frac{e_i}{\tau_i^{+}}, -e_0-\frac{e_i}{\tau_i^-}:i=1,\ldots, d \right)\\
&=
-\pos\left(e_0-\frac{e_i}{\tau_i^{+}}, e_0+\frac{e_i}{\tau_i^-}:i=1,\ldots, d \right)
=
-\Ccrosspoly_d(\vec \tau^{*}).
\end{align*}
This is precisely (a). The bipolar theorem gives (b).
\end{proof}

\subsection{Solid angles}
In the following two results we determine the solid angles of the cones $\Ccrosspoly_d(\vec\tau)$ and $\Ccube_d(\vec\tau)$.
\begin{theorem}[Angle of $\Ccrosspoly_d(\vec\tau)$] \label{theo:angle_of_cone_over_crosspolytope}
Let $d\geq1$ be an integer, and let
$\tau_1^+,\ldots,\tau_d^+,\tau_1^-,\ldots,\tau_d^->0$. Then
\begin{multline*}
\alpha (\Ccrosspoly_d(\vec\tau))
=
\frac12 + \frac{\ii}{2\pi} \int_0^{\infty}  \frac{\eee^{-t^2/2}}{t}	\Bigg[	\prod_{j=1}^d \left( \eee^{-(\tau_j^+ t)^2/2}\Phi(\ii \tau_j^+ t) + \eee^{-(\tau_j^- t)^2/2}\Phi(\ii \tau_j^- t) \right)\\
-\prod_{j=1}^d \left( \eee^{-(\tau_j^+ t)^2/2}\Phi(-\ii \tau_j^+ t)+ \eee^{-(\tau_j^- t)^2/2}\Phi(-\ii \tau_j^- t) \right) \Bigg] \dd t.
\end{multline*}
The integrand on the right-hand side admits continuous extension to $t=0$ and is $o(\eee^{-t^2/2})$ as $t\to +\infty$. In particular, the integral is absolutely convergent.
\end{theorem}

\begin{proof}
	First, observe that
	$$
	\Ccrosspoly_d(\vec\tau)
	=
	\Big\{ x\in \R^{d+1}: \sum_{j=1}^d \big( \tau_j^+ x_j \ind_{x_j \geq 0} - \tau_j^- x_j \ind_{x_j \leq 0} \big) \leq x_0 \Big\}.
	$$
	Let $\xi_0, \xi_1, \ldots, \xi_d \sim \mathrm{N}(0,1)$ be independent and standard normal random variables. Then
	\begin{align*}
		\alpha(\Ccrosspoly_d(\vec\tau))
		=
		\P\left[(\xi_0, \ldots, \xi_d) \in \Ccrosspoly_d(\vec\tau) \right]
		=
		\P\left[ \xi_0 \geq \sum_{j=1}^d \left( \tau_j^+ \xi_j \ind_{\xi_j \geq 0} -  \tau_j^- \xi_j \ind_{\xi_j \leq 0} \right) \right]
		=
		\P[Z \leq 0]
	\end{align*}
	with
	\begin{equation}\label{eq:proof_angle_over_crosspoly_Z_random_var}
	Z
	:=
	-\xi_0 + \sum_{j=1}^d \left( \tau_j^+ \xi_j \ind_{\xi_j \geq 0} -  \tau_j^- \xi_j \ind_{\xi_j \leq 0} \right).
	\end{equation}
    The rest of the proof will rely on a Fourier inversion formula expressing $\P[Z \leq 0]$ in terms of the characteristic function of the random variable $Z$; see Lemma~\ref{lem:appendix_probability_for_negative_Z_via_char_function} in Appendix~\ref{sec:gil_pelaez_inversion}.
	Let us show that the characteristic function of $Z$ is given by
	\begin{equation}\label{eq:proof_angle_cone_over_crosspolytope_char_function}
	\E\left[ \eee^{\ii t Z} \right]
	=
	\eee^{-t^2/2} \prod_{j=1}^d \left( \eee^{-(\tau_j^+ t)^2/2} \Phi(\ii \tau_j^+ t) + \eee^{-(\tau_j^- t)^2/2} \Phi(\ii \tau_j^- t) \right), \qquad t\in \R.
	\end{equation}
	Indeed, by independence of \(\xi_0,\xi_1,\ldots,\xi_d\), we have
	\begin{align*}
		\E\left[\eee^{\ii t Z}\right]
		&=
		\E\left[\eee^{-\ii t \xi_0}\right] \prod_{j=1}^d \E\left[ \exp\left( \ii t \left( \tau_j^+ \xi_j \ind_{\{\xi_j\geq 0\}} - \tau_j^- \xi_j \ind_{\{\xi_j\leq 0\}} \right) \right) \right].
	\end{align*}
	The first factor is
	$
	\E[\eee^{-\ii t \xi_0}]
	=
	\eee^{-t^2/2}.
	$
	For the $j$-th factor,
	$$
	\E\left[ \exp\left( \ii t \left( \tau_j^+ \xi_j \ind_{\{\xi_j\geq 0\}} - \tau_j^- \xi_j \ind_{\{\xi_j\leq 0\}}	\right)	\right) \right]
	=
	\frac{1}{\sqrt{2\pi}} \int_0^\infty \eee^{\ii t\tau_j^+ x} \eee^{-x^2/2}\,\dd x
	+
	\frac{1}{\sqrt{2\pi}} \int_{-\infty}^0 \eee^{-\ii t\tau_j^- x}	\eee^{-x^2/2}\,\dd x.
	$$
	In the second integral we substitute $x=-y$, obtaining
	$$
	\frac{1}{\sqrt{2\pi}} \int_{-\infty}^0 \eee^{-\ii t\tau_j^- x} \eee^{-x^2/2}\,\dd x
	=
	\frac{1}{\sqrt{2\pi}} \int_0^\infty \eee^{\ii t\tau_j^- y} \eee^{-y^2/2}\,\dd y.
	$$
	Using
	$$
	\frac{1}{\sqrt{2\pi}} \int_0^\infty \eee^{\ii s x} \eee^{-x^2/2}\,\dd x
	=
	\eee^{-s^2/2}\Phi(\ii s),
	\qquad s\in\R,
	$$
	we therefore get
	$$
	\E\left[ \exp\left( \ii t \left(\tau_j^+ \xi_j \ind_{\{\xi_j\geq 0\}}- \tau_j^- \xi_j \ind_{\{\xi_j\leq 0\}} \right)	\right) \right]
	=
	\eee^{-(\tau_j^+ t)^2/2}\Phi(\ii \tau_j^+ t) +\eee^{-(\tau_j^- t)^2/2}\Phi(\ii \tau_j^- t).
	$$
	Multiplying these identities over \(j=1,\ldots,d\) yields the claimed characteristic function~\eqref{eq:proof_angle_cone_over_crosspolytope_char_function}.

In order to apply Lemma~\ref{lem:appendix_probability_for_negative_Z_via_char_function}, note that $\E |Z| < \infty$, which follows from~\eqref{eq:proof_angle_over_crosspoly_Z_random_var} because $\xi_0,\xi_1,\dots, \xi_d$ are integrable random variables. Moreover, the characteristic function $t\mapsto \E[\eee^{\ii t Z}]$ belongs to $L^1(\R)$. Indeed, by continuity of the characteristic function, it suffices to show integrability at $\pm \infty$.
There, the asymptotics of $\Phi$ imply
$$
\eee^{-x^2/2}\Phi(\ii x)=O\left( 1 \right), \qquad x\to \pm \infty, \,  x\in \R,
$$
and hence, for every $j=1,\ldots, d$,
$$
\eee^{-(\tau_j^\pm t)^2/2}\Phi(\ii \tau_j^\pm t)
=
O\left( 1 \right), \qquad t\to \pm \infty, \, t\in \R.
$$
Consequently, with \eqref{eq:proof_angle_cone_over_crosspolytope_char_function}, we obtain
\begin{equation}\label{eq:proof_angle_cone_over_crosspolytope_char_funct_estimate}
\left|\E[\eee^{\ii tZ}]\right|
=
O\left(\eee^{-t^2/2}\right), \qquad  t\to \pm \infty, \, t\in \R.
\end{equation}
This is integrable at infinity, proving that the characteristic function of $Z$ belongs to $L^1(\R)$.
Hence, Lemma~\ref{lem:appendix_probability_for_negative_Z_via_char_function} is applicable, and it yields
	\begin{align*}
		\P[Z\leq 0]
		=
		\frac12- \frac{1}{2\pi \ii} \int_0^\infty \frac{\E[\eee^{\ii tZ}]-\mathbb E[\eee^{-\ii tZ}]}{t} \dd t
		=
		\frac12	+ \frac{\ii}{2\pi} \int_0^\infty \frac{1}{t} \left( \E[\eee^{\ii tZ}] -	\E[\eee^{-\ii tZ}] \right) \dd t.
	\end{align*}
Inserting the characteristic function~\eqref{eq:proof_angle_cone_over_crosspolytope_char_function} gives
\begin{multline*}
	\alpha (\Ccrosspoly_d(\vec\tau))
	=
	\frac12 + \frac{\ii}{2\pi} \int_0^{\infty}  \frac{\eee^{-t^2/2}}{t}	\Bigg[	\prod_{j=1}^d \left( \eee^{-(\tau_j^+ t)^2/2}\Phi(\ii \tau_j^+ t) + \eee^{-(\tau_j^- t)^2/2}\Phi(\ii \tau_j^- t) \right)\\
	-\prod_{j=1}^d \left( \eee^{-(\tau_j^+ t)^2/2}\Phi(-\ii \tau_j^+ t)+ \eee^{-(\tau_j^- t)^2/2}\Phi(-\ii \tau_j^- t) \right) \Bigg] \dd t,
\end{multline*}
which is precisely the claimed formula. Since $\E|Z|<\infty$, dominated convergence gives
\[
\frac{\E[\eee^{\ii tZ}]-\E[\eee^{-\ii tZ}]}{t}
=
2\ii\,\E\left[\frac{\sin(tZ)}{t}\right]
\to
2\ii\,\E Z
\qquad\text{as } t\downarrow 0.
\]
Hence, the integrand admits a continuous extension to $t=0$. By~\eqref{eq:proof_angle_cone_over_crosspolytope_char_funct_estimate},
the integrand is
$O(\eee^{-t^2/2}/t)=o(\eee^{-t^2/2})$ as $t\to+\infty$. Thus, the integral is absolutely convergent.
\end{proof}

\begin{theorem}[Angle of $\Ccube_d(\vec\tau)$]\label{theo:angle_of_cone_over_box}
Let $d\geq1$ be an integer, and let
$\tau_1^+,\ldots,\tau_d^+,\tau_1^-,\ldots,\tau_d^->0$.  Then
$$
\alpha (\Ccube_d(\vec\tau))
=
\frac{1}{\sqrt{2 \pi}} \int_0^{\infty} \prod_{j=1}^d \left(\Phi(\tau_j^+ y) - \Phi(-\tau_j^- y)\right) \eee^{-y^2/2} \dd y.
$$
\end{theorem}
\begin{proof}
We can represent the cone as
$$
\Ccube_d(\vec\tau)
=
\left\{ x\in \R^{d+1}: x_0\geq 0,\, -\tau_1^- x_0 \leq x_1 \leq \tau_1^+ x_0,\,  \ldots,\,  - \tau_d^- x_0 \leq x_d \leq \tau_d^+ x_0  \right\}.
$$
With independent standard Gaussian random variables $\xi_0, \xi_1, \ldots, \xi_d$, we have
\begin{align*}
\alpha\left(\Ccube_d(\vec\tau) \right)
&=
\P\left[ (\xi_0, \xi_1, \ldots, \xi_d) \in \Ccube_d(\vec\tau) \right]
\\
&=
\P \left[ \xi_0\geq 0,\, -\tau_1^- \xi_0 \leq \xi_1 \leq \tau_1^+ \xi_0,\, \ldots,\, -\tau_d^- \xi_0 \leq \xi_d \leq \tau_d^+ \xi_0 \right]\\
&=
\frac{1}{\sqrt{2 \pi}} \int_0^{\infty} \P\left[ -\tau_1^- y \leq \xi_1 \leq \tau_1^+ y,\, \ldots,\,  -\tau_d^- y \leq \xi_d \leq \tau_d^+ y \right] \eee^{-y^2/2} \, \dd y\\
&=
\frac{1}{\sqrt{2 \pi}} \int_0^{\infty} \prod_{j=1}^d \left( \Phi(\tau_j^+ y) - \Phi(-\tau_j^- y) \right) \, \eee^{-y^2/2} \, \dd y,
\end{align*}
where we conditioned on $\xi_0 = y$ and used the independence of $\xi_0, \xi_1,\ldots, \xi_d$.
\end{proof}

\begin{remark}[The case $d=0$]
For later use, we adopt the convention
\[
\Ccube_0=\Ccrosspoly_0:=[0,\infty)\subseteq\R.
\]
We also understand products over an empty index set to be equal to $1$.
With these conventions, the integral formulas in
Theorems~\ref{theo:angle_of_cone_over_crosspolytope}
and~\ref{theo:angle_of_cone_over_box} remain valid for $d=0$ and give the correct values $\alpha(\Ccrosspoly_0)=\alpha(\Ccube_0)=1/2$.
The polarity relations in
Proposition~\ref{prop:polar_cones_over_box_and_crosspolytope}
also remain valid for $d=0$.
\end{remark}

\section{Angles and volumes of asymmetric crosspolytopes}
Fix an integer $d\geq 2$ and a curvature $\kappa\in \R$. Let $\tau_1^+,\ldots,\tau_d^+>0$ and $\tau_1^-,\ldots,\tau_d^->0$.  Consider the asymmetric crosspolytope
\begin{equation}\label{eq:crosspolytope_def}
P := \conv \left\{\frac{e_1}{\tau_1^+},\ldots,  \frac{e_d}{\tau_d^+}, -\frac{e_1}{\tau_1^-},\ldots, -\frac{e_d}{\tau_d^-} \right\}\subseteq \R^d.
\end{equation}
We assume that \(P\) is contained in the closure of the Klein model $\bar{\mathbb B}^d(\kappa)$. This is always the case for $\kappa \geq 0$. For $\kappa< 0$, this is equivalent to requiring that every vertex of \(P\) lies in \(\bar{\mathbb B}^d(\kappa) \) and hence the  containment condition is
\[
\min\left\{\tau_1^+,\ldots, \tau_d^+, \tau_1^-,\ldots, \tau_d^-\right\}
\geq
\sqrt {-\kappa} \qquad   \qquad (\kappa <0).
\]

\subsection{Internal and external angles of asymmetric crosspolytopes}
The next result characterizes the tangent and normal cones of $P$ at its faces and provides explicit one-dimensional integral representations for its internal and external angles. For $k\in\{1,\ldots,d\}$, consider the $(k-1)$-dimensional face
\begin{equation}\label{eq:crosspolytope_face_def}
	F
	:=
	\conv\left( \frac{e_1}{\tau_1^+}, \ldots, \frac{e_k}{\tau_k^+} \right).
\end{equation}
After applying a suitable signed permutation to the coordinates and the parameters, there is no loss of generality in
considering a face of this form.

\begin{theorem}[Internal and external angles of crosspolytopes] \label{theo:hyperbolic_angles_crosspolytopes}
Fix an integer $d\geq 2$ and a curvature $\kappa\in \R$. Let $\tau_1^+,\ldots,\tau_d^+>0$ and $\tau_1^-,\ldots,\tau_d^->0$. Let  $P$ be an asymmetric crosspolytope defined by~\eqref{eq:crosspolytope_def} and assume that $P \subseteq \bar{\mathbb B}^d(\kappa)$. For $k\in\{1,\ldots,d\}$, let $F$ be the face defined
by~\eqref{eq:crosspolytope_face_def}, and let
$x\in\relint(F)\cap\mathbb B^d(\kappa)$.  Put $s_k(\kappa) = (\kappa + (\tau_1^+)^2 + \ldots + (\tau_k^+)^2)^{1/2}$. (Under the preceding assumptions, $s_k(\kappa)>0$.) Then the following assertions hold.
	\begin{itemize}
		\item[(a)] The normal cone $N_x(F,P)$ of $P$ at $F$ is isometric to the cone
		\begin{equation}\label{eq:cross_const_curv_normal_cone}
\Ccube_{d-k}\left( \left( \frac{\tau_{k+1}^+}{s_k(\kappa)}, \frac{\tau_{k+1}^-}{s_k(\kappa)} \right), \ldots, \left(\frac{\tau_{d}^+}{s_k(\kappa)}, \frac{\tau_{d}^-}{s_k(\kappa)} \right) \right).
		\end{equation}
\item[(b)] The external angle at $F$ is given by
\[
\alpha_x(N_x(F,P))
=
\frac{1}{\sqrt{2\pi}}
\int_0^\infty
\prod_{j=k+1}^d
\left(
\Phi\left(
\frac{\tau_j^+ y}{s_k(\kappa)}
\right)
-
\Phi\left(
-\frac{\tau_j^- y}{s_k(\kappa)}
\right)
\right)
\eee^{-y^2/2}\dd y.
\]
	    \item[(c)] The tangent cone $T_x(F,P)$ of $P$ at $F$ is isometric
to the direct orthogonal sum of $\R^{k-1}$ and the cone
\begin{equation}\label{eq:cross_const_curv_tangent_cone}
\Ccrosspoly_{d-k}\left(
\left(
\frac{\tau_{k+1}^-}{s_k(\kappa)},
\frac{\tau_{k+1}^+}{s_k(\kappa)}
\right),
\ldots,
\left(
\frac{\tau_d^-}{s_k(\kappa)},
\frac{\tau_d^+}{s_k(\kappa)}
\right)
\right).
\end{equation}
\item[(d)] The internal angle at $F$ is given by
\begin{multline*}
\alpha_x(T_x(F,P))
=
\frac12
+
\frac{\ii}{2\pi}
\int_0^\infty
\frac{\eee^{-t^2/2}}{t}
\Bigg[
\prod_{j=k+1}^d
\left(
\eee^{-\frac{(\tau_j^-t)^2}{2s_k(\kappa)^2}}
\Phi\left(\frac{\ii\tau_j^-t}{s_k(\kappa)}\right)
+
\eee^{-\frac{(\tau_j^+t)^2}{2s_k(\kappa)^2}}
\Phi\left(\frac{\ii\tau_j^+t}{s_k(\kappa)}\right)
\right)
\\
-
\prod_{j=k+1}^d
\left(
\eee^{-\frac{(\tau_j^-t)^2}{2s_k(\kappa)^2}}
\Phi\left(-\frac{\ii\tau_j^-t}{s_k(\kappa)}\right)
+
\eee^{-\frac{(\tau_j^+t)^2}{2s_k(\kappa)^2}}
\Phi\left(-\frac{\ii\tau_j^+t}{s_k(\kappa)}\right)
\right)
\Bigg]\dd t.
\end{multline*}
\end{itemize}
\end{theorem}
\begin{proof}
\emph{Proof of (a).} We are going to apply Proposition~\ref{prop:normal_cone_in_constant_curvature}.
The irredundant H-representation of $P$ is
$$
P
=
\bigg\{ x\in \R^d: \sum_{i=1}^d \eps_i \tau_i^{[\eps_i]} x_i \leq 1 \text{ for all } \eps = (\eps_1,\ldots, \eps_d) \in \{\pm 1\}^d \bigg\}.
$$
The facets of $P$ are naturally indexed by $\eps = (\eps_1, \ldots, \eps_d) \in \{\pm 1\}^d$ and have the form
\begin{equation}\label{eq:pf_angles_of_crosspolytopes_faces_of_P}
P \cap \bigg\{ x\in \R^d: \sum_{i=1}^d \eps_i \tau_i^{[\eps_i]} x_i = 1 \bigg\}.
\end{equation}
The Euclidean outward normal vector to this facet can be chosen to be $a_\eps = \sum_{i=1}^d \eps_i \tau_i^{[\eps_i]} e_i$, with corresponding $b_\eps = 1$.
The facets of $P$ containing $F$ are precisely those sets~\eqref{eq:pf_angles_of_crosspolytopes_faces_of_P} with
$$
\eps = (\underbrace{1,\ldots, 1}_{k \text{ times }},\eps_{k+1}, \ldots, \eps_d) \qquad \text{ for any } \eps_{k+1}, \ldots, \eps_d \in \{\pm 1\}.
$$
Now let
$x\in\relint(F)\cap\mathbb B^d(\kappa)$. By Proposition~\ref{prop:normal_cone_in_constant_curvature},
parts~(b) and~(c), the normal cone $N_x(F,P)$ is given by
$$
N_x(F,P) = \pos\big(\tilde n_\eps(x): \eps = (1,\ldots, 1,\eps_{k+1}, \ldots, \eps_d) \in \{1\}^k \times \{\pm 1\}^{d-k} \big)
$$
with vectors $\tilde n_\eps(x)$ satisfying
$$
g_x(\tilde n_\eps(x), \tilde n_\delta(x))
=
\langle a_\eps,a_\delta\rangle+\kappa b_\eps b_\delta = \langle a_\eps,a_\delta\rangle+\kappa,
$$
for  $\eps, \delta \in \{1\}^k \times \{\pm 1\}^{d-k}$. Now
\begin{align*}
\lan a_{\eps}, a_\delta \ran
&=
\left\lan \sum_{\ell=1}^k \tau_\ell^+ e_\ell + \sum_{\ell = k+1}^d \eps_\ell \tau_\ell^{[\eps_\ell]} e_\ell, \sum_{\ell=1}^k \tau_\ell^+ e_\ell + \sum_{\ell = k+1}^d \delta_\ell \tau_\ell^{[\delta_\ell]} e_\ell \right\ran\\
&=
(\tau^+_1)^2 + \ldots + (\tau_k^+)^2 + \sum_{\ell = k+1}^d \eps_\ell \delta_\ell \tau_\ell^{[\eps_\ell]} \tau_\ell^{[\delta_\ell]}.
\end{align*}
The Riemannian scalar products of the rescaled vectors
$$
n_\eps'(x) := \frac{\tilde n_\eps(x)}{\sqrt{\kappa + ({\tau_1^+})^2 + \ldots + ({\tau_k^+})^2}} = \frac{\tilde n_\eps(x)}{s_k(\kappa)}
$$
are then
\begin{equation}\label{eq:proof_normal_cones_crosspoly_assymmetric_products}
g_x(n_\eps'(x), n_\delta'(x))
=
\frac{\lan a_\eps, a_\delta \ran + \kappa}{\kappa + (\tau_1^+)^2 + \ldots + (\tau_k^+)^2}
=
1 + \sum_{\ell = k+1}^d \eps_\ell \delta_\ell \frac{\tau_\ell^{[\eps_\ell]}}{s_k(\kappa)} \frac{\tau_\ell^{[\delta_\ell]}}{s_k(\kappa)}.
\end{equation}
On the other hand, consider the cone~\eqref{eq:cross_const_curv_normal_cone}, that is
\begin{align*}
\Ccube_{d-k}(\vec\mu)
=
\pos \left\{
e_0+\sum_{i=k+1}^d \eps_i \mu_i^{[\eps_i]} e_i
:
\eps_{k+1},\ldots, \eps_d\in \{-1,+1\}
\right\}
\subseteq \lin (e_0, e_{k+1},\ldots, e_d)\equiv \R^{d-k+1},
\end{align*}
with the parameters
$
\vec\mu = ((\mu_{k+1}^+,\mu_{k+1}^-), \ldots, (\mu_d^+, \mu_d^-))
$
given in \eqref{eq:cross_const_curv_normal_cone}.
A computation of the Euclidean scalar products of these generators shows that they agree with the Riemannian scalar products in~\eqref{eq:proof_normal_cones_crosspoly_assymmetric_products}.
Thus, $N_x(F,P)$ is indeed isometric  to the cone~\eqref{eq:cross_const_curv_normal_cone}.

\vspace*{0.2cm}
\noindent
\emph{Proof of (c).}
Let
$
I:(T_x\mathbb B^d(\kappa),g_x)
\to
(\R^d,\langle\cdot,\cdot\rangle)
$
be a linear isometry. Then $IT_x(F,P)$ is the Euclidean polar cone of
$IN_x(F,P)$. By part~(a), the cone $IN_x(F,P)$ is isometric to the
cone in~\eqref{eq:cross_const_curv_normal_cone}. Moreover,
$
\dim\lin N_x(F,P)=d-k+1
$
by Proposition~\ref{prop:normal_cone_in_constant_curvature},
part~(a). Hence,
Proposition~\ref{prop:polar_cones_over_box_and_crosspolytope} shows
that $T_x(F,P)$ is isometric to the direct orthogonal sum of
$\R^{k-1}$ and the cone
in~\eqref{eq:cross_const_curv_tangent_cone}. Here we used the fact
that a cone and its negative are isometric.

\vspace*{0.2cm}
\noindent
\emph{Proof of (b) and (d).}
Part~(b) follows from~(a) and
Theorem~\ref{theo:angle_of_cone_over_box}. Since taking a direct
orthogonal sum with a Euclidean space does not change the solid angle,
part~(d) follows from~(c) and
Theorem~\ref{theo:angle_of_cone_over_crosspolytope}.
\end{proof}

\subsection{Volumes of asymmetric crosspolytopes in even dimension}
By decomposing asymmetric crosspolytopes into rectangular simplices
and applying
Theorem~\ref{theo:volume_rectangular_simplex}, we obtain a one-dimensional integral representation for their Riemannian volumes in even dimensions.

\begin{theorem}[Volume in even dimension]
\label{theo:volume_asymmetric_crosspolytope_even_d}
Let $d\geq2$ be even, fix a curvature $\kappa\in\R\setminus\{0\}$, and let
$\tau_1^+,\ldots,\tau_d^+,\tau_1^-,\ldots,\tau_d^->0$.
Let $P$ be the asymmetric crosspolytope defined
by~\eqref{eq:crosspolytope_def}, and assume that
$P\subseteq\bar{\mathbb B}^d(\kappa)$. Then the Riemannian volume of
$P$ is given by
\begin{multline*}
\Vol_{d,\kappa}(P)
=
\frac{\omega_{d+1}}{\kappa^{d/2}} \bigg[ \frac12+ \frac{\ii}{2\pi} \int_0^\infty \frac{\eee^{-\kappa y^2/2}}{y} \bigg( \prod_{j=1}^d \Big( \eee^{-(\tau_j^+)^2y^2/2}\Phi(\ii \tau_j^+y) + \eee^{-(\tau_j^-)^2y^2/2}\Phi(\ii\tau_j^-y) \Big)\\
-
\prod_{j=1}^d\Big(\eee^{-(\tau_j^+)^2y^2/2}\Phi(-\ii \tau_j^+y) +\eee^{-(\tau_j^-)^2 y^2/2}\Phi(-\ii\tau_j^-y) \Big)\bigg) \, \dd y \bigg].
\end{multline*}
\end{theorem}

\begin{proof}
The crosspolytope $P$ is the union of the rectangular simplices
\[
P_\eps
:=
\conv\left\{
0,
\frac{\eps_1e_1}{\tau_1^{[\eps_1]}},
\ldots,
\frac{\eps_de_d}{\tau_d^{[\eps_d]}}
\right\},
\qquad
\eps\in\{\pm1\}^d,
\]
and any two distinct such simplices intersect in a set of Riemannian
$d$-volume zero. Since reflections in coordinate hyperplanes are isometries of the Klein model, $P_\eps$ is isometric to
$$
\conv\left\{ 0,\frac{e_1}{\tau_1^{[\eps_1]}}, \ldots, \frac{e_d}{\tau_d^{[\eps_d]}}\right\}.
$$
Thus, by the formula for the volume of rectangular simplices, Theorem~\ref{theo:volume_rectangular_simplex},
\begin{multline*}
\Vol_{d,\kappa}(P)
=
\sum_{\eps\in\{\pm1\}^d}\Vol_{d,\kappa}(P_\eps)
=
\frac{\omega_{d+1}}{\kappa^{d/2}} \bigg[ \sum_{\eps\in\{\pm1\}^d}\frac1{2^{d+1}}
\\
+
\frac{\ii}{2\pi} \int_0^\infty\frac{\eee^{-\kappa y^2/2}}{y} \sum_{\eps\in\{\pm1\}^d} \eee^{-\sum_{j=1}^d(\tau_j^{[\eps_j]})^2 y^2/2} \bigg( \prod_{j=1}^d\Phi(\ii\tau_j^{[\eps_j]}y) -\prod_{j=1}^d\Phi(-\ii\tau_j^{[\eps_j]}y) \bigg) \,\dd y \bigg].
\end{multline*}
Observe that
$
\sum_{\eps\in\{\pm1\}^d}\frac1{2^{d+1}}
=
\frac12.
$
Using
\begin{align*}
\sum_{\eps\in\{\pm1\}^d}
\eee^{-\sum_{j=1}^d(\tau_j^{[\eps_j]})^2y^2/2}
\prod_{j=1}^d
\Phi(\ii\tau_j^{[\eps_j]}y)
=
\prod_{j=1}^d
\left(
\eee^{-(\tau_j^+)^2y^2/2}\Phi(\ii\tau_j^+y)
+
\eee^{-(\tau_j^-)^2y^2/2}\Phi(\ii\tau_j^-y)
\right),
\end{align*}
and the analogous identity with $\ii$ replaced by $-\ii$, the preceding
formula simplifies to
\begin{align*}
\Vol_{d,\kappa}(P)
&=
\frac{\omega_{d+1}}{\kappa^{d/2}} \bigg[ \frac12 +\frac{\ii}{2\pi} \int_0^\infty \frac{\eee^{-\kappa y^2/2}}{y} \bigg( \prod_{j=1}^d \left( \eee^{-(\tau_j^+)^2y^2/2}\Phi(\ii\tau_j^+y) +\eee^{-(\tau_j^-)^2y^2/2}\Phi(\ii\tau_j^-y) \right) \\
&\hspace{4.8cm} - \prod_{j=1}^d \left( \eee^{-(\tau_j^+)^2y^2/2}\Phi(-\ii\tau_j^+y) + \eee^{-(\tau_j^-)^2y^2/2}\Phi(-\ii\tau_j^-y) \right) \bigg) \,\dd y \bigg],
\end{align*}
which is precisely the stated formula.  The two products in the integrand are complex conjugates. Hence, their
difference is purely imaginary, so the right-hand side is real.
\end{proof}

\subsection{Example: Angles and volumes of regular crosspolytopes}
In this subsection, we specialize the preceding results to regular
crosspolytopes and express them in terms of the geodesic edge length of
the crosspolytope.

\begin{corollary}[Angles and volumes of regular crosspolytopes]
\label{cor:angles_volumes_regular_crosspolytopes_constant_curvature}
Let $d\geq2$ and $\kappa\in\R$. Let
$\Diamond_\ell^d(\kappa)\subseteq\mathbb B^d(\kappa)$ be a
$d$-dimensional regular crosspolytope of geodesic edge length $\ell$.
Such a crosspolytope exists for every $\ell\in(0,\infty)$ if
$\kappa\leq0$, whereas for $\kappa>0$ it exists precisely when $0<\ell<\frac{\pi}{2\sqrt{\kappa}}$.
Put
\[
C:=C_\kappa(\ell)
:=
\begin{cases}
\cosh(\sqrt{-\kappa}\,\ell),&\kappa<0,\\
1,&\kappa=0,\\
\cos(\sqrt{\kappa}\,\ell),&\kappa>0.
\end{cases}
\]
Then, for every $k\in\{0,\ldots,d-1\}$, every $k$-dimensional face
$F$ of $\Diamond_\ell^d(\kappa)$, and every $x\in\relint(F)$, the
external and internal angles at $F$ are given by
\begin{align}
\alpha_x(N_x(F,\Diamond_\ell^d(\kappa)))
&=
\frac{1}{\sqrt{2\pi}} \int_0^\infty \left( 2\Phi \left( \sqrt{\frac{C}{1+ kC}} \, y \right) -1\right)^{d-k-1} \eee^{-y^2/2} \, \dd y,
\label{eq:regular_crosspolytope_external_angle}
\\
\alpha_x(T_x(F,\Diamond_\ell^d(\kappa)))
&=
\frac12
-
\frac{2^{d-k-1}}{\pi} \int_0^\infty
\Im \Bigg[\Phi\Bigg( \ii \sqrt{\frac{C}{1+(d-1)C}}\, y \Bigg)^{d-k-1}\Bigg] \eee^{-y^2 / 2} \, \frac{\dd y}{y}.
\label{eq:regular_crosspolytope_internal_angle}
\end{align}
If $d$ is even and $\kappa\neq0$, then the Riemannian volume of
$\Diamond_\ell^d(\kappa)$ is given by
\begin{align}
\Vol_{d,\kappa}\bigl(\Diamond_\ell^d(\kappa)\bigr)
=
\frac{\omega_{d+1}}{\kappa^{d/2}}
\Bigg[
\frac12
-
\frac{2^d}{\pi}
\int_0^\infty
\Im \Bigg[
\Phi\left(
\ii\sqrt{\frac{C}{1+(d-1)C}}\,y
\right)^d
\Bigg]\eee^{-y^2/2} \, \frac{\dd y}{y}
\Bigg]. \label{eq:volume_regular_crosspolytope_curvature_kappa_even_d}
\end{align}
For $\kappa<0$, let
\[
\Diamond_\infty^d(\kappa)
:=
\conv\left\{
\frac{e_1}{\sqrt{-\kappa}},\ldots,
\frac{e_d}{\sqrt{-\kappa}},
-\frac{e_1}{\sqrt{-\kappa}},\ldots,
-\frac{e_d}{\sqrt{-\kappa}}
\right\}\subseteq \bar{\mathbb B}^d(\kappa)
\]
denote the regular ideal crosspolytope. Its angle and volume formulas
correspond to the limiting values $\ell\to\infty$ and
$C\to\infty$. More precisely, for every
$k\in\{1,\ldots,d-1\}$, every $k$-dimensional face $F$ of
$\Diamond_\infty^d(\kappa)$, and every
$x\in\relint(F)\cap\mathbb B^d(\kappa)$, one has
\begin{align}
\alpha_x\bigl(N_x(F,\Diamond_\infty^d(\kappa))\bigr)
&=
\frac{1}{\sqrt{2\pi}}
\int_0^\infty
\left(
2\Phi\left(\frac{y}{\sqrt{k}}\right)
-
1
\right)^{d-k-1}
\eee^{-y^2/2}\dd y,
\label{eq:regular_crosspolytope_ideal_external_angle}
\\
\alpha_x\bigl(T_x(F,\Diamond_\infty^d(\kappa))\bigr)
&=
\frac12
-
\frac{2^{d-k-1}}{\pi}
\int_0^\infty
\Im \Bigg[
\Phi\left(
\frac{\ii y}{\sqrt{d-1}}
\right)^{d-k-1}
\Bigg]
\eee^{-y^2/2} \, \frac{\dd y}{y}.  \label{eq:regular_crosspolytope_ideal_internal_angle}
\end{align}
If $\kappa<0$ and $d$ is even, then the Riemannian volume of the
regular ideal crosspolytope is given by
\begin{align}
\Vol_{d,\kappa}\bigl(\Diamond_\infty^d(\kappa)\bigr)
&=
\frac{\omega_{d+1}}{\kappa^{d/2}}
\Bigg[
\frac12
-
\frac{2^d}{\pi}
\int_0^\infty
\Im
\left[
\Phi\left(
\frac{\ii y}{\sqrt{d-1}}
\right)^d
\right]
\eee^{-y^2/2} \,  \frac{\dd y}{y}
\Bigg].
\label{eq:volume_regular_crosspolytope_ideal_curvature_kappa_even_d}
\end{align}
\end{corollary}
\begin{proof}
Take $\tau_1^+=\cdots=\tau_d^+
=
\tau_1^-=\cdots=\tau_d^-=\tau>0$ in the setting of asymmetric crosspolytopes, and put
\[
P
:=
\conv\left\{
\frac{e_1}{\tau},\ldots,\frac{e_d}{\tau},
-\frac{e_1}{\tau},\ldots,-\frac{e_d}{\tau}
\right\}.
\]
For $\kappa\geq0$, there is no restriction on $\tau>0$. If
$\kappa<0$, then
$P\subseteq\mathbb B^d(\kappa)$
if and only if  $\tau>\sqrt{-\kappa}$,
while $\tau=\sqrt{-\kappa}$ corresponds to the regular ideal
crosspolytope.

Let $\ell$ be the geodesic edge length of $P$. Consider two adjacent vertices
$\eps_i e_i/\tau$ and $\eps_j e_j/\tau$, where $i\neq j$. Their
Euclidean scalar product is $0$, and both have squared norm
$1/\tau^2$. Thus, the distance formula
\eqref{eq:geodesic_distance_in_B_d_kappa} gives, for $\kappa\neq0$,
\[
C
=
C_\kappa(\ell)
=
\frac{1}{1+\kappa/\tau^2}
=
\frac{\tau^2}{\tau^2+\kappa}.
\]
The identity remains valid for $\kappa=0$, since both sides are equal
to $1$. If $\kappa>0$, then $C$ ranges over $(0,1)$ as $\tau$ ranges over
$(0,\infty)$. Since
$C=\cos(\sqrt{\kappa}\,\ell)$, this gives $0<\ell<\pi/(2\sqrt{\kappa})$.
If $\kappa<0$, then $C$ ranges over $(1,\infty)$ as $\tau$ ranges over
$(\sqrt{-\kappa},\infty)$. Since
$C=\cosh(\sqrt{-\kappa}\,\ell)$, this gives
$\ell\in(0,\infty)$. Finally, for $\kappa=0$, one has
$\ell=\sqrt2/\tau$, and hence again $\ell\in(0,\infty)$.

Let $F$ be a $k$-dimensional face of $P$. By applying a signed
permutation of the coordinates, we may assume that
\[
F
=
\conv\left(
\frac{e_1}{\tau},\ldots,\frac{e_{k+1}}{\tau}
\right).
\]
We apply Theorem~\ref{theo:hyperbolic_angles_crosspolytopes} with $k+1$ in
place of $k$. Note that  $s_{k+1}(\kappa)^2 =\kappa+(k+1)\tau^2$.
The relation between $C$ and $\tau$ implies
\begin{equation}\label{eq:regular_crosspolytope_tau_s_ratio}
a
:=
\frac{\tau^2}{s_{k+1}(\kappa)^2}
=
\frac{C}{1+kC}.
\end{equation}
Applying part~(b) of
Theorem~\ref{theo:hyperbolic_angles_crosspolytopes} gives
\begin{align*}
\alpha_x\bigl(N_x(F,P)\bigr)
&=
\frac{1}{\sqrt{2\pi}}
\int_0^\infty
\left(
\Phi(\sqrt a\,y)-\Phi(-\sqrt a\,y)
\right)^{d-k-1}
\eee^{-y^2/2}\dd y
\\
&=
\frac{1}{\sqrt{2\pi}}
\int_0^\infty
\left(2\Phi(\sqrt a\,y)-1\right)^{d-k-1}
\eee^{-y^2/2}\dd y.
\end{align*}
Together with~\eqref{eq:regular_crosspolytope_tau_s_ratio}, this is
precisely~\eqref{eq:regular_crosspolytope_external_angle}.
Similarly, part~(d) of
Theorem~\ref{theo:hyperbolic_angles_crosspolytopes} yields
\begin{align*}
\alpha_x\bigl(T_x(F,P)\bigr)
&=
\frac12
+
\frac{2^{d-k-1}\ii}{2\pi}
\int_0^\infty
\frac{\eee^{-(1+(d-k-1)a)t^2/2}}{t}
\left(
\Phi(\ii\sqrt a\,t)^{d-k-1}
-
\Phi(-\ii\sqrt a\,t)^{d-k-1}
\right) \dd t
\\
&=
\frac12
-
\frac{2^{d-k-1}}{\pi}
\int_0^\infty
\Im \left[
\Phi(\ii\sqrt a\,t)^{d-k-1}
\right] \eee^{-(1+(d-k-1)a)t^2/2} \frac{\dd t}{t}.
\end{align*}
The substitution $y=\sqrt{1+(d-k-1)a}\,t$ gives~\eqref{eq:regular_crosspolytope_internal_angle} since
\[
1+(d-k-1)a
=
\frac{1+(d-1)C}{1+kC},
\qquad
\frac{a}{1+(d-k-1)a}
=
\frac{C}{1+(d-1)C}.
\]

Suppose now that $d$ is even and $\kappa\neq0$. Theorem~\ref{theo:volume_asymmetric_crosspolytope_even_d},  applied with all parameters
equal to $\tau$, gives
\begin{align*}
\Vol_{d,\kappa}(P)
&=
\frac{\omega_{d+1}}{\kappa^{d/2}}
\Bigg[
\frac12
+
\frac{2^d\ii}{2\pi}
\int_0^\infty
\frac{\eee^{-(\kappa+d\tau^2)t^2/2}}{t}
\left(
\Phi(\ii\tau t)^d
-
\Phi(-\ii\tau t)^d
\right)\dd t
\Bigg]
\\
&=
\frac{\omega_{d+1}}{\kappa^{d/2}}
\Bigg[
\frac12
-
\frac{2^d}{\pi}
\int_0^\infty
\Im \left[
\Phi(\ii\tau t)^d
\right]
\eee^{-(\kappa+d\tau^2)t^2/2} \,  \frac{\dd t}{t}
\Bigg].
\end{align*}
The substitution $y=\sqrt{\kappa+d\tau^2}\,t$
yields~\eqref{eq:volume_regular_crosspolytope_curvature_kappa_even_d} since
$
\frac{\tau^2}{\kappa+d\tau^2}
=
\frac{C}{1+(d-1)C}.
$

Finally, suppose that $\kappa<0$ and
$\tau=\sqrt{-\kappa}$. Then $P=\Diamond_\infty^d(\kappa)$ is the regular ideal crosspolytope. For $k\geq1$, we have
\[
a= \frac{\tau^2}{s_{k+1}(\kappa)^2}
=
\frac{-\kappa}{\kappa+(k+1)(-\kappa)}
=
\frac1k,
\qquad
\frac{1/k}{1+(d-k-1)/k}
=
\frac1{d-1},
\qquad
\frac{\tau^2}{\kappa+d\tau^2}
=
\frac1{d-1}.
\]
Consequently, the preceding formulas for $\alpha_x(N_x(F,P))$  and $\alpha_x(T_x(F,P))$ give
\eqref{eq:regular_crosspolytope_ideal_external_angle} and
\eqref{eq:regular_crosspolytope_ideal_internal_angle}. If $d$ is even,
the preceding formula for $\Vol_{d,\kappa}(P)$ gives
\eqref{eq:volume_regular_crosspolytope_ideal_curvature_kappa_even_d}.
\end{proof}

\begin{example}[The Euclidean case $\kappa=0$]
Let $\Diamond_\ell^d(0)$ be a Euclidean regular crosspolytope of edge
length $\ell>0$. Its angle formulas are obtained from
\eqref{eq:regular_crosspolytope_external_angle}
and~\eqref{eq:regular_crosspolytope_internal_angle} by setting $C=1$.
Moreover,
$\Vol_{d,0}(\Diamond_\ell^d(0))=2^{d/2}\ell^d/d!$.
\end{example}

\begin{remark}[Relation to earlier work]
The volume of the regular ideal hyperbolic crosspolytope was studied
earlier by Smith~\cite{Smith1988Thesis,Smith2000Simplexity}.  He proved
the asymptotic formulas
\[
\Vol_{d,-1}\bigl(\Diamond_\infty^d(-1)\bigr)
\sim
\frac{\eee\,2^d}{d!},
\qquad
\frac{\Vol_{d,-1}\bigl(\Diamond_\infty^d(-1)\bigr)}
     {\Vol_{d,0}\bigl(\Diamond_\infty^d(-1)\bigr)}
\sim  \eee,
\qquad d\to\infty.
\]
The volume asymptotic is also quoted by Marshall~\cite{Marshall1998Cubes}, who attributes it to Smith. Smith's thesis~\cite{Smith1988Thesis} also develops numerical
quadrature methods for radial integrals over regular polytopes, which
can be used to compute their hyperbolic volumes numerically.
\end{remark}

\appendix
\section{A Fourier inversion formula}\label{sec:gil_pelaez_inversion}
In this appendix we record a version of the Gil--Pelaez inversion formula.
\begin{lemma}\label{lem:appendix_probability_for_negative_Z_via_char_function}
Let $Z$ be a real-valued random variable with characteristic function $\varphi(t) := \mathbb{E}\eee^{\ii tZ}$, $t\in\mathbb{R}$.
Assume that $\mathbb{E}|Z|<\infty$ and $\varphi\in L^1(\mathbb{R})$.
Then
\[
    \mathbb{P}[Z\leq 0]
    =
    \frac12
-
\frac{1}{2\pi \ii}
\int_0^\infty
\frac{\varphi(t)-\varphi(-t)}{t}\,\dint t
    =
    \frac12
    -
    \frac1\pi
    \int_0^\infty \frac{\operatorname{Im}\varphi(t)}{t}\,\dint t.
\]
Moreover, the integrals on the right-hand side are absolutely convergent.
\end{lemma}

In fact, the formula is true for every real-valued random variable $Z$ under the sole condition that $\P[Z=0] = 0$, see~\cite{gilpelaez51} and~\cite[Exercise~10 on p.~293]{ChowTeicher1997}, but then the integral has to be understood as an improper integral
both at $0$ and at $\infty$; see~\cite{Wendel_on_GilPelaez1961}. For our purposes, it is convenient to impose the additional assumptions $\mathbb{E}|Z|<\infty$ and $\varphi\in L^1(\mathbb{R})$, which make the integral absolutely convergent.

\begin{proof}[Proof of Lemma~\ref{lem:appendix_probability_for_negative_Z_via_char_function}]
For $t \in \R$ define
\begin{equation}\label{eq:proof_gil_pelaez_psi_t_def}
    \psi(t)
    :=
    \E\sin(tZ)
    =
    \E \left(\frac{\eee^{\ii tZ}-\eee^{-\ii tZ}}{2 \ii}\right)
    =
    \frac{\varphi(t) - \varphi(-t)}{2 \ii} = \operatorname{Im}\varphi(t).
\end{equation}
It suffices to prove
$$
    \P[Z\leq 0]
    =
    \frac12 - \frac1\pi \int_0^\infty \frac{\psi(t)}{t} \,\dd t.
$$

We first check that  $\int_{0}^\infty |\psi(t)/t| \, \dd t < \infty$. For $0<t\leq 1$, observe that
$
    |\psi(t)|
    \leq
    \E|\sin(tZ)|
    \leq
    t\, \E|Z|.
$
Therefore
$$
    \int_0^1 \left|\frac{\psi(t)}{t}\right|\,\dd t
    \leq
    \E|Z|
    <
    \infty.
$$
For $t\geq 1$,  we have
$$
    \left|\frac{\psi(t)}{t}\right|
    \leq
    \left| \psi(t) \right|
    \leq
    \frac{|\varphi(t)|+|\varphi(-t)|}{2}.
$$
Since $\varphi \in L^1(\mathbb{R})$, this implies
$$
    \int_1^\infty \left|\frac{\psi(t)}{t}\right|\,\dd t
    <
    \infty.
$$
Thus $|\psi(t)/t|$ is integrable over $(0,\infty)$.

Next, since $\varphi\in L^1(\R)$, the Fourier inversion theorem implies
that the distribution of $Z$ has a bounded continuous density. In
particular, $\P[Z=0]=0$ and hence
$$
    \mathbb P[Z\leq 0]
    =
    \frac12\left(1-\E\sgn (Z)\right).
$$

It remains to express $\E \sgn (Z)$ in terms of $\psi$. A direct integration of the Fourier inversion formula over $(-\infty,0]$ would require an unjustified interchange of non-absolutely convergent integrals. Instead, the sign function is approximated by the $\arctan$ function as follows.
For \(\varepsilon>0\) and \(x\in\mathbb{R}\), there is the elementary identity
\[
    \arctan\left(\frac{x}{\varepsilon}\right) = \int_0^\infty \eee^{-\varepsilon t}\frac{\sin(tx)}{t}\,\dint t.
\]
Applying this identity with $x=Z$, taking expectation and using Fubini's theorem, which is justified by
$$
    \E \int_0^\infty \eee^{- \varepsilon t} \left| \frac{\sin(tZ)}{t} \right| \,\dd t
    \leq
    \E \int_0^\infty \eee^{- \varepsilon t}|Z|\,\dd t
    =
    \frac{\E|Z|}{\varepsilon}
    <
    \infty,
$$
we obtain
\begin{equation}\label{eq:proof_gil_pelaez_eps}
\E \arctan \left( \frac{Z}{\varepsilon} \right)
=
\E\left[ \int_0^{\infty} \eee^{-\varepsilon t} \frac{\sin(tZ)}{t} \,\dd t \right]
=
\int_0^\infty \eee^{-\varepsilon t} \frac{\psi(t)}{t} \, \dd t.
\end{equation}
We now let $\varepsilon \downarrow 0$.
Since $\psi(t)/t$ is absolutely integrable, dominated convergence yields
$$
\lim_{\varepsilon \downarrow 0} \int_0^\infty \eee^{-\varepsilon t}\frac{\psi(t)}{t}\,\dd t
=    \int_0^\infty \frac{\psi(t)}{t}\,\dd t.
$$
Moreover, since $\P[Z=0]=0$,
$$
    \arctan\left(\frac{Z}{\varepsilon}\right)
    \overset{a.s.}{\underset{\eps \downarrow 0} \longrightarrow}
    \frac{\pi}{2}\sgn(Z).
$$
Hence, the dominated convergence theorem gives
$$
\lim_{\eps \downarrow 0}\E\arctan\left(\frac{Z}{\varepsilon}\right)
=
\frac{\pi}{2}\E\sgn(Z).
$$
Thus, letting $\eps \downarrow 0$ in~\eqref{eq:proof_gil_pelaez_eps} gives
$$
    \frac{\pi}{2}\E\sgn(Z) = \int_0^\infty \frac{\psi(t)}{t}\,\dd t.
$$
Therefore we obtain
\[
    \P[Z\leq 0]
    =
    \frac12(1-\E\sgn(Z))
    =
    \frac12-\frac1\pi
    \int_0^\infty \frac{\psi(t)}{t}\,\dd t,
\]
which proves the claimed formula in view of~\eqref{eq:proof_gil_pelaez_psi_t_def}.
\end{proof}

\section*{Acknowledgement}
Supported by the DFG under Germany's Excellence Strategy EXC 2044 - 390685587, Mathematics M\"unster: \emph{Dynamics-Geometry-Structure}, by the DFG priority program SPP 2265 \emph{Random Geometric Systems}, and by the DFG Research Training Group \emph{Rigorous Analysis of
Complex Random Systems} (RTG 3027, Project Number 524444762).

\section*{Declarations}

\subsection*{AI declaration}
The authors used ChatGPT-6 (OpenAI) for mathematical discussions, assistance with proof checking, literature exploration, and suggestions for improving the language and presentation of the manuscript. All AI-generated suggestions incorporated into the manuscript were critically examined by the authors, who independently verified the mathematical content and take full responsibility for the paper.

\subsection*{Conflict of interest statement}
The authors declare that they have no conflicts of interest.

\subsection*{Data availability statement}
We do not analyse or generate any datasets.

\addcontentsline{toc}{section}{References}
\bibliography{angles_volumes_crosspolytopes_boxes_bib}
\bibliographystyle{plainnat}

\vspace{1cm}

\footnotesize

\textsc{Zakhar Kabluchko: Institut f\"ur Mathematische Stochastik,
	Universit\"at M\"unster,
	Orl\'eans-Ring 10,
	48149 M\"unster, Germany}\\
\textit{E-mail}: \texttt{zakhar.kabluchko@uni-muenster.de}\\

\textsc{Philipp Schange: Institut f\"ur Mathematische Stochastik, Universit\"at M\"unster,
	Orl\'eans-Ring 10,
	48149 M\"unster, Germany}\\
\textit{E-mail}: \texttt{philipp.schange@uni-muenster.de}

\end{document}